\RequirePackage{fix-cm}
\documentclass[11pt]{article}
\usepackage[margin=1in]{geometry} 
\usepackage{setspace}
\usepackage[numbers]{natbib}

\usepackage{color}
\usepackage{xcolor}
\usepackage{amsmath, amssymb} 
\usepackage{amsthm}
\usepackage{graphicx}
\usepackage{amsfonts}
\usepackage{booktabs}
\usepackage{subcaption}
\usepackage{url}
\usepackage{algorithm}
\usepackage{algpseudocode}
\usepackage{hyperref}
\usepackage[labelfont=bf,labelsep=period]{caption}
\usepackage{subcaption}

\def\tsc#1{\csdef{#1}{\textsc{\lowercase{#1}}\xspace}}
\tsc{WGM}
\tsc{QE}

\newcommand{\Div}{\text{div}}
\newcommand{\DivS}{\text{div}^S}

\newcommand{\HessS}{\mathcal{H}^S}
\newcommand{\grad}{\nabla}
\newcommand{\gradS}{\nabla^S}
\newcommand{\pLap}{\Delta_p}
\newcommand{\pLapS}{\pLap^S}
\newcommand{\infLap}{\Delta_{\infty}}
\newcommand{\infLapS}{\infLap^S}

\newcommand{\lamp}{\lambda_{p}}
\newcommand{\Lamp}{\Lambda_{p}}
\newcommand{\lampApprox}{\lambda_{p, h}}
\newcommand{\LamInf}{\Lambda_{\infty}}
\newcommand{\up}{u_{p}}
\newcommand{\upApprox}{u_{p, h}}
\newcommand{\uInf}{u_{\infty}}

\newcommand{\bx}{{\bf x}}
\newcommand{\bgamma}{\boldsymbol{\gamma}}
\newcommand{\bv}{{\bf v}}
\newcommand{\by}{{\bf y}}

\newcommand{\distx}{\mathrm{dist}(\bx, \partial \Omega)}

\newtheorem{theorem}{Theorem}[section]
\newtheorem{proposition}[theorem]{Proposition}
\newtheorem{definition}[theorem]{Definition}
\newtheorem{remark}[theorem]{Remark}

\usepackage{etoolbox}

\makeatletter
\AtBeginDocument{
  \setlength{\abovedisplayskip}{4pt plus 1pt minus 2pt}
  \setlength{\belowdisplayskip}{4pt plus 1pt minus 2pt}
  \setlength{\abovedisplayshortskip}{4pt plus 1pt minus 2pt}
  \setlength{\belowdisplayshortskip}{4pt plus 1pt minus 2pt}
  \setlength{\jot}{4pt}
}
\makeatother

\begin{document}

\title{\bfseries Numerical approximation of the first $p$-Laplace eigenpair}

\author{
Hannah~Potgieter\thanks{Department of Mathematics, Simon Fraser University, Burnaby, British Columbia, Canada} \thanks{Corresponding author. Email: hpa59@sfu.ca}
\and Razvan~C.~Fetecau\footnotemark[1]
\and Steven~J.~Ruuth\footnotemark[1] }

\date{\today}

\maketitle

\begin{abstract}
We approximate the first Dirichlet eigenpair of the $p$-Laplace operator for $2 \leq p < \infty$ on both Euclidean and surface domains. We emphasize large $p$ values and discuss how the $p \to \infty$ limit connects to the underlying geometry of our domain. Working with large $p$ values introduces significant numerical challenges. We present a surface finite element numerical scheme that combines a Newton inverse-power iteration with a new domain rescaling strategy, which enables stable computations for large $p$. Numerical experiments in $1$D, planar domains, and surfaces embedded in $\mathbb{R}^3$ demonstrate the accuracy and robustness of our approach and show convergence towards the $p \to \infty$ limiting behavior.

\vspace{0.5cm}

\end{abstract}

\maketitle


\section{Introduction}
\label{sec:intro}
The $p$-Laplacian is a nonlinear generalization of the classical Laplacian operator, given in Euclidean domains by
\[\pLap u = \Div(|\grad u|^{p-2} \grad u),\]
where $2 \leq p < \infty$ is the considered range of $p$ values. Note that $p = 2$ reduces to the linear Laplace operator. For surface (and more generally, manifold) domains, the differential operators are intrinsic and $p=2$ reduces to the Laplace-Beltrami operator.  In this work, we are concerned with the numerical approximation of the first Dirichlet $p$-Laplace eigenpair on both Euclidean and surface domains, and in particular we investigate large $p$ values.

As $p$ becomes large, numerically handling the $p$-Laplace operator presents challenges and has significant geometric implications associated with the bounded domain $\Omega$ under consideration. This motivates our numerical study. Approximating the first eigenpair of the $p$-Laplacian for large $p$ introduces several difficulties. The $p$-Laplacian is highly nonlinear and becomes increasingly degenerate as $p$ increases~\cite{Takac2004}. 
For the nonlinear case ($p > 2$) we have eigenfunctions lying in $C^{1, \sigma}(\overline{\Omega})$ where $0 < \sigma < 1$~\cite{Barles1988, Bellonietal2006, KawaiNakauchi2003}, which means the numerical approximation method needs to be robust against reduced smoothness. Several numerical methods have been employed to approximate the first eigenpair of the $p$-Laplacian~\cite{Biezuneretal2012, BognarSzabo2003,Horak2011,Patraetal2019}. We build on these methods and employ a surface finite element method (SFEM) discretization combined with a Newton inverse-power iteration scheme. 

The study of the first eigenpair of the $p$-Laplace operator has attracted attention, both theoretically and numerically. The $p \to \infty$ limit has connections with the geometry of the underlying domain. Juutinen et al.~\cite{Juutinenetal1999} established that the $p$th root of the first eigenvalue $\lamp$ converges to the inverse of the maximum distance-to-boundary, i.e., $\lamp^{1/p} \to \| \mathrm{dist}(\cdot, \partial \Omega) \|_\infty^{-1}$
as $p \to \infty$, and as noted in~\cite{Grosjean2005}, this also holds for surface domains. 
Further, the first eigenfunction, suitably normalized, converges to the distance-to-boundary function under a certain symmetry assumption on the domain, which we discuss in Section~\ref{sec:first}. 
Also, Barles~\cite{Barles1988} and Kawohl \& Lindqvist~\cite{KawohlLindqvist2006} proved in the Euclidean setting that the first eigenpair is unique (modulo scaling the eigenfunction) and that the first eigenfunction is positive. There are some works which extend some of these results beyond Euclidean space. Grosjean~\cite{Grosjean2005} extended the eigenvalue limit as $p \to \infty$ to Riemannian manifolds, and Belloni et al.~\cite{Bellonietal2006} examined the entire limiting eigenvalue problem on domains equipped with Finsler metrics. 

In the one-dimensional case, we can  characterize the first eigenpair of the $p$-Laplacian explicitly. This is the only case for which we know the explicit eigenpairs for any $p \in [2, \infty]$, as derived by Lindqvist~\cite{Lindqvist1993} and Edmunds \& Lang~\cite{EdmundsLang2015} using generalized trigonometry. For our numerical purposes, these explicit expressions of the eigenpairs allow us to test the performance of our numerical schemes on exact solutions.

Numerically, we may use the inverse-power method to update the eigenvalue and eigenpair alongside one another. In Euclidean domains, Bozorgnia~\cite{Bozorgnia2016} proved convergence to the unique first eigenpair by alternately updating $\lamp$ and $\up$. Existing numerical approaches are however not robust against large $p$ values on general domains. Additionally, numerical literature which considers surface domains for $p > 2$ is presented for specialized cases. For example, El Soufi et al.~\cite{ElSoufietal2012} numerically approximate the first eigenpair for low $p$ values on the unit sphere by discretizing the radially symmetric ordinary differential equation. Lanza et al.~\cite{Lanzaetal2025} consider surfaces represented as graphs also for low $p$ values, primarily $p \in (1,2)$. Finite element, constrained descent, and spline-based formulations~\cite{Biezuneretal2012,BognarSzabo2003,Horak2011,Patraetal2019} have all been implemented  for Euclidean domains, but report eigenpairs only for small values of $p$ (up to about $p = 10$). However, Biezuner et al.~\cite{Biezuneretal2012} show computed eigenvalues up to $p=290$ for the restricted case of Euclidean unit balls.
Recently, Bozorgnia et al.~\cite{Bozorgniaetal2024} introduced a monotone scheme specifically for the $p=\infty$ case  and report results on planar domains. To our knowledge, no works have yet bridged the gap between (small) finite $p$ and the $p \to \infty$ limit for the principal eigenpair on different domains.  

Studying the first Dirichlet $p$-Laplace eigenpair for large $p$ values is important both computationally and in applications such as shape optimization~\cite{Antunes2019, Mohammadietal2019}. From a numerical standpoint, an accurate approximation of the first eigenpair is a prerequisite for algorithms that compute higher modes, since these typically rely on the principal mode as an initialization or reference~\cite{Horak2011, Lanzaetal2025, YaoZhou2007}. In addition to its algorithmic role, the first eigenpair itself carries meaningful geometric information: as $p$ increases, the eigenfunction develops a distance-like structure in the large $p$ limit that reflects the shape of the domain. 

Beyond the first eigenpair, several works in the Euclidean setting have focused on the computation of the second Dirichlet $p$-Laplacian eigenpair through its characterization as an optimal bi-partition problem. In symmetric domains, Bozorgnia \& Arakelyan~\cite{BozorgniaArakelyan2023} exploit this structure to construct an algorithm specifically targeting the second eigenpair, while also emphasizing the associated optimal bi-partitioning problem as an application in its own right. More generally, Bobkov \& Ğalimov~\cite{BobkovGalimov2025} propose an iterative nodal-balancing framework that is motivated by the second eigenvalue problem but whose convergence behavior depends on the choice of initialization, and may therefore lead to higher eigenpairs in practice. The use of $p$-Laplace eigenpairs has recently attracted attention in modern data-driven and geometric applications. In particular, Lanza et al.~\cite{Lanzaetal2025} proposed a variational Alternating Direction Method of Multipliers (ADMM) framework for computing multiple graph $p$-Laplacian eigenpairs, demonstrating promising results for clustering and dimensionality reduction tasks. Their formulation, however, is theoretically restricted to the range $1 < p \leq 2$, where convexity of the subproblems is preserved.

\subsection{Contributions}
\label{subsec:contrib}

This work investigates numerical approximation of the first Dirichlet 
$p$-Laplace eigenpair for $2 \leq p < \infty$ on both Euclidean and surface 
domains, with particular emphasis on the large $p$ regime. We introduce a 
Newton inverse-power iteration scheme combined with a surface finite element 
(SFEM) discretization, extending prior FEM-based approaches to the surface 
setting. A continuation method in $p$, together with a domain rescaling 
strategy, enables stable computation of the first eigenpair for large $p$ values. We provide what appears to be the first systematic treatment of the surface $p$-Laplacian eigenvalue problem across a wide range of $p$.

Previous numerical studies, to our knowledge, have reported results only up 
to $p = 10$ unless restricting to Euclidean unit balls, whereas we show computations for $p$ as large as $100$. This extended $p$ range allows us to explore the limiting relationship between the first eigenpair and the distance-to-boundary function, bridging the gap between finite $p$ computations and the asymptotic $p \to \infty$ regime. A key feature of our approach is the new domain rescaling strategy, which mitigates eigenvalue growth or decay with respect to $p$, thereby improving the numerical conditioning of the problem. 

We validate the method through one-dimensional convergence tests against known 
analytical solutions and through self-convergence studies in higher dimensions 
where exact solutions are unavailable. To support the numerical findings, we include in Appendix~\ref{app:sinp-regularity} a derivation establishing the regularity class $C^{1,1/(p-1)}$ of the one-dimensional eigenfunction which explains the observed loss of smoothness for large $p$. Finally, we discuss the large $p$ behavior of the eigenpairs on surfaces, and our computational observations are consistent 
with the expected asymptotic properties of the $p$-Laplacian.

\section{Dirichlet $p$-Laplace eigenvalue problem}
\label{sec:Dirichlet}

For a given surface $S$ lying in $\mathbb{R}^3$, we consider the $p$-Laplace operator on the surface $S$, given by 
\[
\pLapS u = \DivS \Bigl( |\gradS u |^{p-2} \gradS u \Bigr),
\]
where the superscript $S$ denotes a differential operator intrinsic to the surface and $|\cdot|$ indicates the Euclidean norm. 

For surfaces of codimension $1$, the intrinsic gradient and divergence are given in terms of the corresponding Euclidean operators by
\begin{equation*}
    \gradS u = \grad \tilde u - \mathbf{n} (\mathbf{n} \cdot \grad  \tilde u), \qquad 
    \DivS( \mathbf{ v}) = \Div (\mathbf{\tilde v}) - \mathbf{n}^{\top} \grad \mathbf{\tilde v} \: \mathbf{n},  
\end{equation*}
where $\mathbf{n}$ is the unit vector normal to the surface $S$~\cite{DziukElliott2013}. Here $\tilde u$ and $\mathbf{\tilde v}$ are smooth extensions of $u$ and $\mathbf{v}$ to a tubular neighborhood of $S$. 

Let $\Omega$ be an open, bounded subset of $S$, with $p \geq 2$ fixed. The Dirichlet $p$-Laplace eigenvalue problem on $\Omega \subset S$ (cf.,~\cite{Mao2014}) is given by
\begin{subequations}
    \begin{alignat}{3}
    - \pLapS \up 
    = \lamp |\up |^{p-2} \up, & \qquad  \text{ in } \Omega, \label{eqn:pLapEval}\\
    \up = 0, & \qquad \text{ on } \partial \Omega,  \\
    \|\up\|_{\infty} = 1 & \label{eqn:normalize}.
    \end{alignat}
    \label{eqn:e-problem}
\end{subequations}
 Note that we impose the normalization condition~\eqref{eqn:normalize} to ensure uniqueness, as otherwise solutions of~\eqref{eqn:pLapEval} are only unique up to rescaling. Another common normalization is $\|\up\|_{p} = 1$~\cite{Grosjean2005}; however for numerical convenience we choose the infinity norm. 

We interpret the Dirichlet $p$-Laplace eigenvalue problem in the classical weak sense as in~\cite{Lindqvist1993}, which considers Euclidean domains. Note the weak form generalizes directly to the surface setting~\cite{DziukElliott2013}. Moreover, the weak formulation lends itself naturally to the 
SFEM discretization employed in our computations. Specifically, the pair $(\lamp,\up)\in \mathbb{R}^+ \times W^{1,p}_0(\Omega)$,
with $\up\not\equiv 0$, is called an eigenpair if
\[
\int_{\Omega} \bigl|\gradS \up\bigr|^{p-2}\,
      \gradS \up \cdot \gradS \varphi \, dA
   = \lamp
     \int_{\Omega} |\up|^{p-2} \up \,\varphi \, dA, 
   \qquad
   \forall \varphi \in  C_0^\infty (\Omega).
\]

The first (principal) eigenpair admits a variational characterization as the minimizer of a nonlinear Rayleigh quotient, which we will introduce in Section~\ref{sec:first}. In contrast, higher eigenpairs do not enjoy such a simple minimization property because the operator is nonlinear and the space $W_0^{1,p}(\Omega)$ is not a Hilbert space for $p > 2$. 
Consequently, orthogonality relations between eigenfunctions are lost, and higher modes must be defined through more delicate minimax principles~\cite{Horak2011, Lanzaetal2025, Lindqvist1993}. 
There remain many open problems concerning higher eigenpairs. While the variational spectrum of the Dirichlet $p$-Laplacian on bounded domains is known to form a discrete sequence of eigenvalues tending to $+ \infty$, it is not known in general whether this sequence exhausts all possible eigenvalues of the nonlinear operator; see~\cite{Lindqvist1993} and references therein.
For these reasons, we will restrict our attention to the principal eigenpair, which is central to understanding the nonlinear behavior of the operator.

For $p > 2$, it is known that the first eigenfunction $\up$ is a positive function lying in $C^{1, \sigma}(\overline{\Omega})$ with $0 < \sigma < 1$~\cite{Barles1988, Bellonietal2006, KawaiNakauchi2003}. 
In Appendix~\ref{app:sinp-regularity} we show that $\sigma \to 0$ as $p \to \infty$ for the one-dimensional case. Our numerical results in Section~\ref{sec:results} suggest that in higher dimensions, similarly to $1$D, the smoothness reduces with increasing $p$.

Of particular interest in our study is the large $p$ behavior and the limiting behavior of both eigenvalues and eigenfunctions as $p \to \infty$. In this limit, the divergence structure of the $p$-Laplacian is lost, and solutions must be interpreted in a viscosity sense. We will discuss the asymptotic behavior of the first eigenpair and its associated geometric properties in the next section.

\section{The first eigenpair}
\label{sec:first}

The first $p$-Laplace eigenvalue is given by
\begin{equation*}
    \lamp := \inf_{u \in W_0^{1, p}, u \not\equiv 0} R_p(u; \Omega), 
\end{equation*}
where $R_p(u; \Omega)$ is the nonlinear Rayleigh quotient 
\begin{equation}
    R_p(u; \Omega) = \frac{\int_\Omega |\gradS u |^p \: dA}{\int_\Omega | u |^p\: dA} = \frac{\|\gradS u\|_p^p}{\|u\|_p^p}.
    \label{eqn:Rayleigh}
\end{equation}
Minimizing the Rayleigh quotient gives rise to the Euler-Lagrange equation~\eqref{eqn:pLapEval}~\cite{Grosjean2005}. The nonlinear term on the right-hand side of~\eqref{eqn:pLapEval} arises naturally from minimizing the Rayleigh quotient, subject to a normalization constraint. The minimizer $u_p$ represents the first (or principal) eigenfunction of the $p$-Laplacian and is unique up to rescaling~\cite{Barles1988}.

The stability of the principal frequency $\lamp$ as $p$ varies is discussed by Lindqvist~\cite{Lindqvist1993qoutients}, where it is proven that
\begin{equation}
    p \left(\lamp\right)^{1/p} < s \left(\lambda_s\right)^{1/s}
    \label{eqn:mono}
\end{equation}
when $1 < p < s < \infty$. Note that there is not a similar monotonicity inequality for just the principal eigenvalues themselves, and in particular $\lambda_p$ may increase or decrease with $p$ depending on the domain as evidenced by a rescaling relation discussed later in Section~\ref{subsec:largep}. The proof of this monotonicity property given in~\cite{Lindqvist1993qoutients} extends directly to surfaces, with primary ingredients being a strategic choice of test function and application of the H\"older inequality. 
 We confirm that our numerical approximations of the eigenvalues adhere to~\eqref{eqn:mono} -- see Figure~\ref{fig:13_ineq_check} in Section~\ref{sec:results}.

The first eigenfunction possesses several noteworthy properties. In particular, the first eigenfunction is simple and does not change sign (see~\cite{Lindqvist1993} for a proof in the Euclidean setting). 
Literature for the principal $p$-Laplace eigenpair on manifolds has largely focused on analytical properties and inequalities~\cite{Grosjean2005, KawaiNakauchi2003, Mao2014}. Numerical works appear limited to specialized cases such as low $p$ values on the unit sphere~\cite{ElSoufietal2012}. In the following subsection we elaborate on the limiting behavior of the principal eigenpair as $p \to \infty$.

\subsection{Asymptotic behavior}
\label{subsec:asymptotic}

The first eigenpair tells us geometric information about the underlying domain. In particular, as $p \to \infty$ we find a connection to the distance function. Section~\ref{sec:results} demonstrates this connection through our numerical experiments. Moreover, knowing the asymptotic behavior of the eigenvalues guides the development of our domain rescaling strategy which allows  computations for large $p$ (see Section~\ref{subsec:largep}). 

Formally, taking a limit $p \to \infty$ involving the $p$-Laplacian Rayleigh quotient leads to the infinity eigenvalue and associated eigenvalue problem. Raising the quotient~\eqref{eqn:Rayleigh} to the power $1/p$ and taking the infimum, gives
\[
\Lamp :=  \lamp^{1/p} = \inf_{u \in W_0^{1, p}, u \not\equiv 0}  \frac{\|\gradS u\|_p}{\|u\|_p}.
\]
Then, formally taking the $p \to \infty$ limit we obtain the infinity eigenvalue
    \[\LamInf := \inf_{u \in W_0^{1, \infty}, u \not\equiv 0}  \frac{\|\gradS u\|_\infty}{\|u\|_\infty}. \]

The limiting PDE for the first eigenfunction of~\eqref{eqn:e-problem} as $p\to \infty$ is given by 
\begin{subequations}
    \begin{alignat}{3}
    \min \left \{|\gradS \uInf| - \LamInf  \uInf, -\infLapS \uInf \right \} = 0, & \qquad \text{in } \Omega, \label{eqn:limitpdea}\\
    \uInf = 0, & \qquad \text{on } \partial \Omega,  \\
    \|\uInf\|_{\infty} = 1 & , 
    \end{alignat}
    \label{eqn:limitpde}
\end{subequations}
where $\infLapS u = \left< \gradS u , \HessS \: u \gradS u\right>$ and $\HessS$ denotes the (intrinsic) Hessian on $S$. The limiting equations are derived in the Euclidean settings in~\cite{Juutinenetal1999} and generalized to non-Euclidean spaces in~\cite{Bellonietal2006}.  The derivation of~\eqref{eqn:limitpde} follows the same arguments used in~\cite{Juutinenetal1999} and extends to compact Riemannian manifolds almost directly, with minor notational changes and a standard application of Arzelà–Ascoli for compactness. We refer the reader to~\cite{Bellonietal2006} for the non-Euclidean setup.

Juutinen et al.~\cite{Juutinenetal1999} also show that 
\begin{equation}
    \lim_{p \to \infty } \Lamp =  \LamInf = \frac{1}{\underset{\bx \in \Omega}{\max} \:  \distx} = \frac{1}{R},
    \label{eqn:lamLimit}
\end{equation}
where $R$ is the radius of the largest ball which may be inscribed in $\Omega$. As noted in~\cite{Grosjean2005} and shown in~\cite{Bellonietal2006}, this holds beyond the Euclidean case and in particular for compact Riemannian manifolds. This limit can be used to ensure that our computed eigenvalues have the correct limiting behavior as $p \to \infty$. 

For some simple symmetric domains, 
the limiting first eigenfunction is given by the normalized distance-to-boundary function given by
\begin{equation}
\label{eqn:uinf}
\uInf(\bx ) = \frac{\distx}{\| \mathrm{dist}(\cdot,\partial \Omega)\|_\infty}.
\end{equation}
Note that the distance function $\mathrm{dist}(\cdot,\partial \Omega)$ is not differentiable everywhere in $\Omega$, so $\uInf$ needs to be interpreted as a solution of~\eqref{eqn:limitpde} in the viscosity sense.

More precisely, as in~\cite{Juutinen1998} we define the ridge set as
\begin{equation}
\label{eqn:calR}
\mathcal{R} := \{ \bx  \in \Omega:  \operatorname{dist}(\cdot,\partial \Omega) \text{ is not differentiable at } \bx \},
\end{equation}
and the maximal distance-to-boundary set (which is a subset of $\mathcal{R}$) as
\begin{equation}
\label{eqn:calM}
\mathcal{M} := \{ \bx  \in \Omega:  \distx = \| \operatorname{dist}(\cdot,\partial \Omega)\|_\infty\}.
\end{equation}
For instance, on a geodesic ball both sets consist of the center point. On the other hand, in a square (or cube) the ridge set $\mathcal{R}$ comprises the diagonals connecting opposite vertices, and only their intersection at the center forms $\mathcal{M}$, so in this case $\mathcal{M}$ is a strict subset of $\mathcal{R}$.

In the special case when $\mathcal{R} = \mathcal{M}$, it can be shown that the normalized geodesic distance-to-boundary function indeed satisfies~\eqref{eqn:limitpde} in the viscosity sense -- see Appendix~\ref{app:surface-limits}. This assumption is satisfied by domains such as geodesic balls, stadiums, and annuli~\cite{CharroParini2010}. On such domains, we may check that numerically computed solutions are indeed converging to the normalized distance-to-boundary function as $p \to \infty$. 

\subsection{Analytic solution in $1$D} 
\label{subsec:1d}

In one dimension, the eigenvalue problem~\eqref{eqn:e-problem} on $\Omega = (a, b)$ is given by
\begin{subequations}
    \begin{alignat*}{3}
    -(|\up'|^{p-2} \up')' = \lamp |\up|^{p-2} \up & , \quad  a < x < b,\\
    \up(a) =  \up(b) = 0 & , \\
    \|\up\|_{\infty} = 1 & .
    \end{alignat*}
\end{subequations}
It is only in $1$D that the exact eigenpairs are known~\cite{Lindqvist1993}. 
Thus, it is crucial to make use of the exact $1$D eigenpairs in assessing performance of numerical schemes. In Section~\ref{subsec:1d-numerics}, we perform $1$D numerical tests. 

Lindqvist~\cite{Lindqvist1993} shows that the first eigenpair is given by
\begin{subequations}
\label{eqn:1dexact}
\begin{alignat}{2}
        \lamp &= (p-1) \left(\frac{\pi_p}{b-a}\right)^p ,\label{eqn:lamp1d} \\
        \up(x) &= \sin_p \left(\pi_p\frac{x-a}{b-a} \right). \label{eqn:up1d}
\end{alignat}
\end{subequations}
This eigenpair involves $\pi_p$ and $\sin_p$ which come from a generalized trigonometry. Edmunds \& Lang~\cite{EdmundsLang2015} discuss details about the generalized trigonometric functions. They define $F_p : [0, 1] \to \mathbb{R}$ by
\[F_p(x) := \int_0^x \frac{1}{(1-t^p)^{1/p}} dt,\]
for $1 < p < \infty$ which is a strictly increasing one-to-one function with range $[0, \pi_p/2]$. We generalize $\pi$ via
\[\frac{\pi_p}{2} =  F_p(1)  = \frac{\pi}{p \sin(\pi/p)},\]
and for $x \in [0, \pi_p/2]$, we define the generalized sine function 
\[\sin_p(x) := F_p^{-1}(x).\]
Note that for $p = 2$ we obtain $\sin_2(x) = \sin(x)$ and $\pi_2 = \pi$.

Since we have the exact eigenvalues, we can expand both $\lamp$ and $\lamp^{1/p}$ about $\epsilon = 1/p = 0$ to show that 
\begin{equation}
\label{eqn:lamp-expand}
    \lamp = \left(\frac{2}{b-a}\right)^p\left(p + \frac16 (\pi^2-6)+\frac{1}{72}\pi^2 (\pi^2-12) \frac1p + \mathcal{O}\left(\frac{1}{p^2}\right) \right),
\end{equation}
and 
\begin{equation}
\label{eqn:lamproot-expand}
  \lamp^{1/p} = \left(\frac{2}{b-a}\right) \left(1 +\frac{\log p}{p}+\frac16 \left( \pi^2-6+3(\log p)^2 \right) \frac{1}{p^2}+ \mathcal{O}\left(\frac{(\log p)^3}{p^3}\right) \right).
\end{equation}

Taking the $p \to \infty$ limit in~\eqref{eqn:lamproot-expand}, we find 
\begin{equation*}
    \lim_{p \to \infty} \left(\lamp\right)^{1/p} = \frac{2}{b-a},
\end{equation*}
and indeed $2/(b-a) = 1/R$ where $R$ is the radius of the largest $1$D ball lying within the domain $(a,b)$. Moreover, we find that $\lamp^{1/p}$ approaches $\LamInf$ at the rate $(\log p)/p$ for large $p$. Also, we see that when $b-a=2$, $\lamp$ grows linearly in $p$ for large $p$.

We obtain high–accuracy reference values of $\sin_{p}(x)$ by solving the initial–value problem
\begin{equation*}
\begin{aligned}
v' &= |w|^{\frac{1}{p-1}}\mathrm{sign}(w), &
v(0) &= 0,\\
w' &= -(p-1)|v|^{p-1}\mathrm{sign}(v), &
w(0) &= 1,
\end{aligned}
\end{equation*}
using the \texttt{MATLAB} code from~\cite{Vasinova2016} with the
\texttt{ode45} solver. Both absolute and relative tolerances are set to $10^{-16}$.

Figure~\ref{fig:1_sinp} shows plots of $\up(x)$ for $p = 2, 3, 4, 8$, along with the $p \to \infty$ limiting solution given by $\uInf(x) = 1-|x|$; note that the domain $(a,b)$ has the symmetry assumption discussed in Section~\ref{subsec:asymptotic} and hence, in the $p \to \infty$ limit $\up(x)$ converges to the normalized distance function given by~\eqref{eqn:uinf}. Additionally, we plot the pointwise differences $\up - \uInf$ between solutions, which shows how $\up(x)$ approaches the distance function as $p$ increases. Note that $\up$ lies above $\uInf$. 

\begin{remark} 
\label{rmk:up-reg}
For $p>2$, near its maximum $x_0 = (a+b)/2$, the leading behavior of the eigenfunction $u_p(x)$ is
\begin{equation*}
     \up(x)
     \sim 1-K_p \left|x-x_0\right|^{p/(p-1)},
\end{equation*}
where $K_p > 0$ is an explicitly defined constant -- see Appendix~\ref{app:sinp-regularity} for the derivation. Consequently, $\up$ belongs to the regularity class $C^{1,1/(p-1)}(\overline{\Omega})$, and in particular, $\up'$ is  H\"older continuous at $x_0$, with exponent $1/(p-1)$. Thus, $\up$ is not in $C^2$ for any $p>2$, and  the smoothness of $\up$ decreases as $p$ increases (as the exponent $1/(p-1)$ decreases).  
\end{remark}

\begin{figure}
    \centering
    \includegraphics[scale = 0.35]{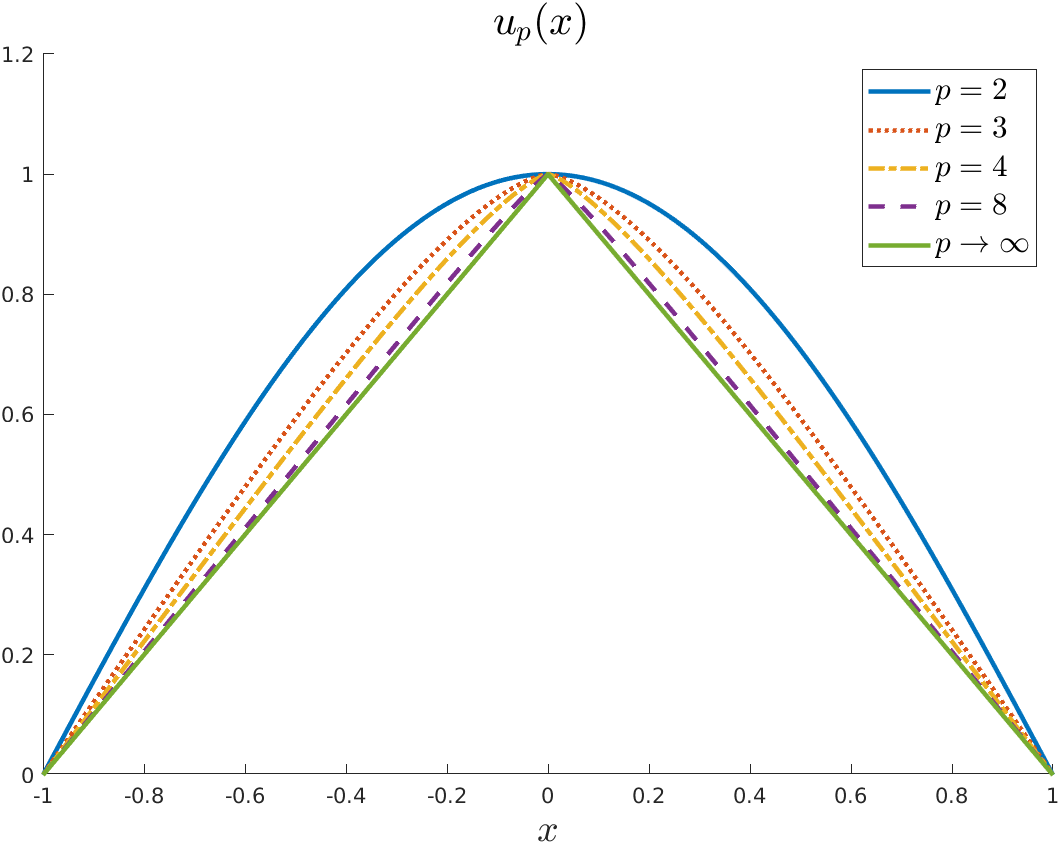}
    \hspace{1.5em}
    \includegraphics[scale = 0.35]{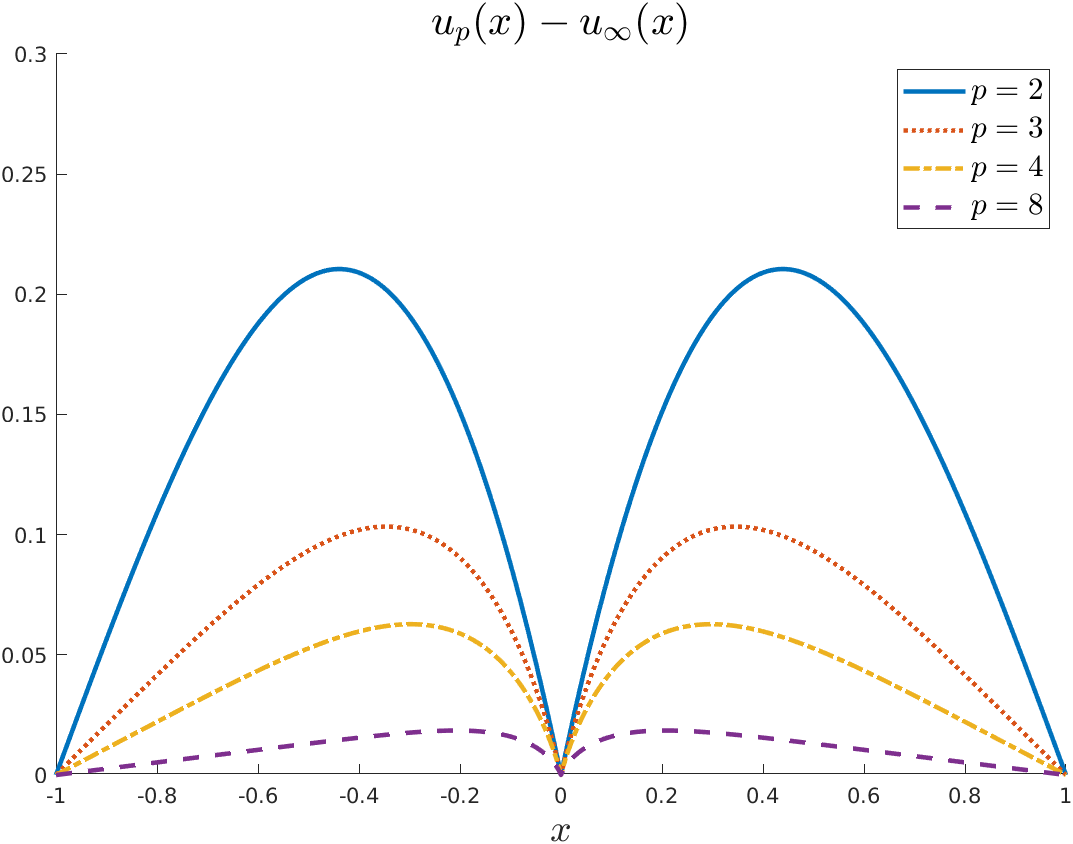}
    \caption{Left: $\up(x) = \sin_p(\pi_p(x+1)/2)$ converges to $\uInf(x) = 1-|x|$ as $p$ increases.  Right: pointwise errors between $\up(x)$ and $\uInf(x)$. }
    \label{fig:1_sinp}
\end{figure}

\section{Numerical scheme}
\label{sec:scheme}

The continuous domain $\Omega$ is replaced by a discrete computational domain which we call $\Omega_h$. We will use a finite element space
\begin{equation*}
   V_h = \{ \phi_h \in C(\Omega_h) : \phi_h|_T  \text{ is in } Q^1 \text{ for each } T \in \mathcal{T}_h\},
\end{equation*}
where $Q^1$ is the space of bilinear polynomials on each quadrilateral cell in $\mathcal{T}_h$, the collection of faces on the surface mesh. Here, the mesh size is defined as $h := \max_{T \in \mathcal{T}_h} \mathrm{diam}(T)$. Numerically approximated eigenvalues will be denoted by $\lampApprox$ and eigenfunctions by $\upApprox$. We use polynomial degree~$1$ Lagrange elements since $p$-Laplace eigenfunctions are not $C^2$ for $p > 2$. 

We note that the finite element method extends to surface domains. For the classical linear surface Poisson problem, $\Delta^S u = f$, Dziuk \& Elliott~\cite{DziukElliott2013} show that under the standard regularity assumption $ \| f-f_h\|_{L_2(\Omega)} \leq c_f h^2$, SFEM achieves  optimal convergence rates analogous with those for Euclidean domains: 
\begin{equation*}
    \| u-u_h\|_{L_2(\Omega)} \leq c h^2, \quad \| \gradS (u-u_h)\|_{L_2(\Omega)} \leq c h .
\end{equation*} 
Although these rates pertain to a linear problem which is only directly applicable when $p=2$, they show how the geometric and interpolation errors behave when extending to a surface framework. Dziuk \& Elliott~\cite{DziukElliott2013} further discuss how SFEM is applicable for nonlinear diffusion equations including the parabolic surface $p$-Laplace equation. Here, we consider the stationary $p$-Laplace eigenvalue problem. 

We use a Newton inverse-power iteration in order to solve the nonlinear problem  for $p \in [2, \infty)$, similarly to~\cite{Patraetal2019}. This means we alternate updating the eigenvalue and eigenfunction approximations. The eigenvalue approximation is obtained by computing the discrete nonlinear Rayleigh quotient~\eqref{eqn:Rayleigh}. For the eigenfunction updates, we seek solutions of 
\begin{equation*}    
 F(u; \: \lambda) := \gradS \cdot \left(\gamma(u; \eta) \gradS u\right) + \lambda |u|^{p-2} u = 0,
\end{equation*}
where the coefficient with small regularization constant $\eta > 0$ is given by
\begin{equation}
\label{eqn:gamma}
     \gamma(u; \eta) := (\eta^2+|\gradS u|^2)^{\frac{p-2}{2}} > 0.
\end{equation}
We solve for the eigenfunction update at Newton iteration $n+1$ (denoted by superscripts) via
\begin{align*}
    F'(u^n; \lambda)[\delta u^n] & = - F(u^n; \: \lambda) \label{NewtonIt1}, \\
    u^{n+1} & = u^n + \beta^n \delta u^n,
\end{align*}
where $F'(u; \lambda)[\delta u]$ is the derivative of $F$ in the direction of $\delta u$. Here, $\eta$ is a small regularization constant which ensures positive-definiteness of the generalized stiffness and mass matrices appearing in the resulting linearized problem, allowing a conjugate gradient (CG) solver to be used.
In our experiments we take $\eta = 10^{-5}$. The damping parameter, $0< \beta^n \leq 1$, is chosen at each iteration via a quadratic interpolation line search with a sufficient decrease condition~\cite{Nocedal2006}. 
That is, we find $0< \beta^n \leq 1$ such that
\begin{equation}
\label{eqn:line-search}
    R_p(u^n + \beta^n \delta u^n; \Omega_h) \leq R_p(u^n; \Omega_h) + c_1 \beta^n R_p'(u^n; \Omega_h) [\delta u^n], 
\end{equation}
where $c_1 = 10^{-3}$ is the prescribed tolerance, $R_p$ is given by~\eqref{eqn:Rayleigh}, and $R_p'(u) [\delta u]$ is the derivative of $R_p$ in the direction of $\delta u$. 

This gives the Newton iteration equations as 
\begin{subequations}
\begin{alignat}{2}
       \gradS \cdot \bigg( \gamma(u^n; \eta) \gradS \delta u^n+ (p-2)\gamma(u^n; \eta) & \frac{\gradS u^n \cdot \gradS \delta u^n}{\eta^2+|\gradS u^n|^2}  \gradS u^n\bigg) + \lambda^n (p-1)|u^n|^{p-2} \delta u^n
     \label{eqn:pLapNewtonIt1}  \\ &
     = - \gradS \cdot \bigg( \gamma(u^n; \eta) \gradS u^n\bigg) - \lambda^n |u^n |^{p-2} u^n  , \nonumber \\
     u^{n+1} &= u^n + \beta^n \delta u^n. 
    \end{alignat}   
\end{subequations}

The weak formulation of~\eqref{eqn:pLapNewtonIt1} is obtained by multiplying with a test function $\phi$ and integrating by parts on both sides. Reducing  to a finite dimensional space with basis $\{\phi_h^0, ..., \phi_h^{N-1} \} \in V_h$, we can write the update $\delta u^n \in H^1(\Omega)$, as 
\begin{equation*}
    \delta u^n = \sum_{j = 0}^{N-1} \delta U^n_j \phi_h^j.
\end{equation*}

Similarly to~\cite{Patraetal2019}, we find the resulting linear system for the update coefficients $\delta U^n$  vector at each iteration to be
\begin{equation}
    \left( K^n - \lambda^n M^n \right) \delta U^n = b^n,
    \label{eqn:linearSystem}
\end{equation}
with the entries of the generalized stiffness matrix $K^n$ given by
\begin{equation}
\label{eqn:K}
    K^n_{ij} :=\int_{\Omega_h} \gradS \phi_h^i \cdot  
     \gamma(u^n; \eta) \left ( I + \frac{(p-2)}{\eta^2 + |\gradS u^n|^2}   [\gradS u^n] \otimes [\gradS u^n] \right )\gradS \phi_h^j dA_h,
\end{equation}
the entries of the weighted mass matrix $M^n$ given by 
\begin{equation}
\label{eqn:M}
    \qquad M^n_{ij}: = (p-1) \int_{\Omega_h} |u^n|^{p-2} \phi_h^i \phi_h^j dA_h,
\end{equation}
and the entries of the right hand side $b^n$ given by
\begin{equation}
\label{eqn:b}
    b^n_{i} := \int_{\Omega_h} \left (\lambda^n |u^n|^{p-2} u^n \phi_h^i - \gradS \phi_h^i \cdot \left( \gamma(u^n; \eta) \gradS u^n \right) \right)  dA_h.
\end{equation}

In practice, following~\cite{Patraetal2019}, we instead solve the shifted  linear system 
$(K^n + \lambda^n M^n)\,\delta U^n = b^n$, rather than~\eqref{eqn:linearSystem}, 
as the latter is not necessarily positive definite.  
The shifted matrix is positive definite, since both $K^n$ and $M^n$ are 
positive definite and $\lambda^n>0$. Note that $M^n$ is positive definite since for the first eigenfunction we have $u^n >0$ under our normalization. This shift preserves the Newton update direction and therefore the convergence to the first eigenpair.

Our code is written in \texttt{C++} using the \texttt{deal.II}~\cite{dealII94} finite element library. In our implementation, the first Dirichlet $p$-Laplacian eigenpair is computed using a nonlinear inverse-power iteration~\cite{Bozorgnia2016}, in which the eigenfunction and eigenvalue are updated in an alternating fashion. The linear system which arises at each iteration becomes ill-conditioned when the gradient is small and when $p$ is large.  
We apply algebraic multigrid (AMG) preconditioning and use a CG solver at each iteration. The AMG preconditioner is implemented through the \texttt{TrilinosWrappers::PreconditionAMG} interface in \texttt{deal.II}. It is configured for symmetric elliptic operators by setting \texttt{elliptic = true} and \texttt{higher\_order\_elements = false}, with two smoother sweeps and an aggregation threshold of $0.02$. Symmetric successive over-relaxation (SSOR) preconditioning was also tested but proved less effective, typically requiring more than $100$ CG iterations even for moderate $p$ values. By comparison, AMG maintained a smaller ($< 50$) and more stable iteration count across different values of $p$ and mesh refinement.
We initialize with the linear $p = 2$ problem using the \texttt{SLEPc} Krylov-Schur eigensolver from the \texttt{deal.II} library. For $p > 2$, we need an initial guess very close to the exact solution in order to see convergence. Accordingly, we iterate in $p$ following Algorithm~\ref{alg:newton_inverse_power}. 

For the examples under our consideration, we take $\mathrm{tol}_\mathrm{Newton} = 10^{-7}$,  $\mathrm{tol}_\mathrm{CG} = 10^{-6}$, and $\delta p = 1$ unless the target $p$ value is not an integer or otherwise specified. We continue incrementing $p$ until the algorithm fails to converge usually due to failure to find a suitable descent direction according to the decrease condition~\eqref{eqn:line-search}. 

\begin{algorithm}
\caption{Newton inverse-power iteration in $p$}
\label{alg:newton_inverse_power}
\begin{algorithmic}[1] 
\Require Discrete domain $\Omega_h$, finite element space $V_h = \text{span}\{\phi_h^j\}_{j=0}^{N-1}$, 
initial eigenpair $(\lambda_{p_0,h}, u_{p_0,h})$, 
step size $\delta p$, target $p_{\max}$,  regularization constant $\eta$,
Newton tolerance $\mathrm{tol}_{\mathrm{Newton}}$, CG tolerance $\mathrm{tol}_{\mathrm{CG}}$, 
line search constant $c_1$.
\Statex

\State $i=0$.
\While{$p_i < p_{\max}$}
    \State $i \gets i+1$
    \State $p_i \gets \min\{p_{i-1}+\delta p, p_{\max}\}$ 
    \State Initialize $\lambda^0 \gets \lambda_{p_{i-1},h}$, $u^0 \gets u_{p_{i-1},h}$
    \For{$n = 0,1,2,\dots$ until convergence}
        \State Compute regularized coefficient $\gamma(u^n; \eta)$ via~\eqref{eqn:gamma}.
        \State Assemble $K^n$, $M^n$ and $b^n$ in the basis of $V_h$ via ~\eqref{eqn:K},~\eqref{eqn:M}, and~\eqref{eqn:b}.
        \State Solve $(K^n + \lambda^n M^n) \delta U^n = b^n$ using CG with AMG preconditioning, stopping when \[\|(K^n + \lambda^n M^n) \delta U^n - b^n\|_2/ \|b^n \|_2 \le \mathrm{tol}_{\mathrm{CG}}.\]
        \State Form update $\delta u^n = \sum_{j=0}^{N-1} \delta U^n_j \phi_h^j \in V_h$.
        \State Perform line search: find $0 < \beta^n \le 1$ such that 
        \[R_{p_i}(u^n + \beta^n \delta u^n; \Omega_h) \le R_{p_i}(u^n; \Omega_h) + c_1 \beta^n R_{p_i}'(u^n; \Omega_h)[\delta u^n].\]
        \State Update $u^{n+1} \gets u^n + \beta^n \delta u^n$.
        \State Normalize $u^{n+1} \gets u^{n+1} / \|u^{n+1}\|_\infty$.
        \State Update $\lambda^{n+1} \gets R_{p_i}(u^{n+1}; \Omega_h)$.
        \If{$\|u^{n+1}-u^n\|_2 / \|u^{n+1}\|_2 \le \mathrm{tol}_{\mathrm{Newton}}$}
            \State Converged: set $u_{p_i,h} \gets u^{n+1}$, $\lambda_{p_i,h} \gets \lambda^{n+1}$ and \textbf{break}.
        \EndIf
    \EndFor
\EndWhile
\State \textbf{Return:} $\{(\lambda_{p_i,h},u_{p_i,h})\}$ for all $p_i$.
\end{algorithmic}
\end{algorithm}

\subsection{Handling large $p$ values}
\label{subsec:largep}

Iterating in $p$ when applying the Newton scheme is one measure we take in order to handle larger $p$ values.  Additionally, we rescale the domain in order to prevent numerical blowup or underflow problems. From~\eqref{eqn:lamLimit} we see that $\lamp \sim (\max_{\bx  \in \Omega} \: \distx)^{-p}$ for large $p$ so if $\max_{\bx  \in \Omega} \: \distx$ is not close to $1$ then $\lamp$ may become very large or small as $p$ increases.

We apply the transformation
\begin{equation}
    T_\alpha : \Omega \to  \alpha \Omega := \{ \alpha \bx  \mid x =  (x_1, \cdots, x_n) \in \Omega \subset \mathbb{R}^n \},
    \label{eqn:transformation}
\end{equation}
with dilation parameter $\alpha$ chosen such that $\underset{\bx  \in \Omega}{\max} \;  \distx$ is close to $1$. Figure~\ref{fig:2_rescale} visualizes this transformation on a rectangular domain. Applying transformation~\eqref{eqn:transformation} impacts the eigenvalues. In particular, 
\begin{equation}
    R_p(\up; \alpha \Omega) = \alpha^{-p} R_p(\up; \Omega),
    \label{eqn:scaling}
\end{equation}
where $R_p$ is given by~\eqref{eqn:Rayleigh}. This can be seen by applying a change of variables. Thus, we can readily recover the eigenvalues on the original domain $\Omega$.

\begin{figure}
    \centering
    \includegraphics[scale = 0.4]{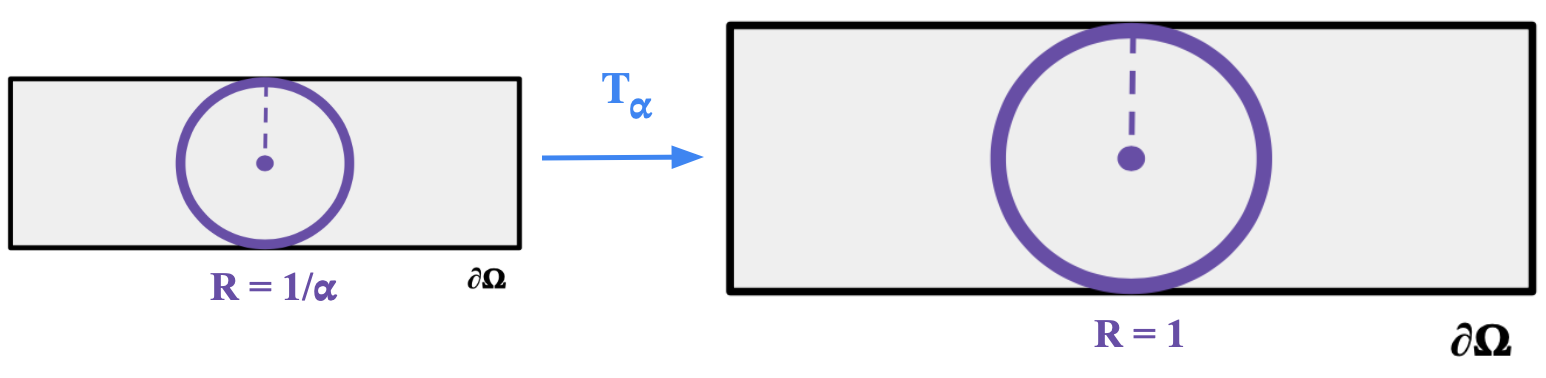}
    \caption{Illustration of the rescaling transformation 
    $T_\alpha$. On the left, the original 
    rectangular domain contains a largest inscribed circle of radius 
    $R = 1/\alpha$. After applying $T_\alpha$, the domain is scaled 
    so that the maximal distance to the boundary becomes $R = 1$.}
    \label{fig:2_rescale}
\end{figure}

For general domains and in particular for surfaces, choosing a suitable $\alpha$ value is not trivial. We take $\alpha = d_m^{-1}$ where $d_m$ is an approximation of $\max_{\bx  \in \Omega} \: \distx$. This can be found via a handful of alternating direction method of multipliers (ADMM) iterations for the $p$-Poisson problem~\cite{Potgieteretal2025}. This helps because the $\lamp^{1/p}$ values then approach a value close to $1$. Figure~\ref{fig:3_geodesic} shows an example on a hand mesh where the maximal geodesic distance $d_m$ is obtained. This approximation of $\max_{\bx  \in \Omega} \distx$ 
is then used to set the scaling parameter $\alpha = d_m^{-1}$.

\begin{figure}
    \centering
        \includegraphics[scale = 0.4]{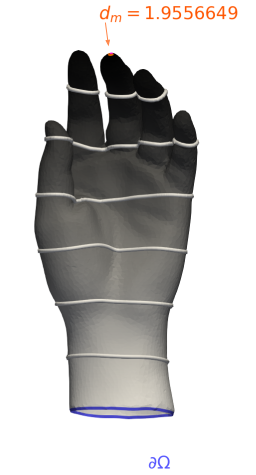}
    \caption{Illustration of the maximal distance approximation $d_m$ on a hand surface mesh. 
The marked point (orange) corresponds to the location where the distance to the boundary 
$\partial \Omega$ (blue) is maximized, giving $d_m \approx 1.956$. 
This value is used in the scaling $\alpha = d_m^{-1}$.}
    \label{fig:3_geodesic}
\end{figure}

In practice, we may employ an adaptive thresholding strategy to keep the computed eigenvalues within a stable range. Let $\tau_{-} > 0$ and $\tau_{+} > 0$  with $\tau_{-} < \tau_{+}$ denote prescribed lower and upper thresholds, respectively. Typically, we use $\tau_{-} = 2$ and $ \tau_{+} = 20$ in our computations.  At each step of the continuation in $p$, once an approximation of $\lampApprox$ is obtained, we check:
 \begin{align*}
    & \text{if } \lampApprox > \tau_{+}, \quad \text{then }   \alpha \;\gets\; \alpha \cdot 2^{1/p},\\
    & \text{if } \lampApprox < \tau_{-} , \quad \text{then }  \alpha \;\gets\; \alpha \cdot 2^{-1/p}.
 \end{align*}
Here $\alpha$ denotes the cumulative scaling factor applied to the domain, so the current working domain is always $\alpha \Omega$. When $\lampApprox$ is outside the range $[\tau_{-}, \tau_{+}]$, $\lampApprox$ is shifted by a factor of two back into range.  
This adaptive strategy is particularly useful when the maximal distance to the boundary
is not known a priori. For planar domains $\distx$ can often be determined explicitly or computed with reasonable effort ~\cite{Crane2020}. However, for surface domains the maximum geodesic distance is generally more difficult to obtain. In such cases, adaptive thresholding avoids the need for exact geometric information while still preventing blowup or underflow in the eigenvalues. The modification of the $p$ continuation Algorithm~\ref{alg:newton_inverse_power} with adaptive thresholding is presented in Algorithm~\ref{alg:newton_inverse_power_rescale}.

\begin{algorithm}
\caption{Newton inverse-power iteration in $p$ with dynamic rescaling}
\label{alg:newton_inverse_power_rescale}
\begin{algorithmic}[1]
\Require Discrete domain $\Omega_h$, finite element space $V_h = \text{span}\{\phi_h^j\}_{j=0}^{N-1}$, 
initial eigenpair $(\lambda_{p_0,h}, u_{p_0,h})$, 
step size $\delta p$, target $p_{\max}$,  regularization constant $\eta$,
Newton tolerance $\mathrm{tol}_{\mathrm{Newton}}$, CG tolerance $\mathrm{tol}_{\mathrm{CG}}$, 
line search constant $c_1$, eigenvalue scaling tolerances $\tau_-$ and $\tau_+$.
\Statex

\State $i=0,  \alpha_0 = 1$.
\While{$p_i < p_{\max}$}
    \State $i \gets i+1$
    \State $p_i \gets \min\{p_{i-1}+\delta p, p_{\max}\}$ 
    \State Initialize $\lambda^0 \gets \lambda_{p_{i-1},h}$, $u^0 \gets u_{p_{i-1},h}$, $\alpha_i \gets \alpha_{i-1}$
 \If{$\lambda^0 > \tau_+$}  
    \State $\alpha_i \gets \alpha_{i-1} \cdot 2^{1/p_i}$
\State Rescale domain: $(\Omega_h)_i \gets (\alpha_i / \alpha_{i-1}) (\Omega_h)_{i-1}$
\State Rescale eigenfunction: $u^0(x) \gets u^0(\alpha_{i-1} x / \alpha_{i})$
    \State Rescale eigenvalue: $\lambda^0 \gets \lambda^0 / 2$
\EndIf
\If{$\lambda^0 < \tau_-$}
    \State $\alpha_i \gets \alpha_{i-1} \cdot 2^{-1/p_i}$
\State Rescale domain: $(\Omega_h)_i \gets (\alpha_i / \alpha_{i-1}) (\Omega_h)_{i-1}$
\State Rescale eigenfunction: $u^0(x) \gets u^0(\alpha_{i-1} x / \alpha_{i})$
    \State Rescale eigenvalue: $\lambda^0 \gets 2 \lambda^0$
\EndIf
    \For{$n = 0,1,2,\dots$ until convergence}
    \State Compute regularized coefficient $\gamma(u^n; \eta)$ via~\eqref{eqn:gamma}.
        \State Assemble $K^n$, $M^n$ and $b^n$ in the basis of $V_h$ via ~\eqref{eqn:K},~\eqref{eqn:M}, and~\eqref{eqn:b}.
        \State  Solve $(K^n + \lambda^n M^n) \delta U^n = b^n$ using CG with AMG preconditioning, stopping when \[\|(K^n + \lambda^n M^n) \delta U^n - b^n\|_2/ \|b^n \|_2 \le \mathrm{tol}_{\mathrm{CG}}.\]
        \State Form update $\delta u^n = \sum_{j=0}^{N-1} \delta U^n_j \phi_h^j \in V_h$.
        \State Perform line search: find $0 < \beta^n \le 1$ such that 
        \[R_{p_i}(u^n + \beta^n \delta u^n; (\Omega_h)_i) \le R_{p_i}(u^n; (\Omega_h)_i) + c_1 \beta^n R_{p_i}'(u^n; (\Omega_h)_i)[\delta u^n].\]
        \State Update $u^{n+1} \gets u^n + \beta^n \delta u^n$.
        \State Normalize $u^{n+1} \gets u^{n+1} / \|u^{n+1}\|_\infty$.
        \State Update $\lambda^{n+1} \gets R_{p_i}(u^{n+1}; (\Omega_h)_i)$.
        \If{$\|u^{n+1}-u^n\|_2 / \|u^{n+1}\|_2 \le \mathrm{tol}_{\mathrm{Newton}}$}
            \State Converged: set $u_{p_i,h} \gets u^{n+1}$, $\lambda_{p_i,h} \gets \lambda^{n+1}$ and \textbf{break}.
        \EndIf
    \EndFor
\EndWhile
\State \textbf{Return:}
$\{ (\hat \lambda_{p_i,h}, \hat u_{p_i,h}) \}$ for all $p_i$, where $\hat \lambda_{p_i,h} = \alpha_i^{p_i}\lambda_{p_i,h} $, $\hat u_{p_i,h}(x) = u_{p_i,h}(\alpha_i x)$ on the original domain $\Omega_h$.
\end{algorithmic}
\end{algorithm}

\section{Computational results}
\label{sec:results}

This section presents a series of numerical experiments that illustrate the accuracy and convergence properties of our SFEM Newton inverse-power scheme for computing the first $p$–Laplace eigenpair.  
We begin with the one–dimensional interval, which is the only case where an analytical solution is known. This allows a clean verification of the predicted convergence rates, which to our knowledge have not previously been reported in the literature. We then move to planar Euclidean domains (disks and squares) and explore the limiting behavior as $p \to \infty$. Finally, we report computations on some two–dimensional surfaces embedded in $\mathbb{R}^3$ to demonstrate that the method extends
naturally to curved manifolds. These surface experiments include both simple surfaces (hemisphere and half torus) and more complicated geometries from Crane's~\cite{craneModelRepo} 3D Model Repository and the Chen et al.~\cite{offMESHES} benchmark set for segmentation. All downloaded meshes are converted to quadrilateral meshes using the open‐source tool \texttt{tethex}~\cite{tethex}
and in the case of the hand model the raw mesh has been preprocessed in \texttt{MeshLab}~\cite{meshlab} by truncating at the wrist and applying smoothing to obtain a natural boundary. 

All experiments shown are carried out with the \texttt{deal.II}~\cite{dealII94} finite element library in \texttt{C++}.  We use
piecewise bilinear elements for the spatial discretization and the computational quadrilateral meshes are either generated using the built-in \texttt{GridGenerator} utilities or read in from a surface mesh file. We use quadrilateral surface meshes as they provide comparable geometric fidelity to triangle surface meshes with fewer elements, and are natively supported in the \texttt{deal.II} framework. For visualization of the eigenfunctions, we export to \texttt{ParaView}~\cite{paraview}.

\subsection{One–dimensional case $\Omega = \{ x \in \mathbb{R} : a < x < b \}$}
\label{subsec:1d-numerics}

We first consider the interval domain, where the first eigenpair is known exactly in terms of $\sin_p$ and $\pi_p$ as discussed in Section~\ref{subsec:1d}. An optimally scaled choice is $a = -1$, $b = 1$ so that the maximum distance to the boundary equals~$1$. Accordingly, we fix $a = -1$ and $b = 1$ in our computations with Algorithm~\ref{alg:newton_inverse_power}, using the high–accuracy reference solution generated as described in Section~\ref{subsec:1d}. This makes for a good benchmark convergence check of the numerical scheme. We successively refine a uniform $1$D mesh and record both the $L^2$ error of the eigenfunction and the relative error of the eigenvalue. 

Table~\ref{tab:1_sinp_conv} reports the results for a few representative values of $p$.
The eigenvalue relative errors converge at roughly second order for small $p$, but the rate drops slightly as $p$ grows, reflecting the decreasing regularity of the eigenfunctions, as discussed in Section~\ref{subsec:1d} (see Remark~\ref{rmk:up-reg}).
The eigenfunction error is above first order when $p$ is small and gradually approaches approximately linear grid convergence as 
$p$ increases. For large $p$, coarse meshes show irregular or overly optimistic rates; only after the mesh is refined past a certain resolution do the convergence rates settle. 

\begin{table}[] 
    \centering
    \begin{subtable}{1\textwidth}
    \centering  
    \begin{tabular}{ccccc}
    \multicolumn{3}{l}{$p=3$} &  \multicolumn{2}{r}{$\lambda_3 = 3.5360952$} \\
    \toprule
     \# cells &  $\upApprox$ $L^2$  error  & Rate & $\lampApprox$ relative error & Rate \\
     \midrule
    $64$ & $   4.4979e-04$ & \phantom{$-$}- & $3.2482e-04$ & \phantom{$-$}-\\
    $128$ & $  1.5426e-04$ & \phantom{$-$}$1.54$ & $8.3566e-05$ & \phantom{$-$}$1.96$\\
     $256$  & $  5.3145e-05$ & \phantom{$-$}$1.54$ & $2.1308e-05$ & \phantom{$-$}$1.97$\\
     $512$ & $  1.8094e-05$ & \phantom{$-$}$1.55$ & $5.4004e-06$ & \phantom{$-$}$1.98$\\
     $1024$ & $   5.7784e-06$ & \phantom{$-$}$1.65$ & $1.3631e-06$ & \phantom{$-$}$1.99$\\
     $2048$ & $   1.7233e-06$ & \phantom{$-$}$1.75$ & $3.4307e-07$ & \phantom{$-$}$1.99$\\
     \bottomrule 
\end{tabular}
    \end{subtable}\\
    \vspace{0.5em}
    \begin{subtable}{1\textwidth}
     \centering
     \begin{tabular}{ccccc}
         \multicolumn{3}{l}{$p=10$} &  \multicolumn{2}{r}{$\lambda_{10} = 10.6149413$} \\
    \toprule
     \# cells & $\upApprox$ $L^2$  error  & Rate & $\lampApprox$ relative error & Rate \\
     \midrule
    $64$ & $   4.5011e-04$ & \phantom{$-$}- & $1.0095e-03$ & \phantom{$-$}- \\ 
    $128$ & $   2.1537e-04$ & \phantom{$-$}$1.06$ & $3.0407e-04$ & \phantom{$-$}$1.73$\\
     $256$  & $   1.0168e-04$ & \phantom{$-$}$1.08$ & $8.7949e-05$ & \phantom{$-$}$1.79$\\
     $512$ & $   4.7525e-05$ & \phantom{$-$}$1.10$ & $2.4745e-05$ & \phantom{$-$}$1.83$\\
     $1024$ & $   2.1981e-05$ & \phantom{$-$}$1.11$ & $6.8241e-06$ & \phantom{$-$}$1.86$\\
     $2048$ & $   1.0219e-05$ & \phantom{$-$}$1.11$ & $1.8536e-06$ & \phantom{$-$}$1.88$\\
     \bottomrule  
\end{tabular}
    \end{subtable} \\
    \vspace{0.5em}
\begin{subtable}{1\textwidth}
    \centering
    \begin{tabular}{ccccc}
     \multicolumn{3}{l}{$p=50$} &  \multicolumn{2}{r}{$\lambda_{50} = 50.6390648$} \\
     \toprule
     \# cells & $\upApprox$ $L^2$  error  & Rate & $\lampApprox$ relative error & Rate \\
     \midrule
    $64$ & $   7.6419e-05$ & \phantom{$-$}- & $2.0586e-03$ & \phantom{$-$}-\\
    $128$ & $  4.6028e-05$ & \phantom{$-$}$0.73$ & $8.2748e-04$ & \phantom{$-$}$1.31$\\
     $256$  & $   2.5431e-05$ & \phantom{$-$}$0.86$ & $2.8879e-04$ & \phantom{$-$}$1.52$\\ 
     $512$ & $   1.3394e-05$ & \phantom{$-$}$0.93$ & $9.2792e-05$ & \phantom{$-$}$1.64$\\  
     $1024$ & $  6.8419e-06$ & \phantom{$-$}$0.97$ & $2.8315e-05$ & \phantom{$-$}$1.71$\\
     $2048$ & $  3.4617e-06$ & \phantom{$-$}$0.98$ & $8.3443e-06$ & \phantom{$-$}$1.76$\\
     \bottomrule 
\end{tabular}
    \end{subtable}\\
    \vspace{0.5em}
    \begin{subtable}{1\textwidth}
     \centering
     \begin{tabular}{ccccc}
          \multicolumn{3}{l}{$p=100$} &  \multicolumn{2}{r}{$\lambda_{100} = 100.642007$} \\
     \toprule
     \# cells & $\upApprox$ $L^2$  error  & Rate & $\lampApprox$ relative error & Rate \\
     \midrule
    $64$ & $   2.3932e-05$ & - & $2.2660e-03$ & \phantom{$-$}-\\
    $128$ & $   1.8642e-05$ & \phantom{$-$}$0.36$ & $1.0163e-03$ & \phantom{$-$}$1.16$\\
     $256$  & $   1.1496e-05$ & \phantom{$-$}$0.70$ & $4.1006e-04$ & \phantom{$-$}$1.31$\\
     $512$ & $  6.4386e-06$ & \phantom{$-$}$0.84$ & $1.4404e-04$ & \phantom{$-$}$1.51$\\   
     $1024$ & $   3.4202e-06$ & \phantom{$-$}$0.91$ & $4.6556e-05$ & \phantom{$-$}$1.63$\\
     $2048$ & $ 1.7749e-06$ & \phantom{$-$}$0.95$ & $1.4281e-05$ & \phantom{$-$}$1.70$\\
     \bottomrule  
\end{tabular}
    \end{subtable}
    \caption{Numerical convergence study of the computed $1$D eigenpairs toward the exact solutions given by~\eqref{eqn:1dexact}. 
   Both the $L^2$ error of the eigenfunction and the relative error of the eigenvalue are reported for various $p$ values. The $1$D mesh is uniform and successively refined across experiments. In each table, the true eigenvalue given by~\eqref{eqn:lamp1d} is reported in single precision.}
    \label{tab:1_sinp_conv}
\end{table}

Figure~\ref{fig:4_eig_many} shows the computed 1D eigenvalues (dashed green line), which exhibit linear growth in $p$ for large $p$, consistent with the asymptotic expansion~\eqref{eqn:lamp-expand}. The corresponding $p$th roots, shown on the right, numerically converge to the limiting value $\LamInf = 1$, approaching this limit at the asymptotic rate $(\log p)/p$ predicted by~\eqref{eqn:lamproot-expand}. These numerical results agree with the expected large $p$ behavior.

\begin{figure}
    \centering
    \includegraphics[scale = 0.52]{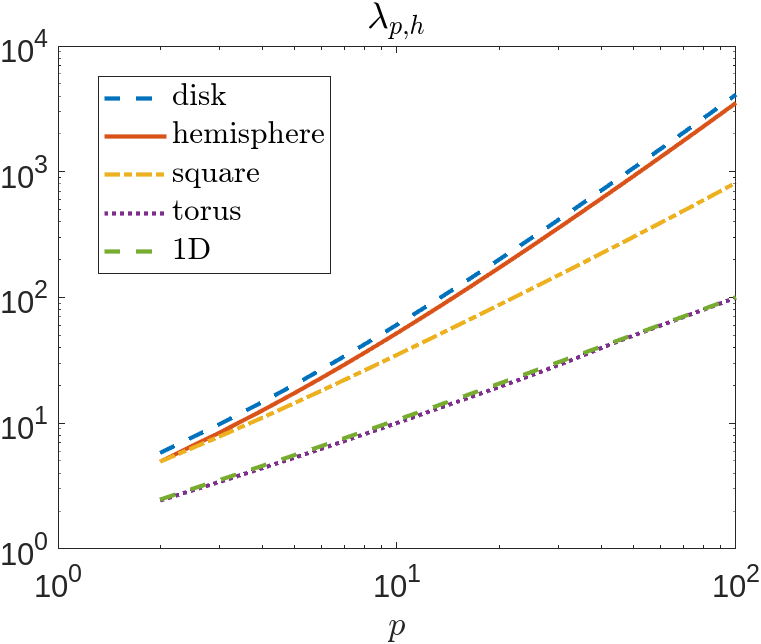}
    \hspace{1.5em}
         \includegraphics[scale = 0.52]{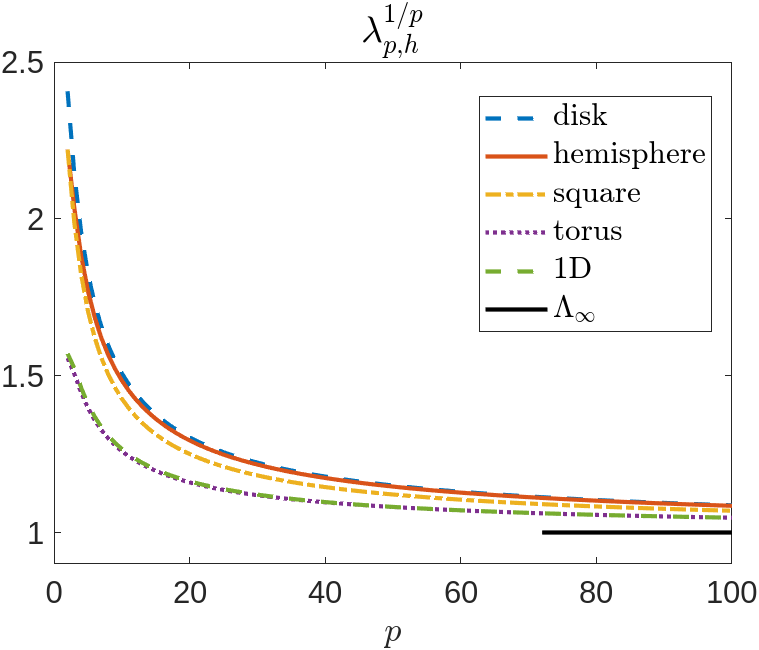}
    \caption{ Left: Eigenvalue approximations on some optimally scaled domains for $ 2 \leq p \leq 100$ shown in log-log scale. Right: Corresponding $p$th roots of eigenvalue approximations and limit $\LamInf$. The disk and hemisphere meshes consist of $327{,}680$ cells, the square mesh consists of $262{,}144$ cells, the half torus mesh consists of $1{,}572{,}864$ cells, and the $1$D mesh consists of $2{,}048$ cells. }
    \label{fig:4_eig_many}
\end{figure}

\subsection{Euclidean planar domains}
\label{subsec:planar}

We consider two simple domains, the disk and the square. For the disk, the condition $\mathcal{R} = \mathcal{M}$ (introduced in Section~\ref{subsec:asymptotic}) is satisfied, so the first eigenfunction converges to the distance-to-boundary function as $p \to \infty$.
However, the square does not satisfy $\mathcal{R} = \mathcal{M}$, and therefore the limiting eigenfunction as $p \to \infty$ is not known. We use Algorithm~\ref{alg:newton_inverse_power} for numerical computations on the disk and square since the maximum distance to the boundary is known. 

\subsubsection{Disk $\Omega = \{ (x, y) \in \mathbb{R}^2 : \sqrt{x^2+y^2} < R \}$ }

In our computations we consider the optimally scaled choice of the unit disk ($R=1$), so that the maximum distance to the boundary is $1$. However, eigenpairs may be recovered for different choices of $R$.  

Figure~\ref{fig:4_eig_many} shows that the computed eigenvalues on the unit disk 
increase the fastest among the examples shown in the plot. At $p=100$, the disk eigenvalue reaches approximately $4000$, whereas the corresponding $1$D eigenvalue is only around $100$. 

Figure~\ref{fig:5_disc} shows contour plots of the first eigenfunction $\upApprox$ on the disk for several values of $p$, ranging from $p=2$ to the limiting $p \to \infty$ case. As $p$ increases, the contour lines become more evenly spaced  and we approach the distance-to-boundary function, which is known to be the limiting $p\to\infty$ solution in this case, as discussed in Section~\ref{subsec:asymptotic} (see also Appendix~\ref{app:surface-limits}, in particular Remark~\ref{rmk:uinf-dist}).

Cross sections along the $y$-axis (Figure~\ref{fig:6_discCross}) further illustrate this convergence, with differences $\upApprox - \uInf$ showing the approach toward the limiting profile. Contrary to the one-dimensional case, where $\upApprox$ is always above the distance function for finite $p$, here on the disk this is not always true. For instance, the profiles for $p=5$ and $p=10$ show regions where $\upApprox$ falls slightly below $\uInf$.

Table~\ref{tab:2_disc_self} reports self-convergence results for both eigenfunctions and eigenvalues on successively refined meshes. Additionally, Table~\ref{tab:2_disc_self} includes the eigenvalue approximations $\lampApprox$. Results are shown for a few select $p$ values of $3, 10, 50,$ and $100$ to highlight the behavior for small, moderate, and large $p$. For small $p$, the eigenfunction numerically converges solidly above first order, while the eigenvalue exhibits approximately second-order convergence. As $p$ increases, the eigenfunction develops nonsmoothness near the origin, resulting in irregular convergence on coarse meshes. The rate gradually stabilizes above first order once the mesh is sufficiently refined. Similarly, the eigenvalue convergence rate decreases somewhat from second order for large $p$, reflecting the increased nonsmoothness in the corresponding eigenfunctions. This is consistent with our $1$D convergence test against the known analytical solution. Our computations show good agreement with the results reported by Horák~\cite{Horak2011} on the unit disk when using a comparable number of degrees of freedom, noting that Horák provides data only up to $p=10$, so the comparison is limited to that range. Note that results on different disk sizes will yield different eigenvalues according to the scaling relation~\eqref{eqn:scaling}.

Figure~\ref{fig:7_rescale_alph} highlights the importance of our domain scaling technique. The log-log plot shows how the first eigenvalue $\lampApprox = R_p(\upApprox; \alpha \Omega_h)$ varies with the scaling factor $\alpha$. For intermediate to large $p$ ($40 \leq p \leq 100$), the inset emphasizes the differences between scalings $\alpha = 0.9, 1, 1.1$. At the (seemingly innocent) extreme, $\alpha = 0.5$ (corresponding to a disk of radius $0.5$), eigenvalues exceed $10^{30}$
posing a numerical problem. Conversely, $\alpha = 2$ (corresponding to a disk of radius $2$) yields eigenvalues that are excessively small (around $10^{-20}$) for large $p$; such values are below machine epsilon, posing a numerical issue.
These observations illustrate that without proper scaling, even slight deviations from optimally scaled domains can lead to numerical blowup or underflow difficulties. 

\begin{figure}
    \centering
     \begin{subfigure}{0.18\textwidth} 
    \centering
         \includegraphics[scale = 0.09]{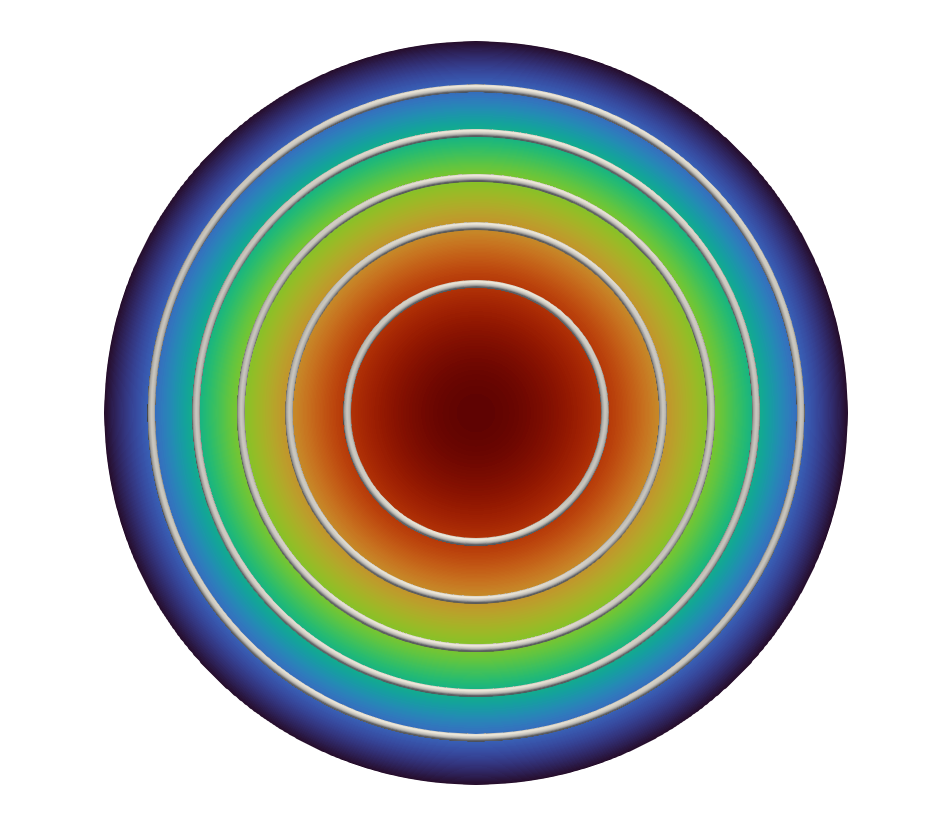}
         \subcaption{$p = 2$}
    \end{subfigure}
    \begin{subfigure}{0.18\textwidth} 
    \centering
         \includegraphics[scale = 0.09]{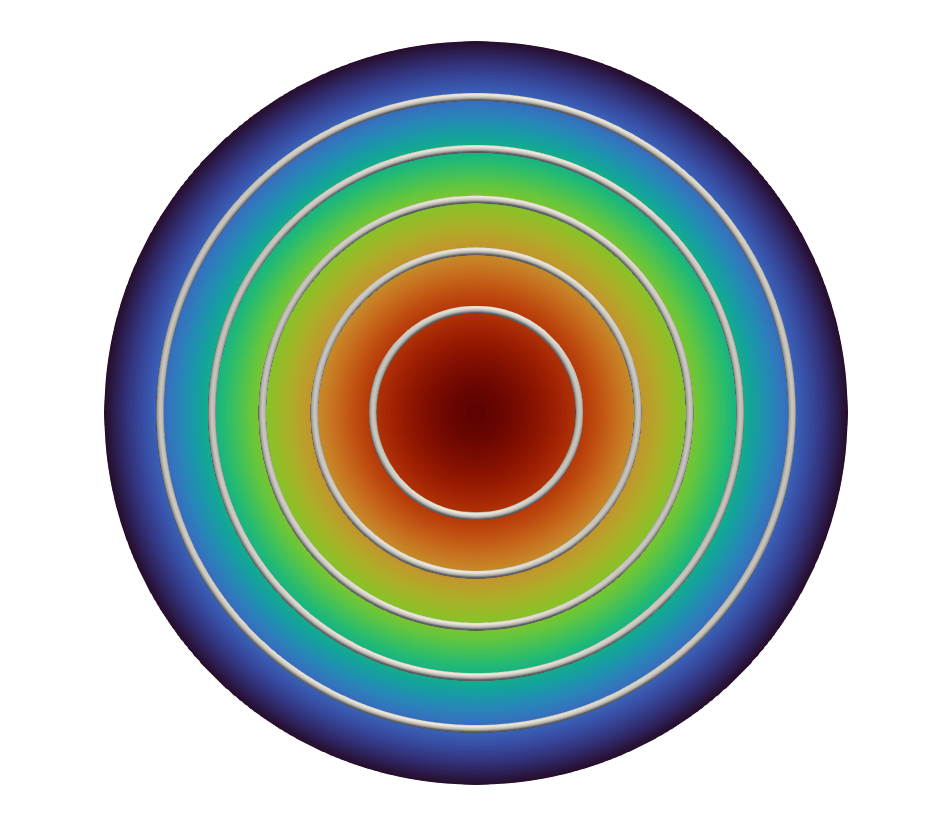}
         \subcaption{$p = 2.5$}
    \end{subfigure}
    \begin{subfigure}{0.18\textwidth}
    \centering
         \includegraphics[scale = 0.09]{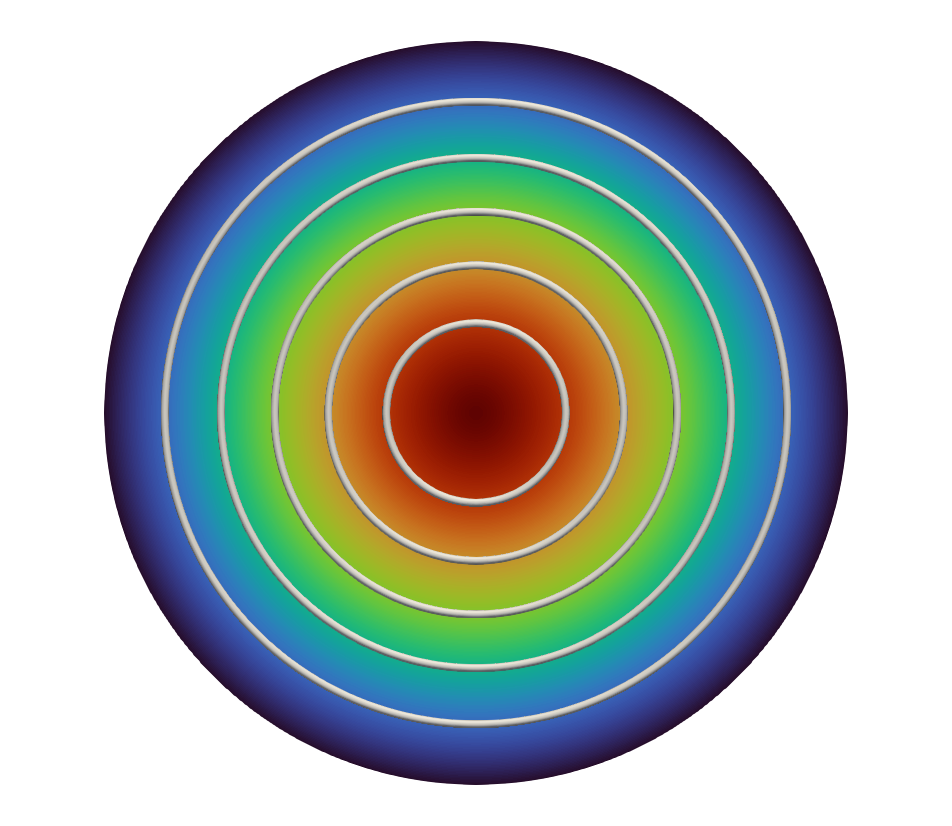}
          \subcaption{$p = 3$}
    \end{subfigure}     
    \begin{subfigure}{0.18\textwidth}
     \centering
         \includegraphics[scale = 0.09]{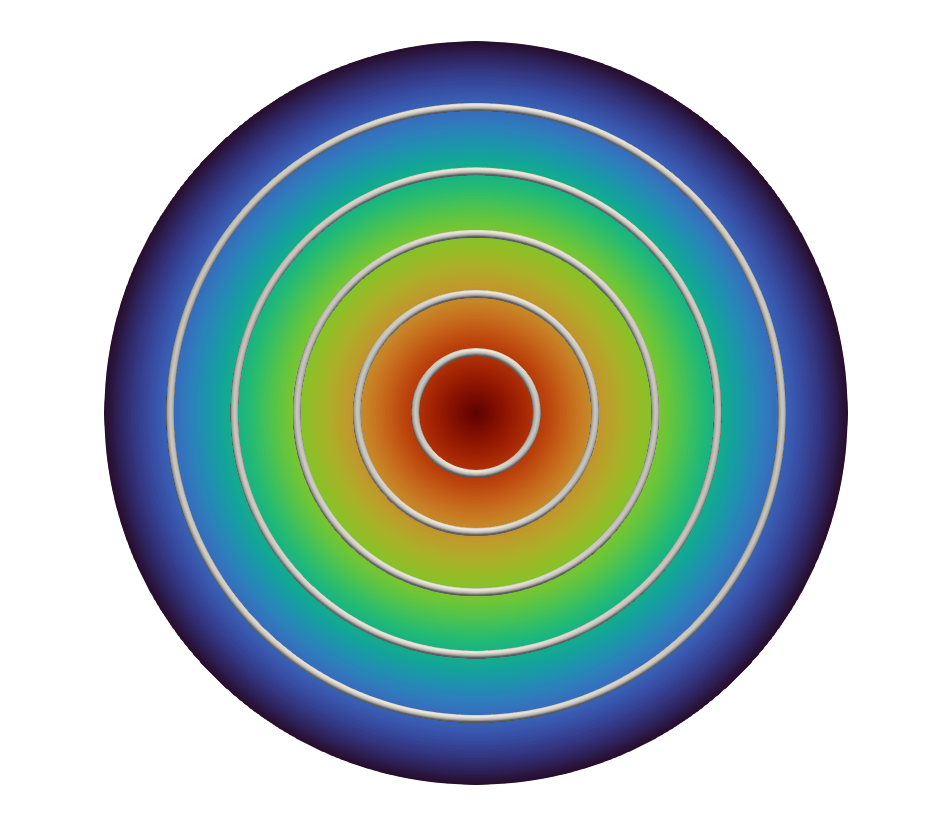}
          \subcaption{$p = 10$}
    \end{subfigure}
    \begin{subfigure}{0.18\textwidth}
     \centering
         \includegraphics[scale = 0.09]{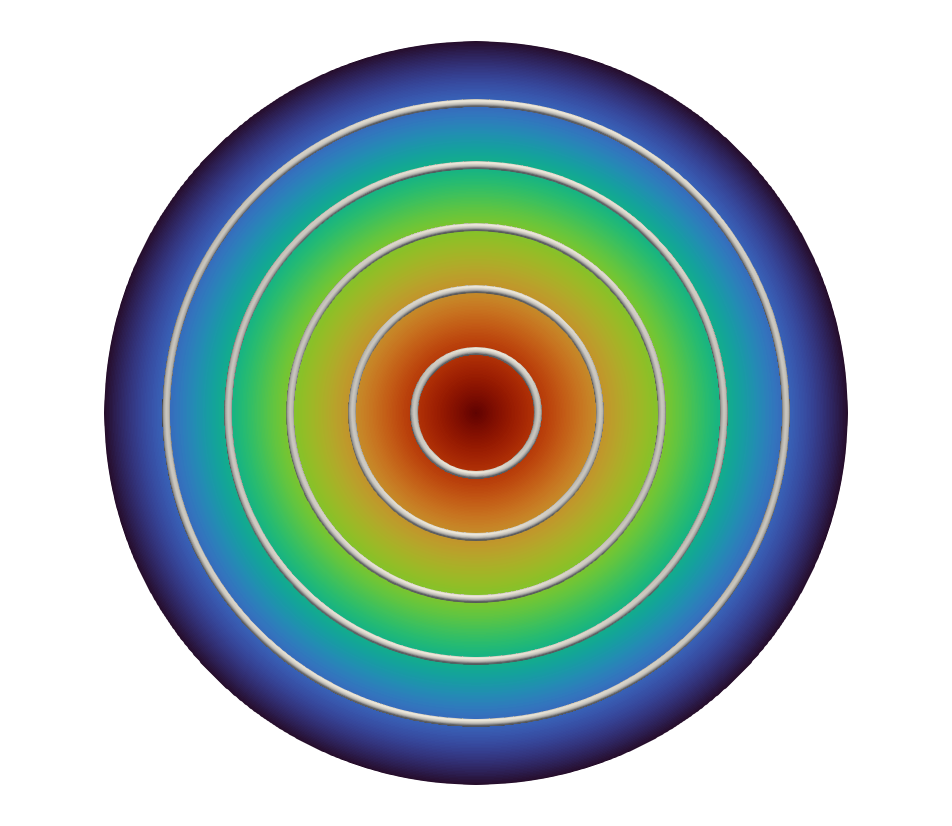}
          \subcaption{$p \to \infty$}
    \end{subfigure}
    \caption{Contour plots of the first eigenfunction $\upApprox$ on the disk. Plots (a)–(d) correspond to different values of $p$, and plot (e) shows the limiting solution as $p \to \infty$. All computations use a mesh with $327{,}680$ cells. Isolines are shown at the same levels across plots.}
    \label{fig:5_disc}
\end{figure}

\begin{figure}
    \centering
    \includegraphics[scale = 0.35]{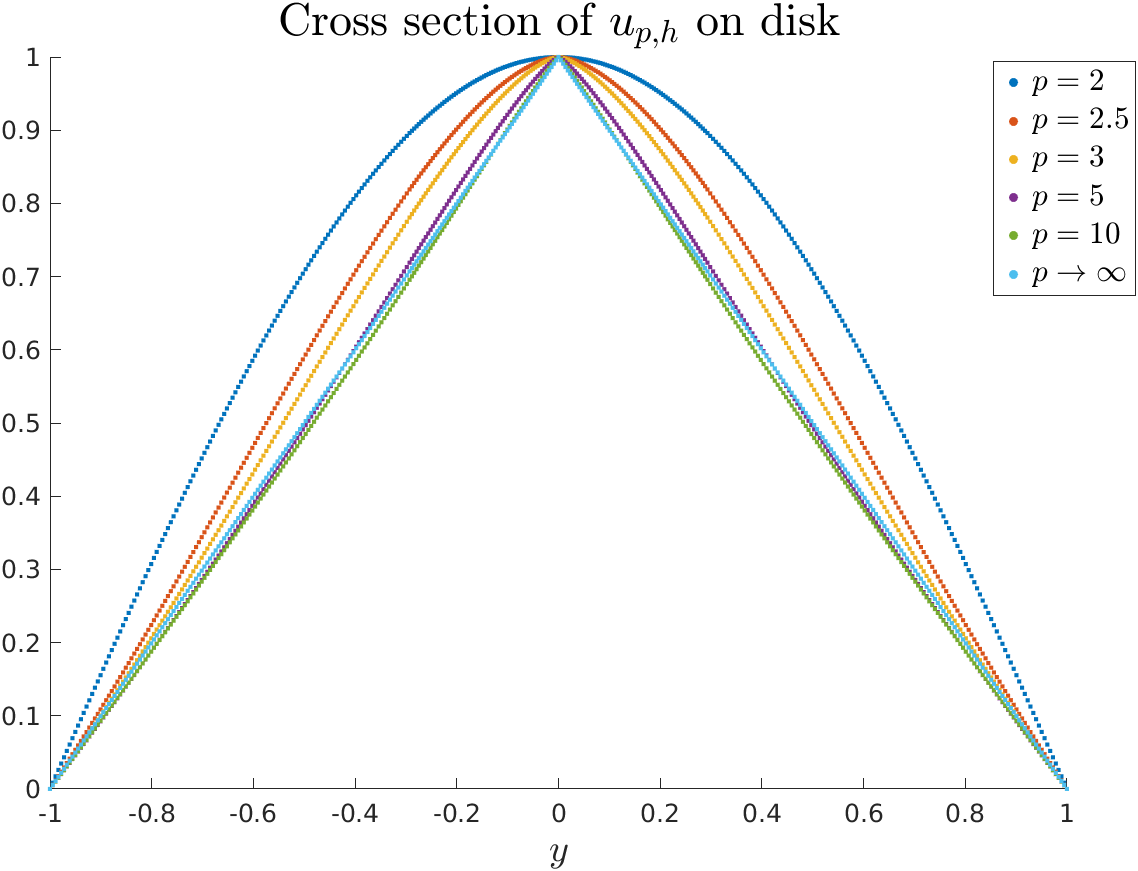}
      \hspace{1.5em}
      \includegraphics[scale = 0.35]{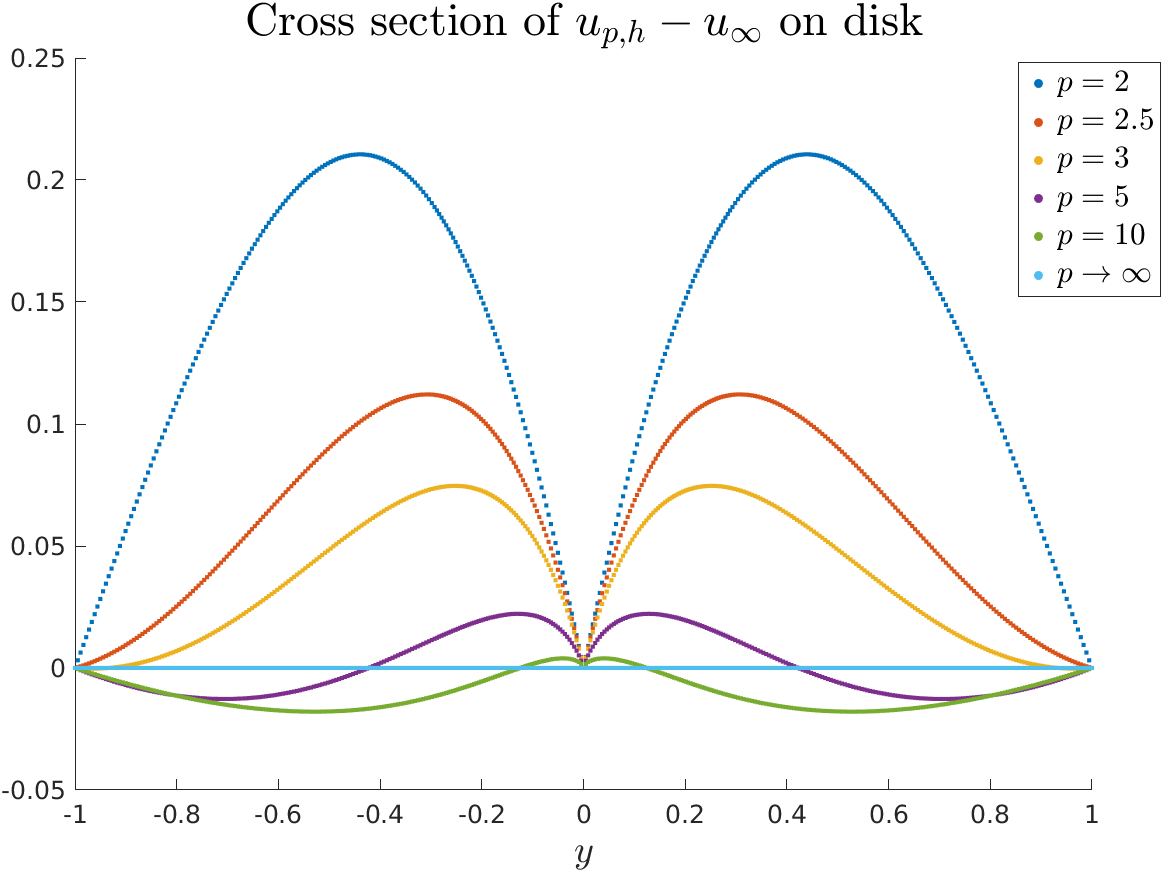}
    \caption{Left: Cross sections along $y$-axis of $\upApprox$ for various $p$ values on the disk obtained using mesh with $327{,}680$ cells. The limiting profile as $p \to \infty$ is shown in cyan. Right: The corresponding differences $\upApprox - \uInf$ illustrating the convergence of $\upApprox$ to the distance-to-boundary function.
    }
    \label{fig:6_discCross}
\end{figure}

\begin{table}[]
    \centering
    \begin{subtable}{1\textwidth} 
    \centering
\begin{tabular}{cccccc}
    \multicolumn{6}{l}{$p=3$} \\
     \toprule
     \# cells & $\upApprox$ $L^2$ error  & Rate & $\lampApprox$ relative error & Rate  & $\lampApprox$\\
     \midrule
     $320$ & $3.9306e-03$ & - & $1.6105e-02$ & - & $9.9900e+00$\\
     $1280$ & $1.1704e-03$ & 1.75 & $3.9897e-03$ & 2.01 &  $9.8709e+00$\\
     $5120$ & $3.9387e-04$ & 1.57 & $9.8437e-04$ & 2.02 &  $9.8413e+00$\\
     $20480$ & $1.6500e-04$ & 1.26 & $2.3444e-04$ & 2.07 & $9.8340e+00$\\   
      $81920$ & $8.1627e-05$ & 1.02 & $4.7046e-05$ & 2.32 & $9.8321e+00$ \\
      \bottomrule
\end{tabular}
    \end{subtable}\\
    \vspace{0.5em}
    \begin{subtable}{1\textwidth} 
     \centering
\begin{tabular}{cccccc}  
    \multicolumn{6}{l}{$p=10$} \\
     \toprule
     \# cells & $\upApprox$ $L^2$ error  & Rate & $\lampApprox$ relative error & Rate & $\lampApprox$ \\
     \midrule
     $320$ & $2.1400e-03$ & - & $8.2241e-02$ & - & $6.5715e+01$ \\
     $1280$ & $1.0672e-03$ & 1.00 & $2.1101e-02$ & 1.96 & $6.2002e+01$\\
     $5120$ & $5.2282e-04$ & 1.03 & $5.3232e-03$ & 1.99 & $6.1044e+01$ \\
     $20480$ & $2.2318e-04$ & 1.23 & $1.2816e-03$ & 2.05 & $ 6.0799e+01$ \\
      $81920$ & $7.2233e-05$ & 1.63 & $2.5761e-04$ & 2.31 & $6.0737e+01$\\
     \bottomrule   
\end{tabular}
    \end{subtable}\\
    \vspace{0.5em}
        \begin{subtable} {1\textwidth}  
         \centering                   
\begin{tabular}{cccccc}
    \multicolumn{6}{l}{$p=50$} \\
     \toprule
     \# cells & $\upApprox$ $L^2$ error  & Rate & $\lampApprox$ relative  error & Rate & $\lampApprox$  \\
     \midrule
     $320$ & $6.1781e-03$ & - & $1.7972e+00$ & - &  $3.0155e+03$\\
     $1280$ & $3.7824e-03$ & 0.71 & $4.9959e-01$ & 1.85 & $1.6166e+03$\\
     $5120$ & $1.9884e-03$ & 0.93 & $1.5316e-01$ & 1.71 &  $1.2432e+03$\\
     $20480$ & $9.0240e-04$ & 1.14 & $4.3962e-02$ & 1.80 &  $1.1254e+03$ \\
      $81920$ & $3.0849e-04$ & 1.55 & $9.9517e-03$ & 2.14 & $1.0888e+03$\\
      \bottomrule
\end{tabular}
    \end{subtable}\\
    \vspace{0.5em}
    \begin{subtable}{1\textwidth}  
     \centering
\begin{tabular}{cccccc}
    \multicolumn{6}{l}{$p=100$} \\
     \toprule
     \# cells & $\upApprox$ $L^2$ error  & Rate & $\lampApprox$ relative error & Rate & $\lampApprox$ \\
     \midrule
     $320$ & $7.2691e-03$ & - & $1.6580e+01$ & - & $7.2373e+04$ \\
     $1280$ & $4.6052e-03$ & 0.66 & $2.6480e+00$ & 2.65 &  $1.5018e+04$\\
     $5120$ & $ 2.4834e-03$ & 0.89 & $7.0793e-01$ & 1.90 &  $7.0313e+03$ \\
     $20480$ & $1.1488e-03$ & 1.11 & $2.0877e-01$ & 1.76 &  $4.9763e+03$\\  
      $81920$ & $3.9813e-04$ & 1.53 & $5.0948e-02$ & 1.63 & $4.3266e+03$ \\
      \bottomrule
\end{tabular}
    \end{subtable}\\
    \caption{
    Self-convergence study on the disk for several values of $p$, with reference taken from the solution on a mesh of $327{,}680$ cells. For each refinement level, the $L^2$ error of the eigenfunction $\upApprox$, the relative error of the first eigenvalue $\lampApprox$, and the computed values $\lampApprox$ are reported, together with the corresponding convergence rates. Note that as $p \to \infty$, the limiting eigenfunction is the distance-to-boundary function.}
    \label{tab:2_disc_self}
\end{table}

\begin{figure}
    \centering 
    \hspace{5em}
        \includegraphics[scale = 0.35]
        {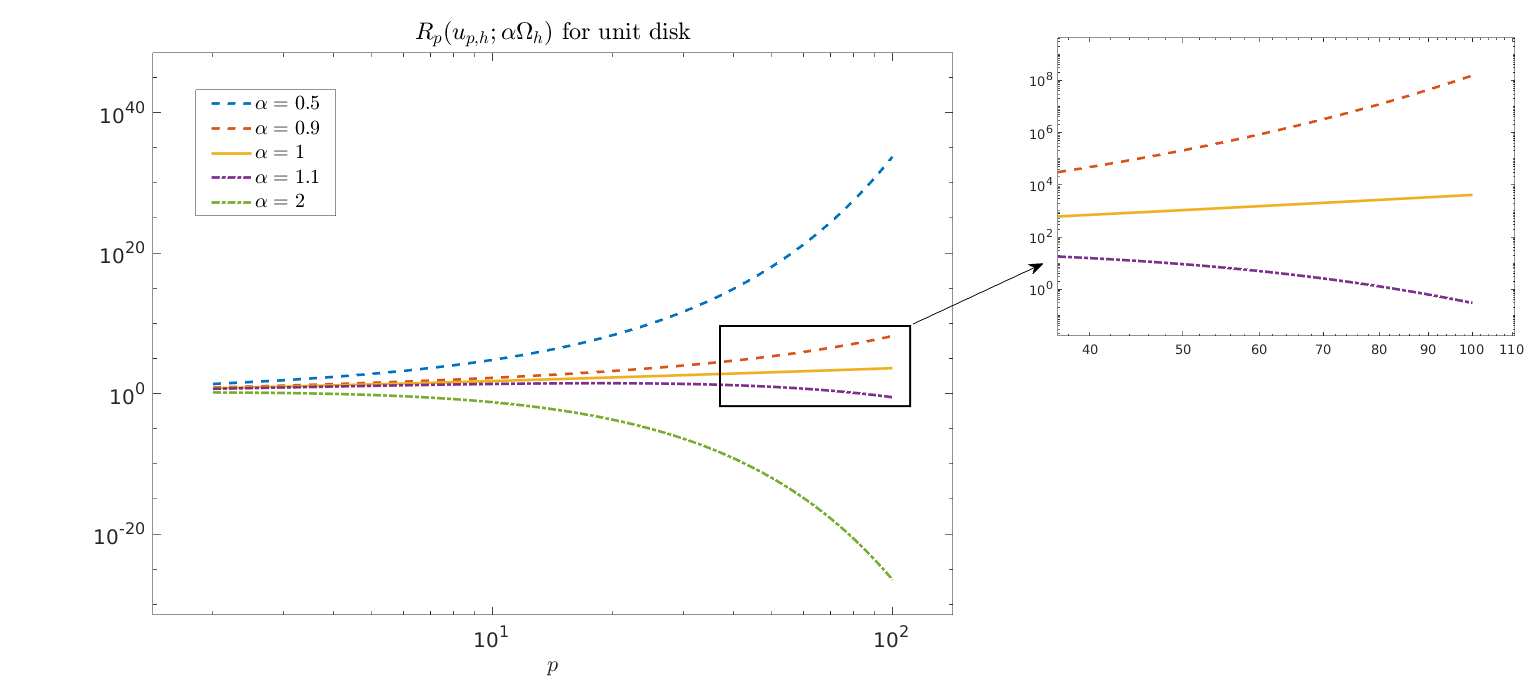}
    \caption{Eigenvalues 
on the unit disk for different scalings of the domain, $\alpha \in \{0.5, 0.9, 1, 1.1, 2\}$, plotted in log-log scale. 
The inset provides a closer view for larger $p$ to show the behavior for $\alpha \in  \{0.9, 1, 1.1\}$.  The disk mesh consists of $327{,}680$ cells under all scalings.}
    \label{fig:7_rescale_alph}
\end{figure}

\subsubsection{Square $\Omega = \{ (x, y) \in \mathbb{R}^2 : |x|+|y| < c \}$ }

In our computations we consider $c = \sqrt{2}$ for optimal asymptotic scaling. 
Note that eigenpairs may be recovered for different choices of  $c$. 
Unlike the disk, the square does not satisfy the $\mathcal{R} = \mathcal{M}$ condition for the limiting solution to coincide with a (normalized) distance-to-boundary function. It is proven by Juutinen et al.~\cite{Juutinenetal1999} that the limiting eigenfunction $\uInf$ as $p \to \infty$ is not in fact the normalized distance-to-boundary function. Consequently, as $p \to \infty$, $\uInf$ is not known analytically.  

Figure~\ref{fig:8_sqr} shows contour plots of the computed first eigenfunction on the square for a few different $p$ values. As $p$ increases, the eigenfunctions become nonsmooth along the diagonals in addition to at the origin. In contrast with the disk, there is a notable visual difference between $p = 10, 50,$ and $100$. 

Table~\ref{tab:3_self_sqr} reports the self-convergence results for the square. For small to moderate $p$, both eigenfunction and eigenvalue relative errors decrease steadily with mesh refinement, showing convergence rates of approximately second order for the eigenvalue and clearly above first order for the eigenfunction, similar to the previous unit disk example. 
As $p$ increases, the convergence rates on coarser meshes become less regular, and the eigenfunction $L^2$ error rates drop below $1$. This behavior is likely due to the nonsmooth features introduced by the square's corners, where the ridge set $\mathcal{R}$ includes the diagonals of the square while the maximal set $\mathcal{M}$ contains only the center point of the square.
However, with sufficient refinement, the errors stabilize. Our computations again show good agreement with the results reported by Horák~\cite{Horak2011}, who considered values up to $p=10$ on the square using a comparable number of degrees of freedom. Similarly to the disk, note that results on different square sizes will yield different eigenvalues according to the scaling relation~\eqref{eqn:scaling}.

\begin{table}[] 
    \centering
    \begin{subtable}{1\textwidth}   
    \centering
\begin{tabular}{cccccc}   
    \multicolumn{6}{l}{$p=3$} \\
     \toprule
     \# cells & $\upApprox$ $L^2$ error  & Rate & $\lampApprox$ relative error & Rate & $\lampApprox$  \\
     \midrule
     $256$ & $3.6137e-03$ & - & $6.6308e-03$ & - & $7.8968e+00$   \\
     $1024$ & $1.1695e-03$ & 1.63 & $1.6619e-03$ & 2.00 & $7.8578e+00$   \\
     $4096$ & $4.1277e-04$ & 1.50 & $4.0923e-04$ & 2.02 & $7.8480e+00$  \\  
      $ 16384$ & $1.7153e-04$ & 1.27  & $9.5140e-05$ & 2.10 &  $7.8455e+00$\\
     $65536$ &  $7.8053e-05$ & 1.14 &  $1.6540e-05$ &  2.52 & $7.8449e+00$\\
      \bottomrule
\end{tabular}
    \end{subtable}\\
    \vspace{0.5em}
    \begin{subtable}{1\textwidth}
     \centering 
\begin{tabular}{cccccc} 
    \multicolumn{6}{l}{$p=10$} \\
     \toprule
     \# cells & $\upApprox$ $L^2$ error  & Rate & $\lampApprox$ relative error & Rate & $\lampApprox$ \\
     \midrule
     $256$ & $3.1717e-03$ & - & $9.5620e-02$ & - & $3.8304e+01$ \\
     $1024$ & $1.5750e-03$ & 1.01 & $ 2.7096e-02$ & 1.82 &  $3.5908e+01$ \\
     $4096$ &  $7.1200e-04$ & 1.15 & $7.1603e-03$ & 1.92 & $3.5211e+01$  \\   
      $16384$ & $2.8564e-04$ & 1.32 & $1.7603e-03$ & 2.02 &  $3.5022e+01$\\  
    $65536$ &  $8.8702e-05$ &  1.69 &  $3.5712e-04$ &  2.30 & $3.4973e+01$\\
      \bottomrule
\end{tabular}
    \end{subtable}\\
    \vspace{0.5em}
    \begin{subtable}{1\textwidth}  
         \centering 
\begin{tabular}{cccccc}
    \multicolumn{6}{l}{$p=50$} \\
     \toprule
     \# cells & $\upApprox$ $L^2$ error  & Rate & $\lampApprox$ relative error & Rate & $\lampApprox$ \\
     \midrule
     $256$ & $1.8446e-02$ & - & $4.7445e+00$ & - & $1.7515e+03$ \\
     $1024$ & $1.1113e-02$ & 0.73 & $1.0510e+00$ & 2.17 &  $6.2534e+02$  \\
     $4096$ & $5.4332e-03$ & 1.03 & $2.9556e-01$ & 1.83 & $3.9502e+02$  \\  
     $16384$ & $2.2000e-03$ & 1.30 &  $8.2135e-02$ &  1.85 & $3.2994e+02$ \\
     $65536$ &  $7.9015e-04$ &  1.48 &  $1.8215e-02$ &  2.17 & $3.1045e+02$\\
      \bottomrule
\end{tabular}
    \end{subtable}\\
    \vspace{0.5em}
    \begin{subtable}{1\textwidth}  
     \centering
\begin{tabular}{cccccc} 
    \multicolumn{6}{l}{$p=100$} \\
     \toprule
     \# cells & $\upApprox$ $L^2$ error  & Rate & $\lampApprox$ relative error & Rate & $\lampApprox$ \\ 
     \midrule
     $256$ & $2.3859e-02$ & - & $6.1194e+02$ & - & $5.8252e+05$  \\
     $1024$ & $2.1900e-02$ & 0.12 & $ 3.2319e+01$ & 4.24 & $3.2627e+04$\\
     $ 4096$ & $ 1.4474e-02$ & 0.60 & $4.4159e+00$ & 2.87 & $5.1471e+03$\\   
       $16384$ &  $8.0884e-03$ &  0.84 &  $8.4617e-01$ &  2.38 & $1.7545e+03$\\
      $65536$ &  $5.0951e-03$ &  0.67 &  $2.4957e-01$ &  1.76 & $1.1840e+03$\\
      \bottomrule
\end{tabular}
    \end{subtable}\\
    \caption{Self-convergence study on the square for several values of $p$, with reference taken from the solution on a mesh of $262{,}144$ cells. For each refinement level, the $L^2$ error of the eigenfunction $\upApprox$, the relative error of the first eigenvalue $\lampApprox$, and the computed values of $\lampApprox$ are reported, together with the corresponding convergence rates. Note that as $p \to \infty$, the limiting eigenfunction is not known. Here, we use a reduced continuation step of $\delta p = 0.25$ in Algorithm~\ref{alg:newton_inverse_power}, since the ridge set includes the two diagonals of the square.
    }
    \label{tab:3_self_sqr}
\end{table}

\begin{figure} 
    \centering
    \begin{subfigure}{0.18\textwidth}
        \centering
         \includegraphics[scale = 0.11]{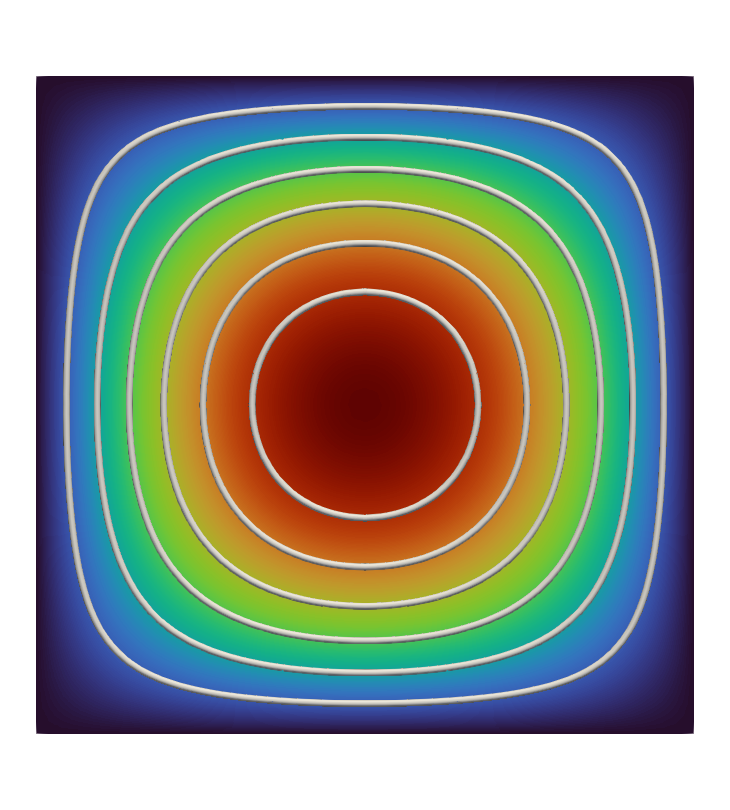}
           \subcaption{  $p = 2$}
    \end{subfigure}
     \begin{subfigure}{0.18\textwidth}
        \centering
         \includegraphics[scale = 0.11]{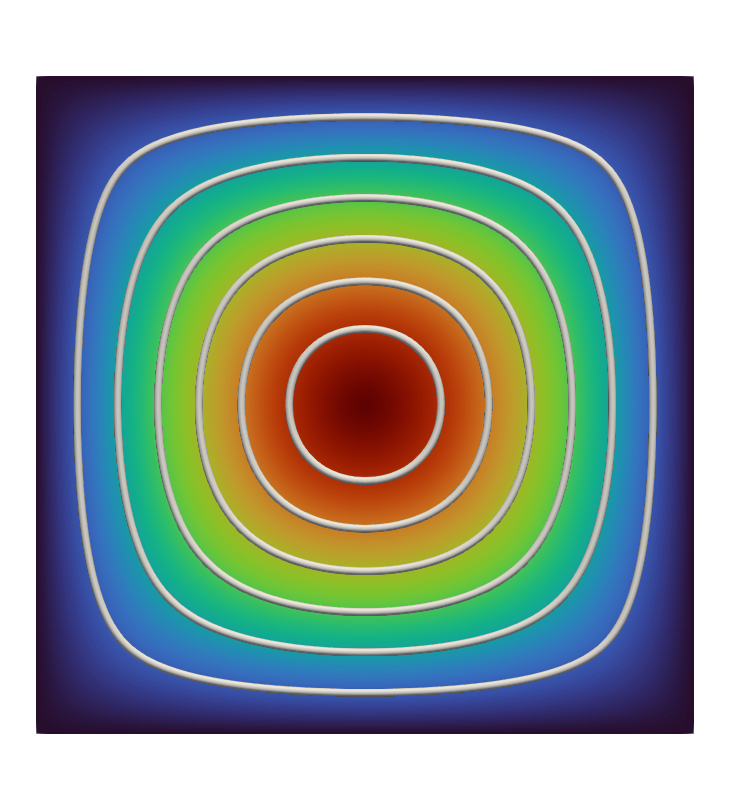}
           \subcaption{  $p = 3$}
    \end{subfigure}
        \centering
    \begin{subfigure}{0.18\textwidth}
        \centering
         \includegraphics[scale = 0.11]{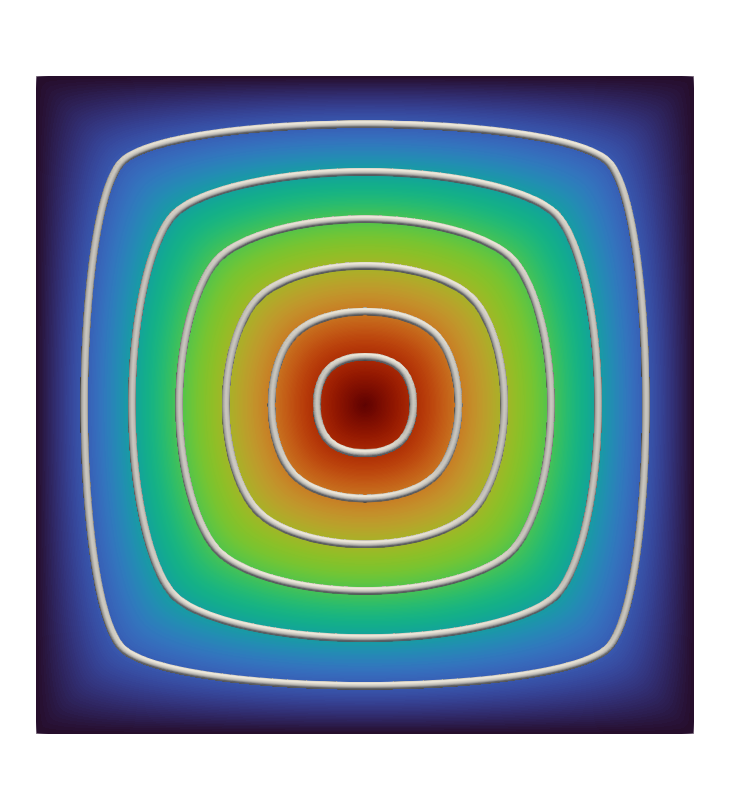}
           \subcaption{  $p = 10$}
    \end{subfigure}
     \begin{subfigure}{0.18\textwidth}
        \centering
         \includegraphics[scale = 0.11]{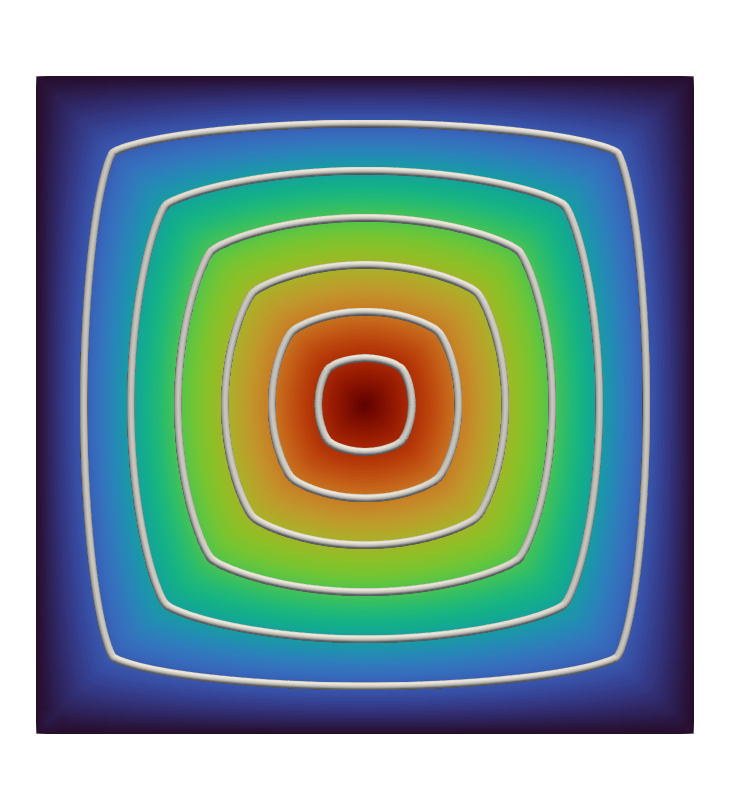}
           \subcaption{  $p = 50$}
    \end{subfigure}
     \begin{subfigure}{0.18\textwidth}
     \centering
         \includegraphics[scale = 0.11]{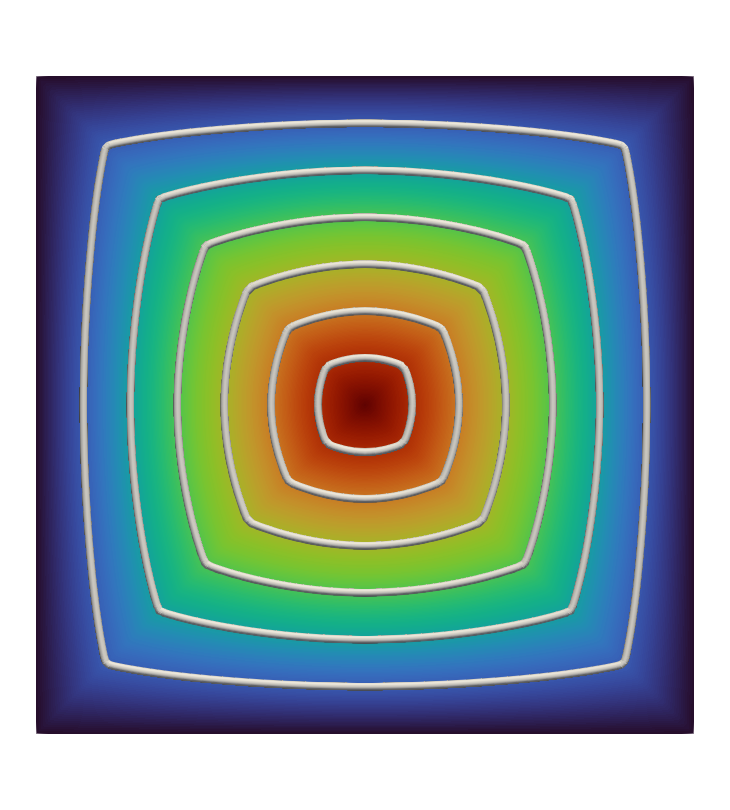}
          \subcaption{  $p = 100$}
    \end{subfigure}\\
    \caption{Contour plots of the first eigenfunction $\upApprox$ on the square. Plots (a)–(e) show results for different values of $p$. All computations use a mesh with $262{,}144$ cells. Contour levels are kept consistent across plots. 
    Note that the limiting $p \to \infty$ eigenfunction is \emph{not} the distance-to-boundary function on the square (see~\cite{Juutinenetal1999}).
    }
    \label{fig:8_sqr}
\end{figure}

Figure~\ref{fig:4_eig_many} shows the computed eigenvalues $\lampApprox$ and corresponding roots $\lampApprox^{1/p}$. We observe that these eigenvalue approximations on the square do not resemble those of any other example considered.
This difference is likely due to the square's geometric characteristics differing in nature from those of the disk, as suggested in the previous paragraph. 

\subsection{Surface domains}
\label{subsec:surf}

In this subsection, we consider $p$-Laplace eigenvalue problems on curved surfaces  of codimension $1$ lying in $\mathbb{R}^3$. 
We first study two simple surfaces, the hemisphere and the half torus, for which $\mathcal{R} = \mathcal{M}$, and therefore, from Section~\ref{subsec:asymptotic}, we know the first eigenfunction should converge toward the (geodesic) distance-to-boundary function as $p \to \infty$, a result detailed in Appendix~\ref{app:surface-limits}. 
To illustrate the versatility of the method, we also consider several ``fun" examples by reading in surface meshes including a duck, a hand, a mug, and a pig. For these domains, the symmetry assumption is generally not satisfied, and the location of the ridge set is less predictable, so the limiting eigenfunction is unknown. We use Algorithm~\ref{alg:newton_inverse_power} for the hemisphere and half torus, where the maximum geodesic distance-to-boundary is known, and Algorithm~\ref{alg:newton_inverse_power_rescale} for the ``fun'' examples, where this value is not known exactly.

\subsubsection{Hemisphere $\Omega = \{ (x, y, z) \in \mathbb{R}^3 : \sqrt{x^2+y^2+z^2} = R, \; z > 0 \}$ }

For the hemisphere, the condition $\mathcal{R} = \mathcal{M}$ discussed in Section~\ref{subsec:asymptotic} is satisfied. 
Therefore, as $p \to \infty$, the first eigenfunction approaches
\begin{equation*}
        \uInf(\mathbf{x}) = \frac{\distx}{\|\textrm{dist}(\cdot,\partial \Omega)\|_\infty} = \frac{2}{\pi} \left( \frac{\pi}{2}-\arccos\left(\frac{z}{R}\right) \right),
\end{equation*}
where $\distx$ denotes the geodesic distance along the hemisphere from the point $\mathbf{x}=(x,y,z)$ to the boundary at the equator. This expression follows directly from the spherical law of cosines.

In our computations, we scale the domain so that the maximal geodesic distance from the boundary is $1$. For a hemisphere of radius $R$, the maximal geodesic distance occurs at the pole $(0,0,R)$ and is given by $R \pi/2$, so the scaling factor $\alpha = 2/\pi$ ensures $\max_{\bx \in \Omega} \distx = 1$. As with the Euclidean examples, results may be recovered for other scalings. 

Table~\ref{tab:4_self_hemi} reports the self-convergence results for various $p$, showing $L^2$ errors of the eigenfunction, relative errors of the first eigenvalue, and computed values of $\lampApprox$. For small to moderate values of $p$ ($p=3,10$), the observed convergence rates are close to $2$ for the eigenvalue relative errors and mostly well above $1$ for the eigenfunction $L^2$ errors. For larger values ($p=50,100$), the convergence rates on coarser meshes are irregular, but they appear to become more consistent on finer meshes. Overall, the convergence behavior is similar to that of the Euclidean disk considered above, and better than that of the square. 

Figure~\ref{fig:9_hemi} shows the contour plots of $\upApprox$ for selected $p$ values, showing how we approach the limiting (geodesic) distance-to-boundary function as $p$ increases. Note this is qualitatively the same behavior as observed for the disk, with visually comparable results for $p=10$ and $p = \infty$.  

\begin{table}[]
    \centering
    \begin{subtable}{1\textwidth}   
    \centering
\begin{tabular}{cccccc}
    \multicolumn{6}{l}{$p=3$} \\
     \toprule
     \# cells & $\upApprox$ $L^2$ error  & Rate & $\lampApprox$ relative error & Rate & $\lampApprox$ \\
     \midrule
     $320$ & $2.0717e-03$ & - & $9.4180e-03$ & - &  $8.4999e+00$\\
     $1280$ & $6.7071e-04$ & 1.63 & $2.3446e-04$ & 2.01 &   $8.4404e+00$\\
     $5120$ & $2.4599e-04$ & 1.45 & $5.7906e-04$ & 2.02 & $8.4255e+00$ \\
     $20480$ & $1.0941e-04$ & 1.17 & $1.3781e-04$ & 2.07 & $8.4218e+00$ \\  
     $81920$ & $5.5967e-05$ & 0.97 & $2.7504e-05$ & 2.32  & $8.4208e+00$ \\
      \bottomrule
\end{tabular}
    \end{subtable}\\
    \vspace{0.5em}
    \begin{subtable}{1\textwidth}
     \centering
\begin{tabular}{cccccc} 
    \multicolumn{6}{l}{$p=10$} \\
     \toprule
     \# cells & $\upApprox$ $L^2$ error  & Rate & $\lampApprox$ relative error & Rate & $\lampApprox$ \\
     \midrule
     $320$ & $2.0009e-03$ & - & $5.3068e-02$ & - & $5.4842e+01$ \\
     $1280$ & $9.8990e-04$ & 1.02 & $1.3721e-02$ & 1.95 &$5.2793e+01$  \\
     $5120$ & $4.2857e-04$ & 1.21 & $3.4622e-03$ & 1.99 & $5.2258e+01$ \\
     $20480$ & $1.5659e-04$ & 1.45 & $8.3294e-04$ & 2.06 & $5.2122e+01$  \\ 
      $81920$ & $4.6054e-05$ & 1.77 & $1.6738e-04$ & 2.32 & $5.2087e+01$ \\
     \bottomrule    
\end{tabular}
    \end{subtable}\\
    \vspace{0.5em}
    \begin{subtable}{1\textwidth}  
         \centering
\begin{tabular}{cccccc}
    \multicolumn{6}{l}{$p=50$} \\
     \toprule
     \# cells & $\upApprox$ $L^2$ error  & Rate & $\lampApprox$ relative error & Rate & $\lampApprox$ \\
     \midrule
     $320$ & $6.0036e-03$ & - & $1.1643e+00$ & - &   $2.0099e+03$\\
     $1280$ & $3.3519e-03$ & 0.84 & $3.5539e-01$ & 1.70 & $1.2587e+03$  \\
     $5120$ & $1.6074e-03$ & 1.06 & $1.1100e-01$ & 1.68 &  $1.0318e+03$ \\
     $20480$ & $6.4206e-04$ & 1.32 & $3.1691e-02$ & 1.81 &  $9.5812e+02$\\  
    $81920$ &  $2.0183e-04$ &  1.67 &    $7.1086e-03$ &    2.16 &    $9.3529e+02$\\
      \bottomrule
\end{tabular}
    \end{subtable}\\
    \vspace{0.5em}
    \begin{subtable}{1\textwidth}  
     \centering
\begin{tabular}{cccccc}
    \multicolumn{6}{l}{$p=100$} \\
     \toprule
     \# cells & $\upApprox$ $L^2$ error  & Rate & $\lampApprox$ relative error & Rate & $\lampApprox$ \\ 
     \midrule
     $320$ & $6.9977e-03$ & - & $8.3074e+00$ & - &  $3.2896e+04$ \\
     $1280$ & $4.0540e-03$ & 0.79 & $1.7604e+00$ & 2.24 & $9.7565e+03$  \\
     $5120$ & $2.0038e-03$ &  1.02 & $5.1570e-01$ & 1.77 & $5.3571e+03$ \\
     $20480$ & $8.1840e-04$ & 1.29 & $1.5555e-01$ & 1.73 & $4.0842e+03$  \\
      $81920$ &    $2.6150e-04$ & 1.65  &    $ 3.7896e-02$ &    2.04 &   $3.6684e+03$  \\
      \bottomrule
\end{tabular}
    \end{subtable}\\
    \caption{
        Self-convergence study on the hemisphere for several values of $p$, with reference taken from the solution on a mesh of $327{,}680$ cells. For each refinement level, the $L^2$ error of the eigenfunction $\upApprox$, the relative error of the first eigenvalue $\lampApprox$, and the computed values of $\lampApprox$ are reported, together with the corresponding convergence rates. Note that as $p \to \infty$, the limiting eigenfunction is the (geodesic) distance-to-boundary function. }
    \label{tab:4_self_hemi}
\end{table}

\begin{figure} 
    \centering
    \begin{subfigure}{0.18\textwidth}
        \centering
         \includegraphics[scale = 0.075]{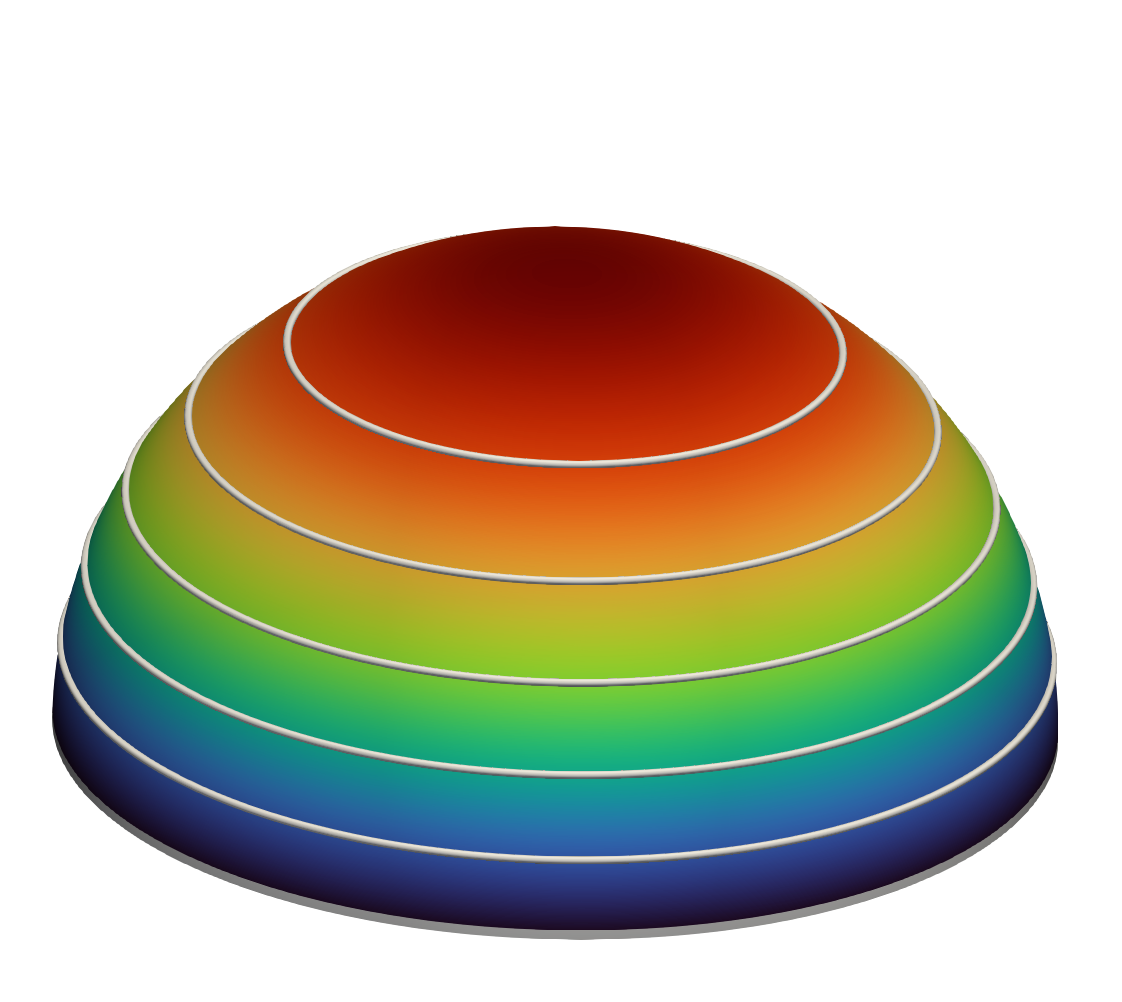}
           \subcaption{  $p = 2$}
    \end{subfigure}
     \begin{subfigure}{0.18\textwidth}
        \centering
         \includegraphics[scale = 0.075]{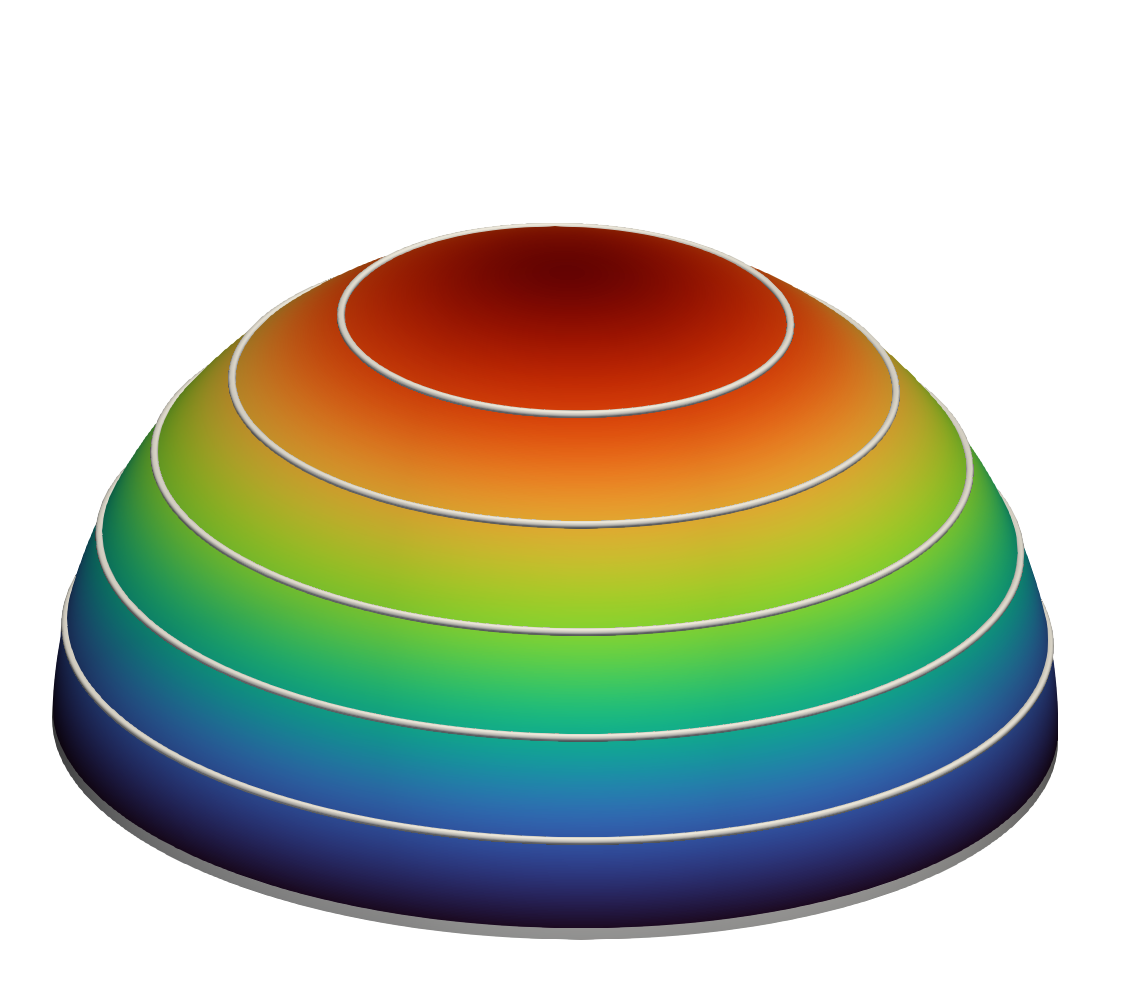}
           \subcaption{  $p = 2.5$}
    \end{subfigure}
     \begin{subfigure}{0.18\textwidth}
     \centering
         \includegraphics[scale = 0.075]{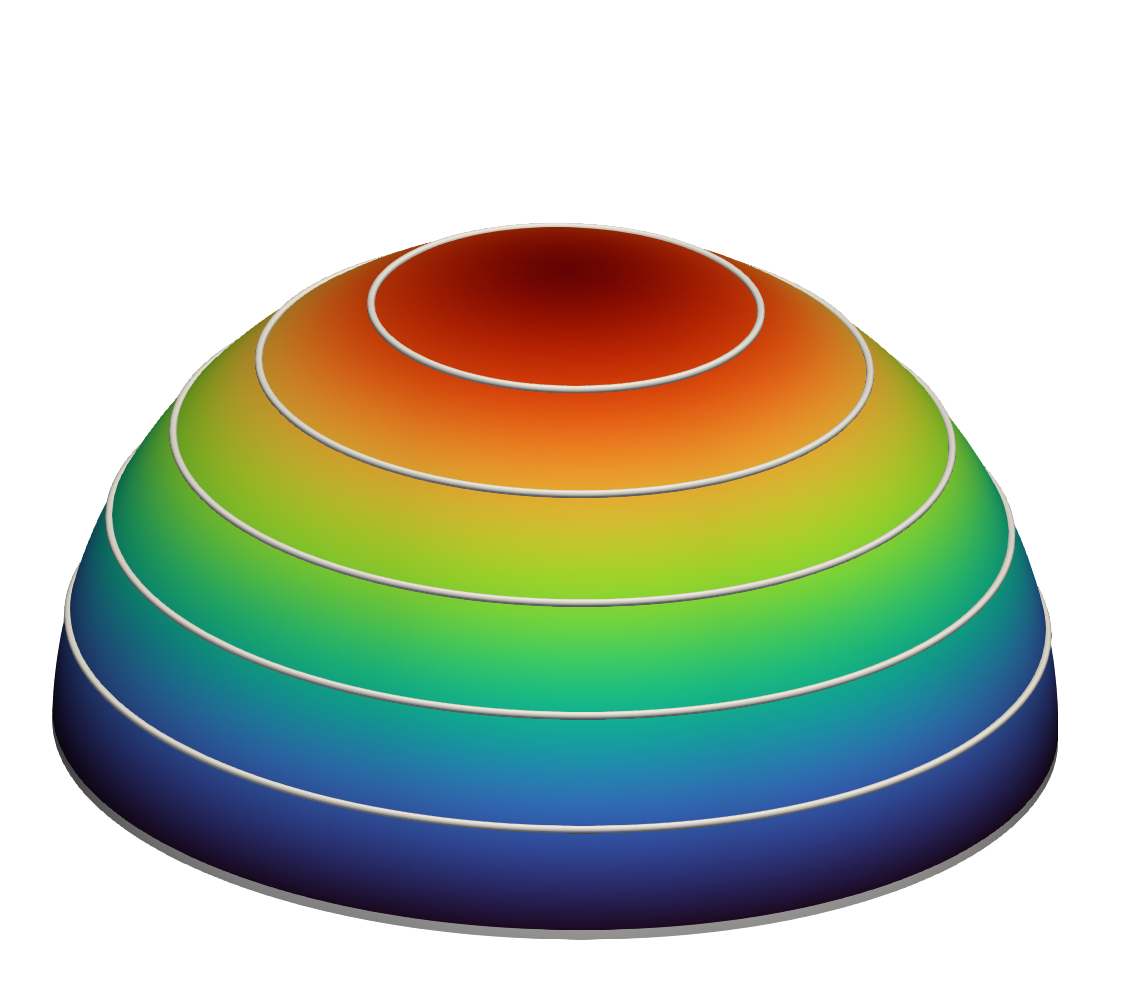}
          \subcaption{  $p = 3$}
    \end{subfigure}
     \begin{subfigure}{0.18\textwidth}
        \centering
         \includegraphics[scale = 0.075]{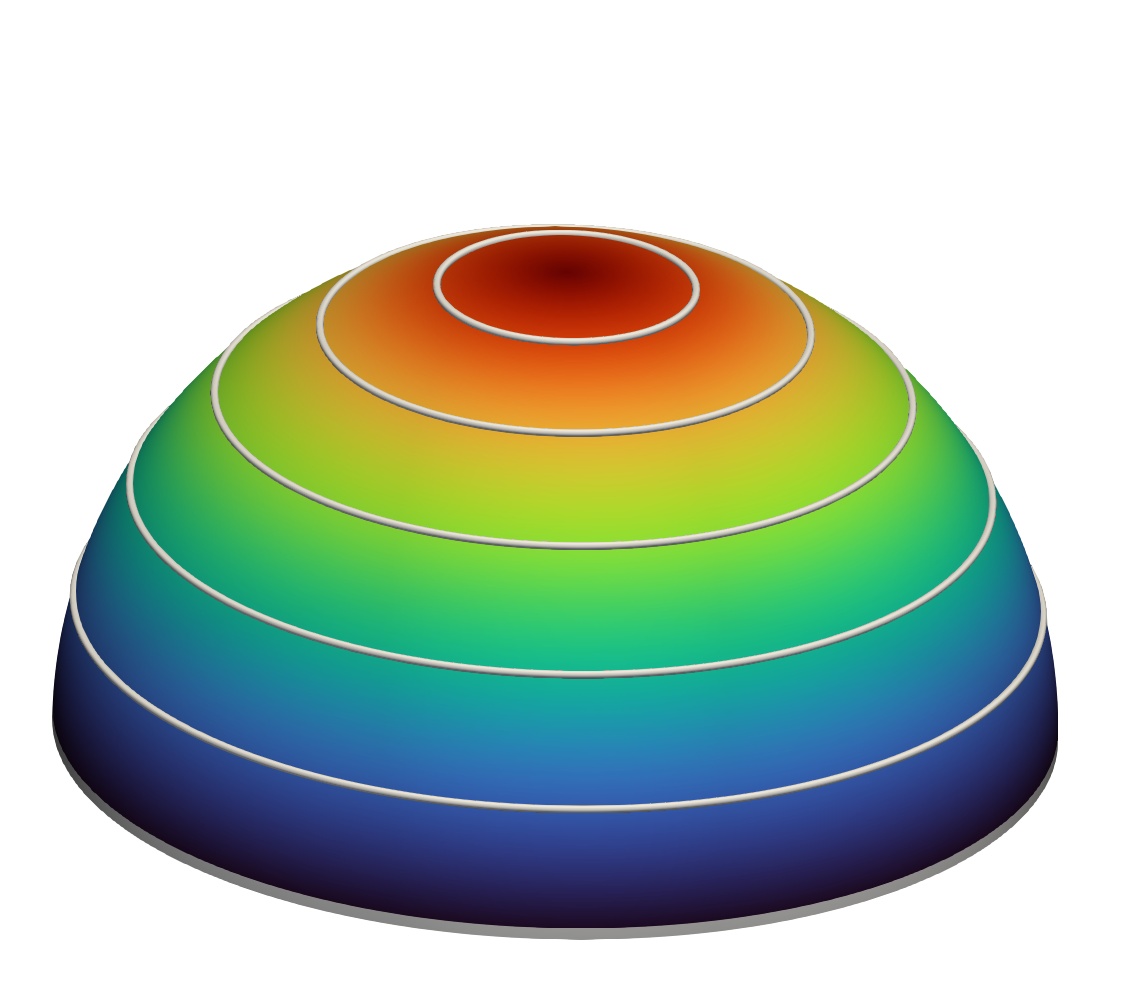}
           \subcaption{  $p = 10$}
    \end{subfigure}
     \begin{subfigure}{0.18\textwidth}
     \centering
         \includegraphics[scale = 0.075]{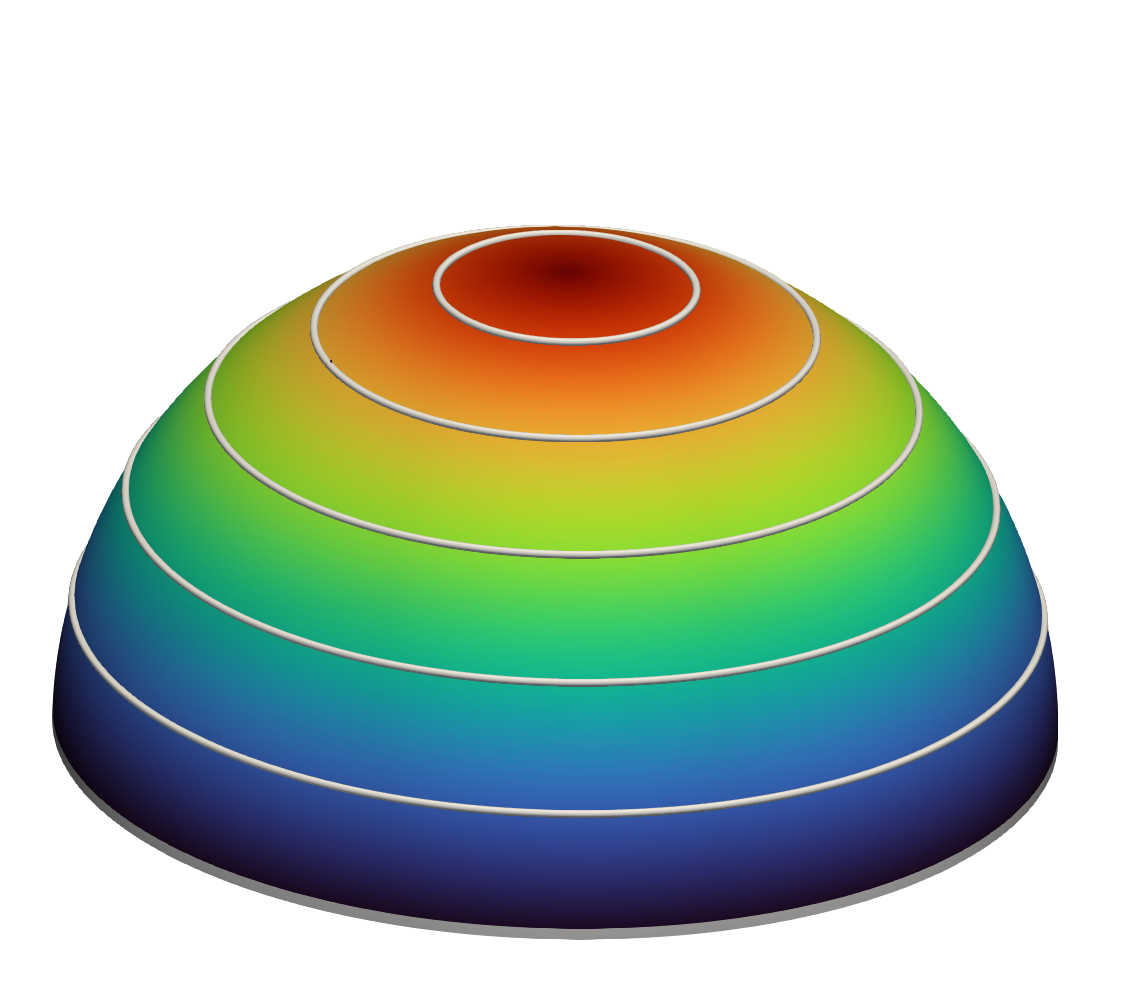}
          \subcaption{$p \to \infty$}
    \end{subfigure}\\
    \caption{Contour plots of $\upApprox$ on the hemisphere.    
    Plots (a)–(d) correspond to different values of $p$, and plot (e) shows the limiting solution as $p \to \infty$. All computations use a mesh with $327{,}680$ cells. Isolines are shown at the same levels across plots.}
    \label{fig:9_hemi}
\end{figure}

Figure~\ref{fig:4_eig_many} shows the computed eigenvalues and their corresponding $p$th roots. The eigenvalues for the hemisphere appear to have similar growth to those for the disk, although lying slightly below. 

\subsubsection{Half torus 
$ \Omega = \{ (x,y,z) \in \mathbb{R}^3 : (2 - \sqrt{x^2 + y^2})^2 + z^2 = 1,
\; z > 0 \}$ }

The domain $\Omega$ corresponds to the upper half of a torus obtained by a transversal cut along the plane $z=0$, analogous to slicing a bagel horizontally. Dirichlet boundary conditions are imposed on $\partial\Omega$, which consists of the two closed curves where $z=0$. These curves separate the torus into symmetric upper and lower halves with respect to the $z=0$ plane.

The contour plots in Figure~\ref{fig:10_tor} illustrate how the first eigenfunction $\upApprox$ on the half torus changes with increasing $p$. As for the hemisphere, the eigenfunction gradually approaches the limiting distance-to-boundary solution, with the contours becoming increasingly aligned with the geodesic distance from the boundary at $z=0$. The qualitative behavior is very similar to that observed in the Euclidean examples and on the hemisphere. For conciseness, we omit a detailed self-convergence study here, as the convergence trends are consistent with those reported previously.

\begin{figure} 
    \centering
        \begin{subfigure}{0.18\textwidth}
        \centering
         \includegraphics[scale = 0.065]{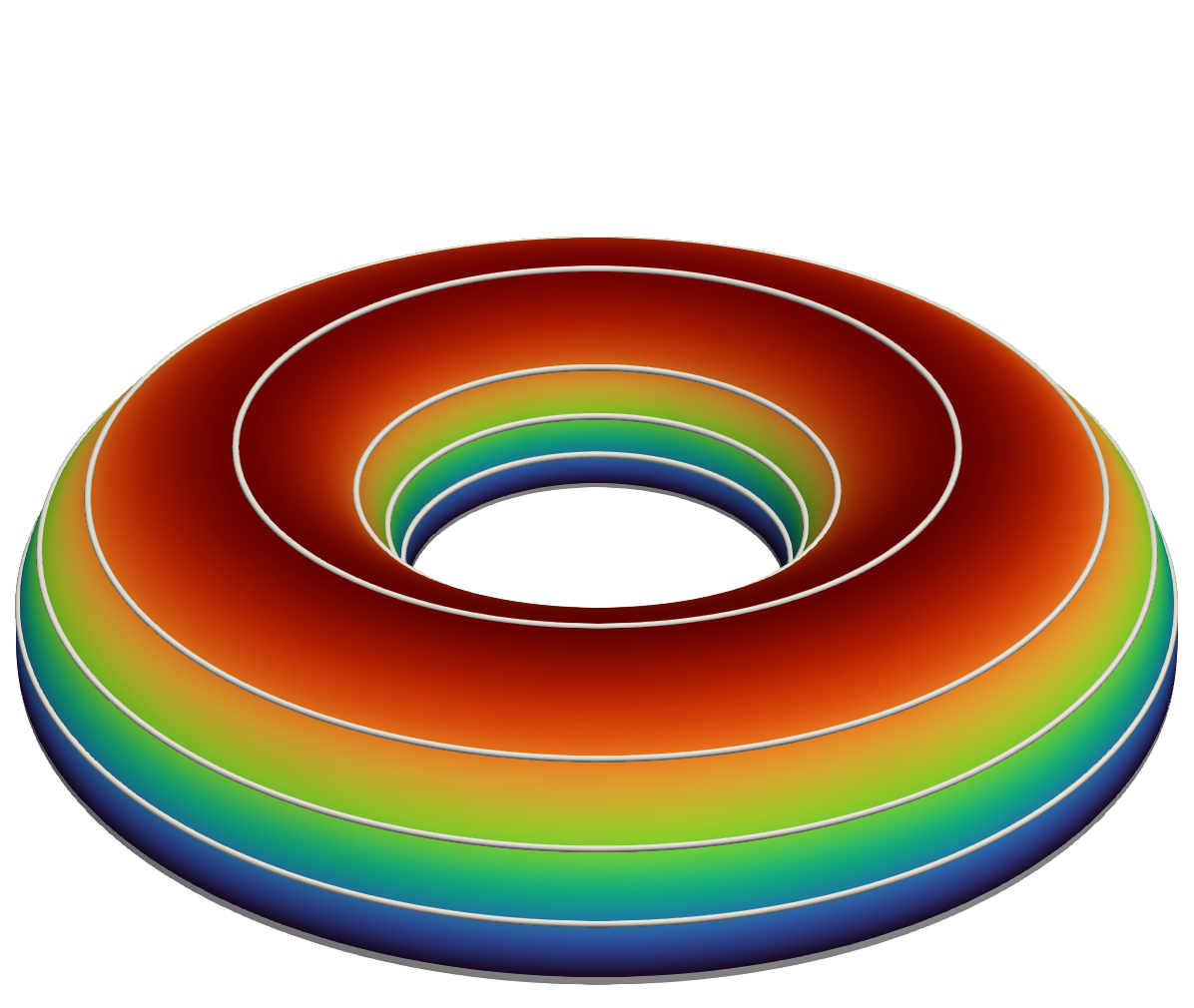}
           \subcaption{  $p = 2$}
    \end{subfigure}
     \begin{subfigure}{0.18\textwidth}
        \centering
         \includegraphics[scale = 0.065]{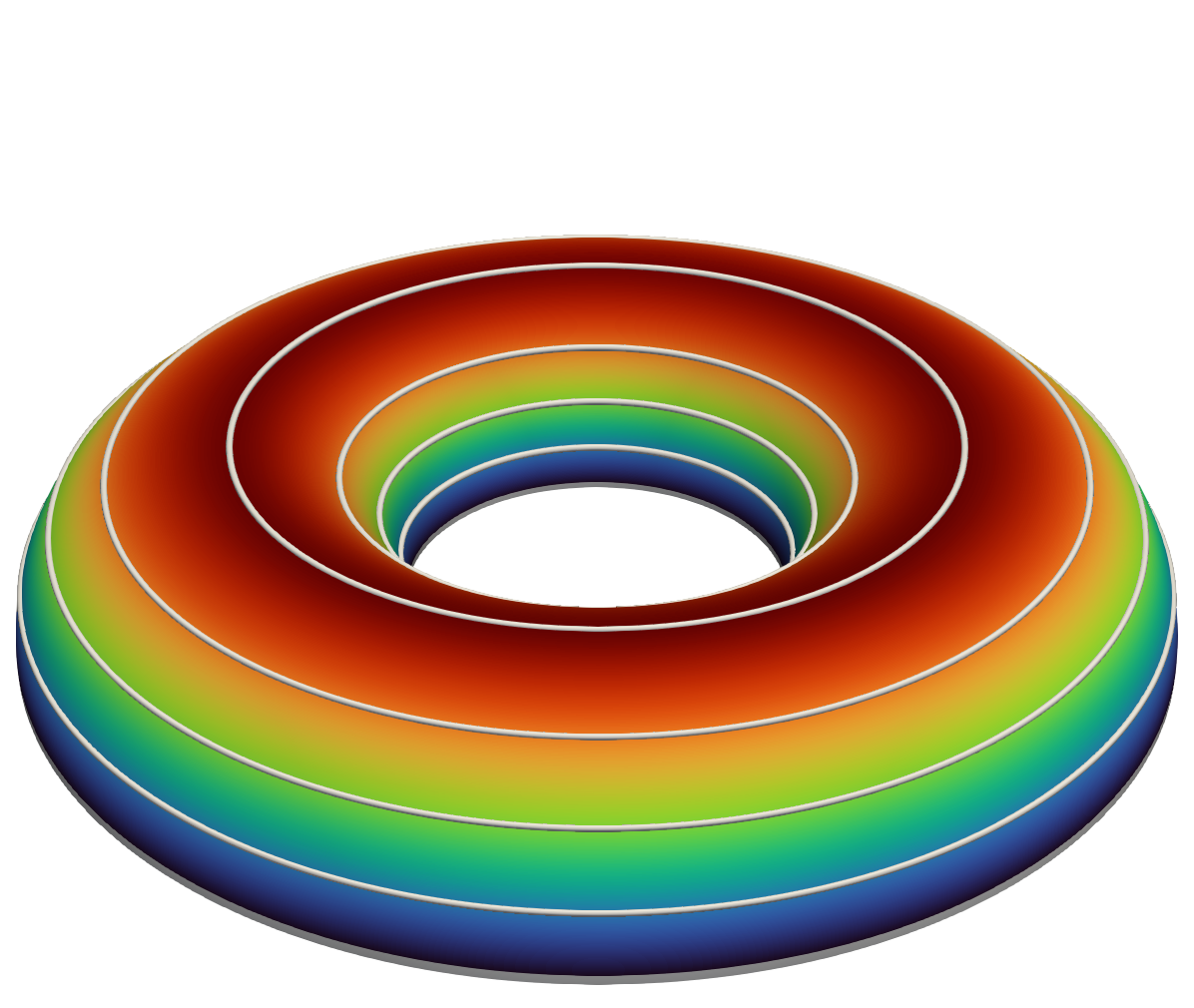}
           \subcaption{  $p = 2.5$}
    \end{subfigure}
     \begin{subfigure}{0.18\textwidth}
     \centering
         \includegraphics[scale = 0.065]{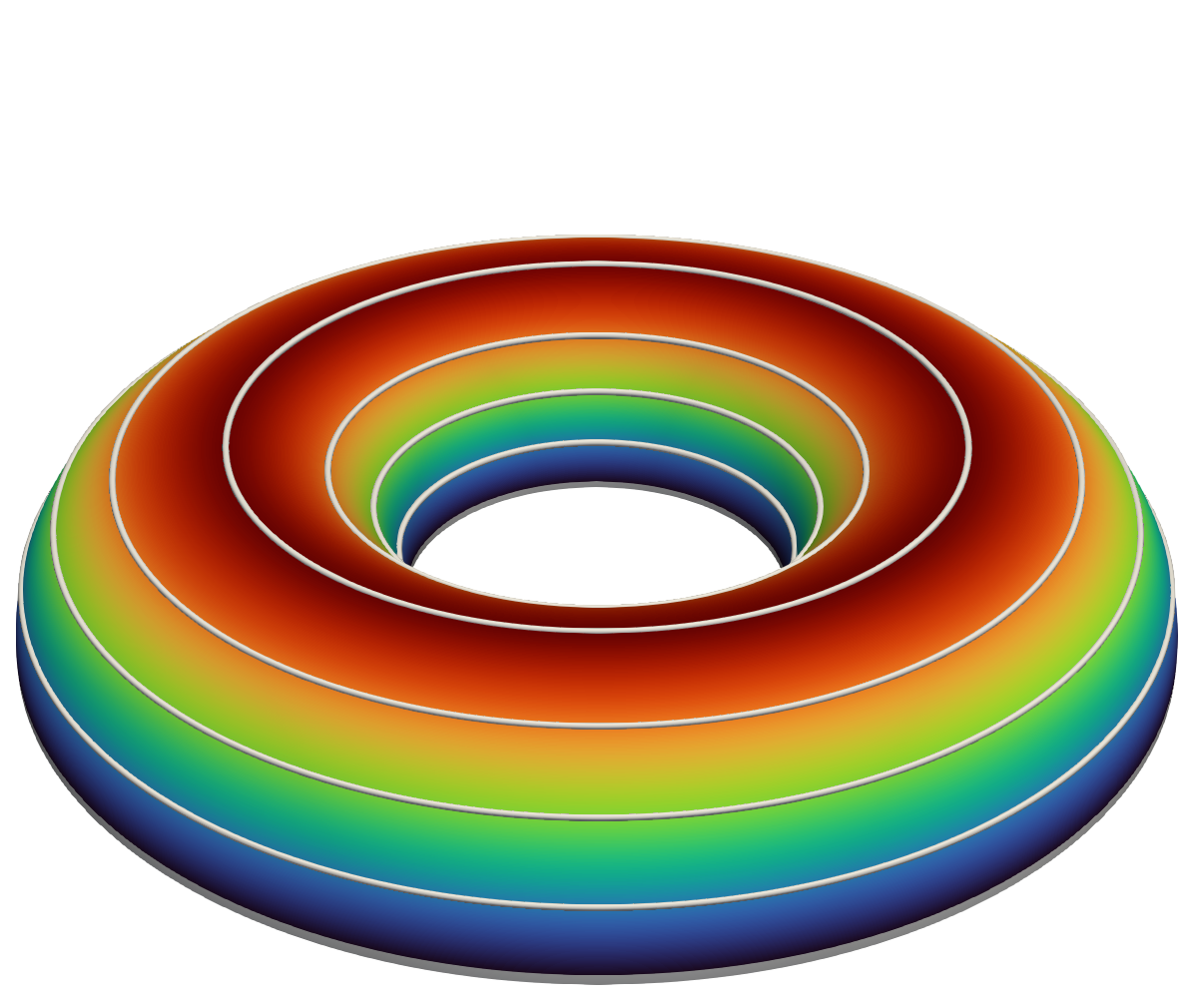}
          \subcaption{  $p = 3$}
    \end{subfigure}
     \begin{subfigure}{0.18\textwidth}
        \centering
         \includegraphics[scale = 0.065]{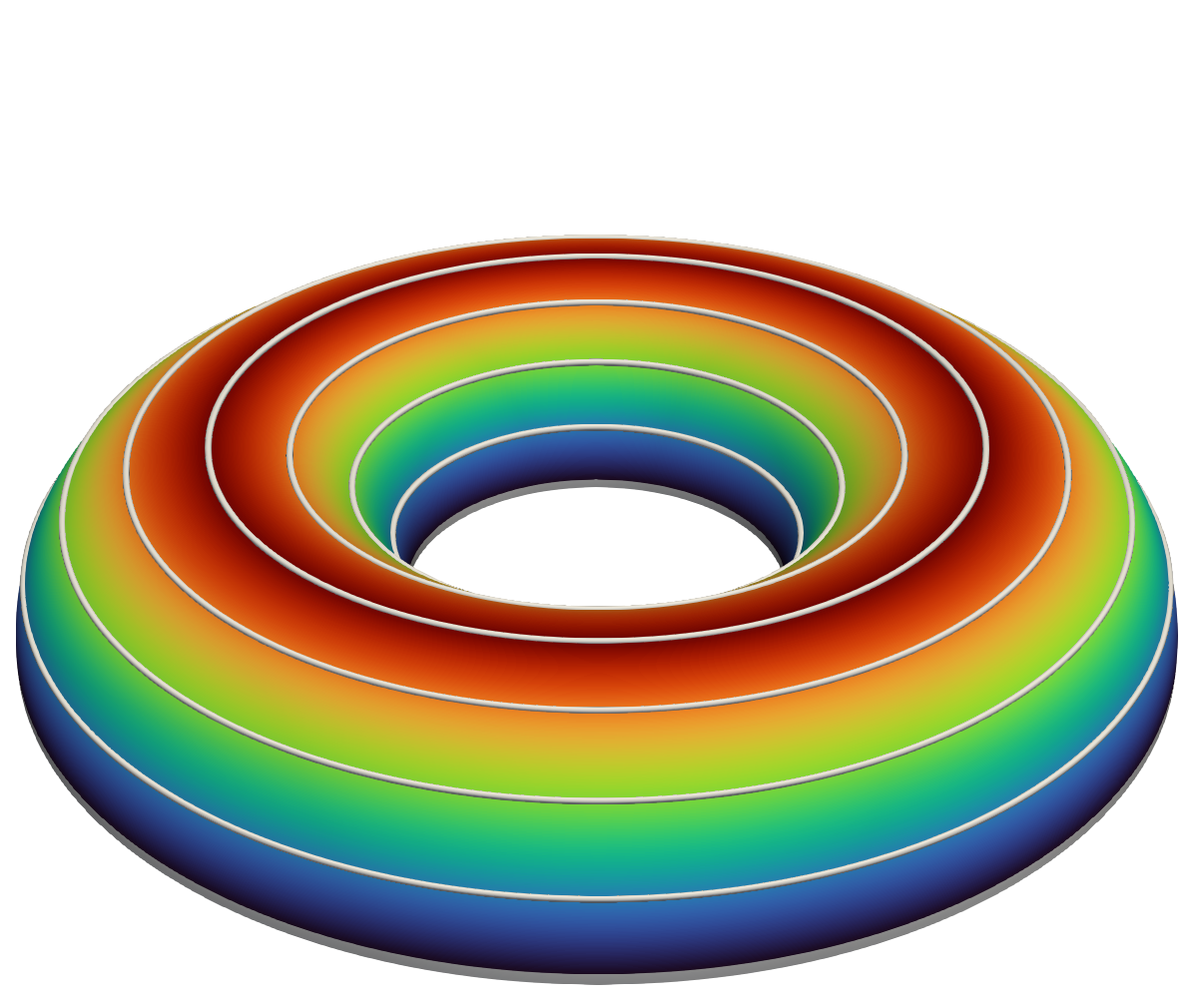}
           \subcaption{  $p = 10$}
    \end{subfigure}
     \begin{subfigure}{0.18\textwidth}
     \centering
         \includegraphics[scale = 0.065]{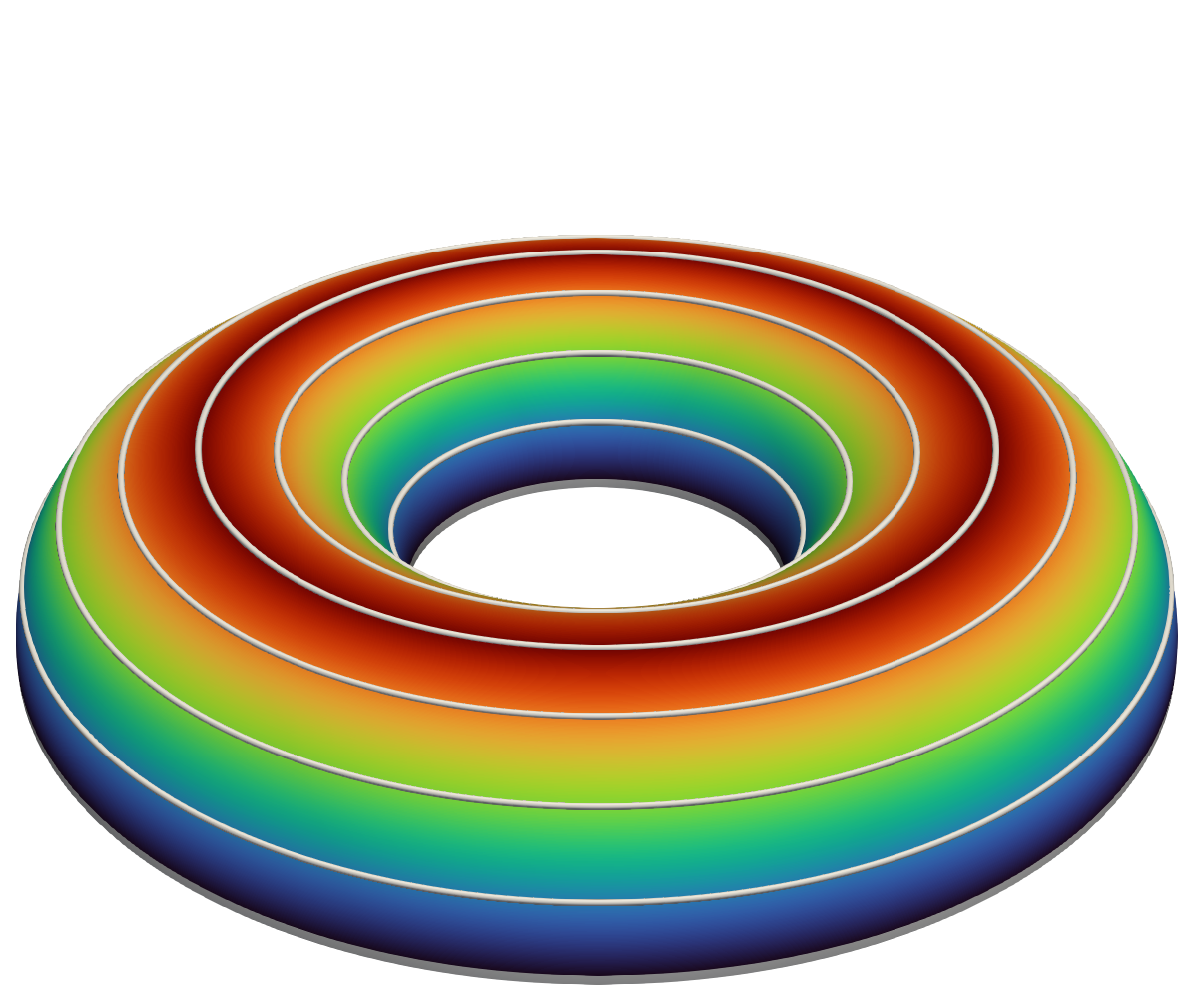}
          \subcaption{$p \to \infty$}
    \end{subfigure}\\
    \caption{ Contour plots of $\upApprox$ on the upper half torus. 
    Plots (a)–(d) correspond to different values of $p$, and plot (e) shows the limiting solution as $p \to \infty$. All computations use a mesh with $784{,}896$ cells. Isolines are shown at the same levels across plots.}
    \label{fig:10_tor}
\end{figure}

Figure~\ref{fig:4_eig_many} shows that the computed eigenvalues for the half torus are very close to the corresponding $1$D eigenvalues, which grow linearly in $p$ for large $p$. 
Intuitively, points furthest from the boundary lie along a closed curve, and the distance to the boundary changes only by moving inward or outward from this curve. This two-direction structure mirrors the $1$D setting, leading to eigenvalues that closely resemble those of a semicircle.

\subsubsection{Geometrically complex surfaces}

In Figure~\ref{fig:11_fun_eig}, we plot the computed eigenvalues $\lampApprox$ along with their $p$th roots for the following examples: (i) duck mesh with $\partial \Omega$ taken along the central axis which divides the surface into two symmetric halves, (ii) hand mesh with $\partial \Omega$ corresponding to the natural wrist boundary, (iii) mug mesh with $\partial \Omega$ defined as a single point (or small $\epsilon$-ball) on the outer surface, and (iv) pig mesh with $\partial \Omega$ defined as a single point (or small $\epsilon$-ball) on the left ear. From a PDE analysis point of view, prescribing Dirichlet conditions on isolated points is degenerate. In this case, one can consider $\partial \Omega$ to be the boundary of an arbitrarily small geodesic ball centered around the single point.

For the hand and mug geometries, the original mesh resolutions exhibited an uptick in $\lampApprox$ around $p \sim 40$, which flattened out after one level of mesh refinement in \texttt{deal.II} and aligned with the trend observed for lower 
$p$. Insufficient mesh resolution can lead to an artificial increase in the nonlinear Rayleigh quotient~\eqref{eqn:Rayleigh}, as seen in prior examples, in particular the square (see Table~\ref{tab:3_self_sqr}). The refined results are therefore presented in Figure~\ref{fig:11_fun_eig}, as they provide a more accurate representation of the true eigenvalue behavior for $p > 40$. 

Visualizations of these surfaces are shown in Figure~\ref{fig:12_fun}. Since the maximum geodesic distance to the boundary is not known exactly for these domains, the spectrum is reported under approximately optimal scaling by taking $\alpha = d_m^{-1}$, where $d_m$ is an approximate value of $\max_{\bx \in \Omega} \distx$, as described in Section~\ref{subsec:largep}. Consequently, the limiting value $\LamInf$ is not known exactly, although it is expected to be close to $1$ for all examples. In Figure~\ref{fig:11_fun_eig} on the right we observe non-monotonicity in $\lampApprox^{1/p}$, particularly pronounced for the mug and pig examples, which both have Dirichlet regions defined at points (or small $\epsilon$-balls). 
This behavior aligns with the rapid changes in the eigenfunctions for small $p$ observed in Figure~\ref{fig:12_fun} for these two cases.

\begin{figure}
    \centering
     \includegraphics[scale = 0.52]{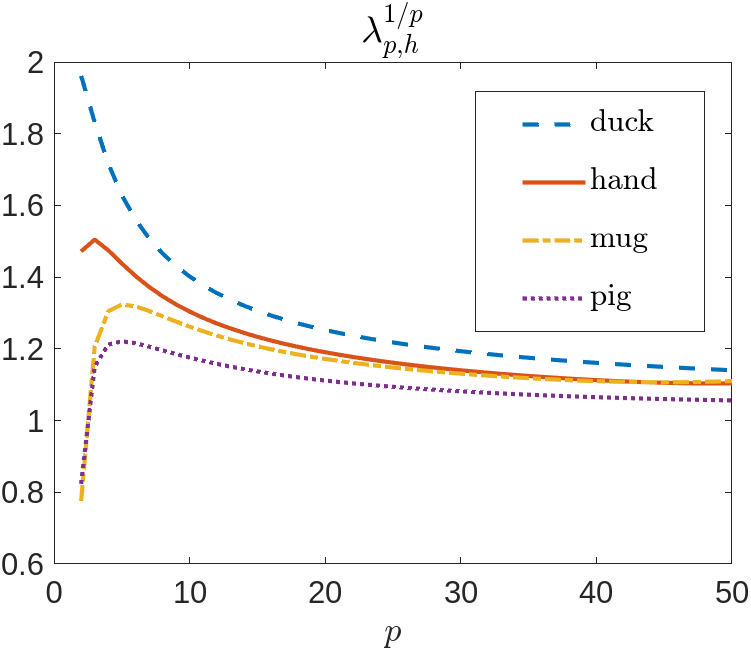}
    \hspace{1.5em}
     \includegraphics[scale = 0.52]{figs/fig11_right.png}
    \caption{Left: Eigenvalue approximations on some approximately optimally scaled domains via geodesic distance approximation for $ 2 \leq p \leq 50$ shown in log-log scale. Right: Corresponding $p$th roots of eigenvalue approximations. Note we do not know $\LamInf$ exactly due to geodesic distance being only approximate. The meshes consist of the following numbers of quadrilateral cells: duck ($32{,}064$), hand ($193{,}836$), mug ($360{,}888$), and pig ($50{,}454$). }
    \label{fig:11_fun_eig}
\end{figure}

Figure~\ref{fig:12_fun} shows the eigenfunctions $\upApprox$ on several surface meshes for a variety of $p$ values. The rows correspond to surfaces of increasing geometric or topological complexity: the duck mesh is similar to a torus, the hand mesh has a more ``natural" boundary at the wrist, and the mug and pig meshes impose Dirichlet conditions at a single point (or a small $\epsilon$-ball). For the point-constrained examples (mug and pig), the eigenfunctions exhibit significant variation at low $p$ so there is an observed sensitivity to the localized Dirichlet condition. As $p$ increases, the isolines in all examples become more evenly spaced and the solutions develop the characteristic nonsmooth features expected in the limiting $p \to \infty$ solution, even though the exact limiting solutions are not known. Figure~\ref{fig:13_ineq_check} verifies the monotonicity property of the first eigenvalue stated in inequality~\eqref{eqn:mono}, showing that the quantity $p \lampApprox^{1/p}$ is strictly increasing with $p$. The slope between consecutive discrete $p$ values is positive for all representative surfaces considered, confirming adherence to the inequality across our examples. Results are shown only for a few representative surface domains for clarity, however all examples under our consideration obeyed the inequality. 

\begin{figure}
    \centering
    \begin{subfigure}{0.18\textwidth}
        \centering
         \includegraphics[scale = 0.07]{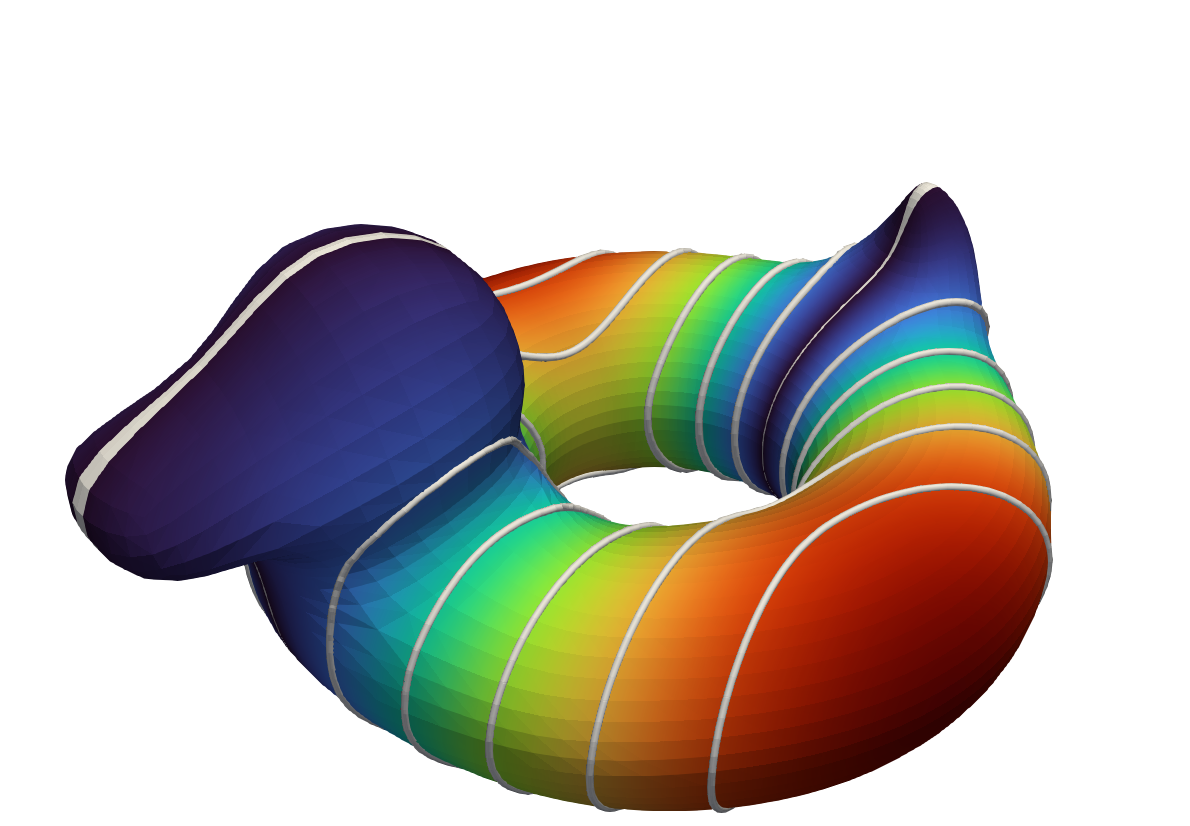}
           \subcaption{$p = 2$}
    \end{subfigure}
     \begin{subfigure}{0.18\textwidth}
        \centering
         \includegraphics[scale = 0.07]{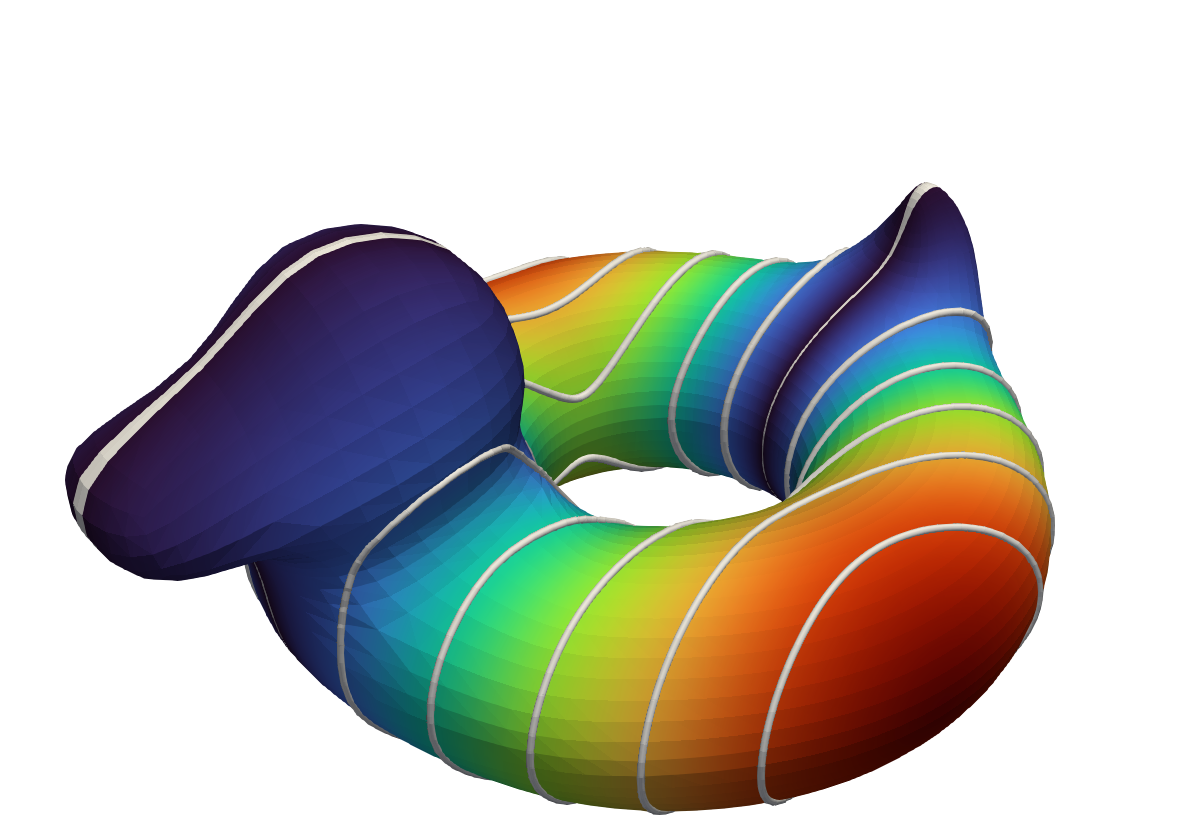}
           \subcaption{$p = 2.5$}
    \end{subfigure}
     \begin{subfigure}{0.18\textwidth}
     \centering
         \includegraphics[scale = 0.07]{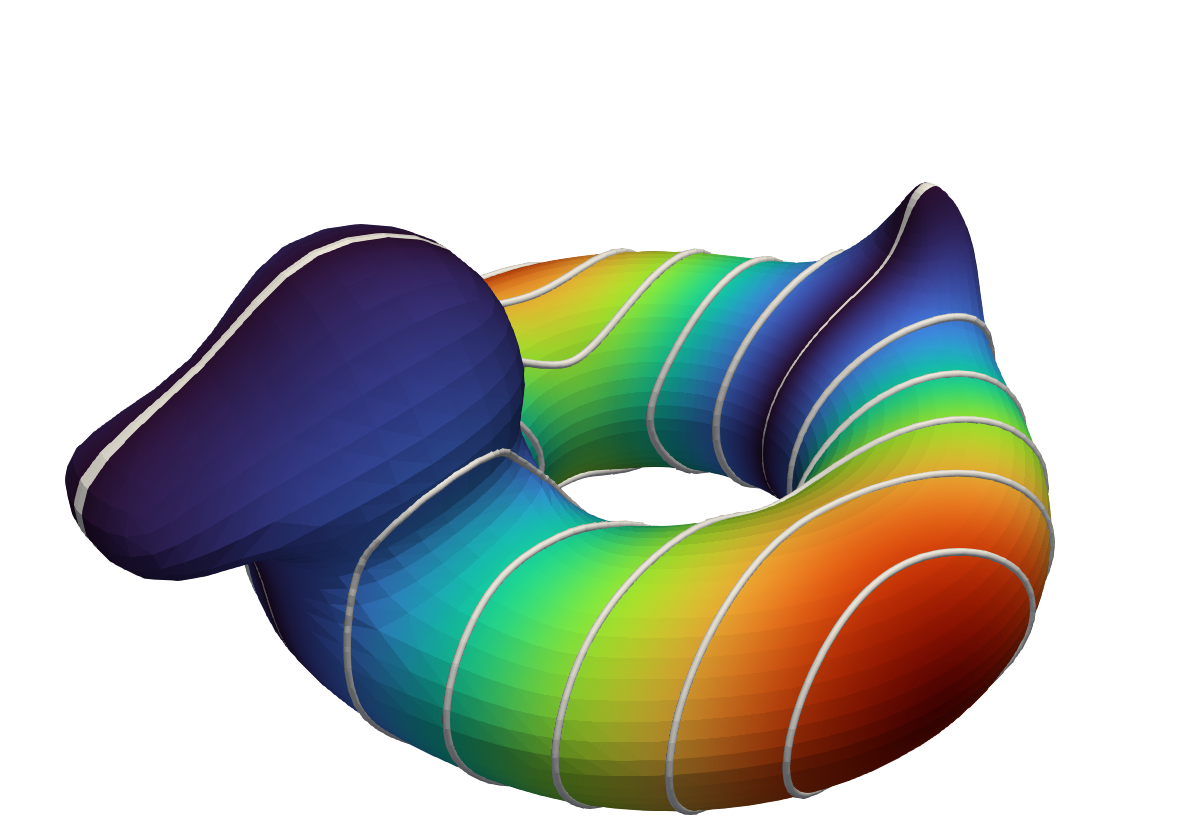}
          \subcaption{$p = 3$}
    \end{subfigure}
     \begin{subfigure}{0.18\textwidth}
        \centering
         \includegraphics[scale = 0.07]{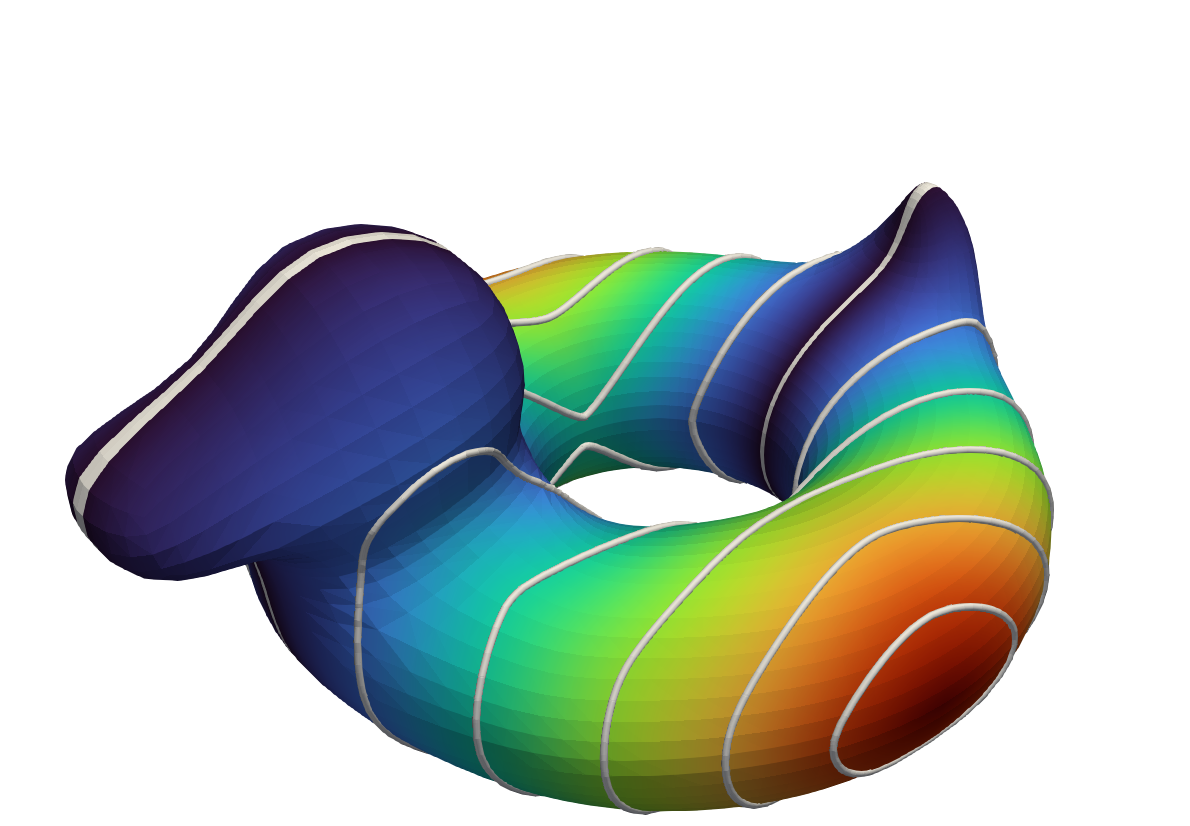}
           \subcaption{$p = 10$}
    \end{subfigure}
     \begin{subfigure}{0.18\textwidth}
     \centering
         \includegraphics[scale = 0.07]{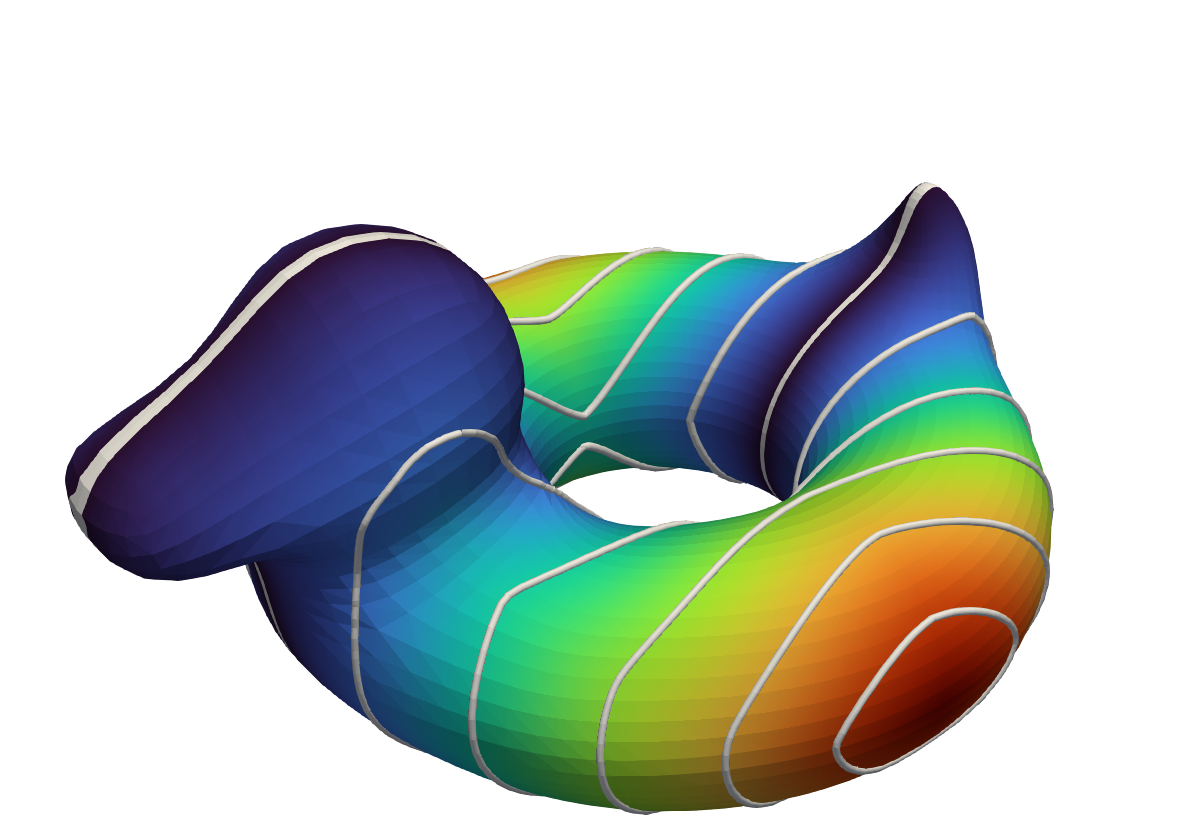}
          \subcaption{$p = 50$} 
    \end{subfigure}\\
        \begin{subfigure}{0.18\textwidth}
        \centering
         \includegraphics[scale = 0.1]{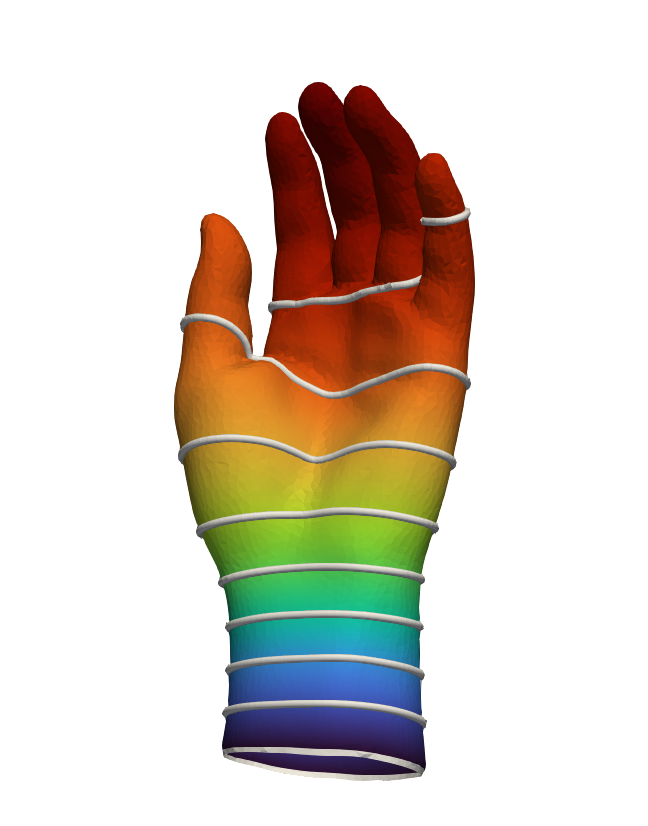}
           \subcaption{$p = 2$}
    \end{subfigure}
     \begin{subfigure}{0.18\textwidth}
        \centering
         \includegraphics[scale = 0.1]{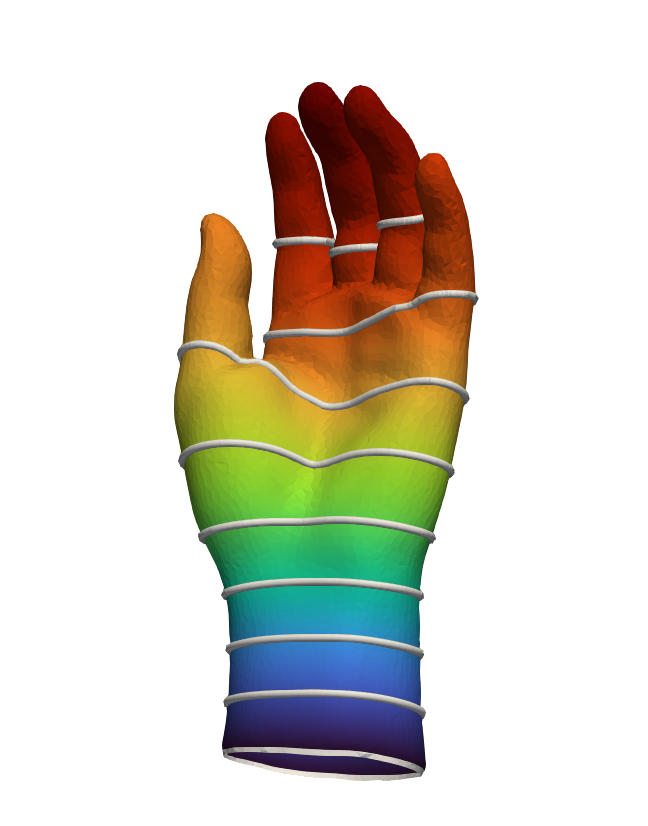}
           \subcaption{$p = 2.5$}
    \end{subfigure}
     \begin{subfigure}{0.18\textwidth}
     \centering
         \includegraphics[scale = 0.1]{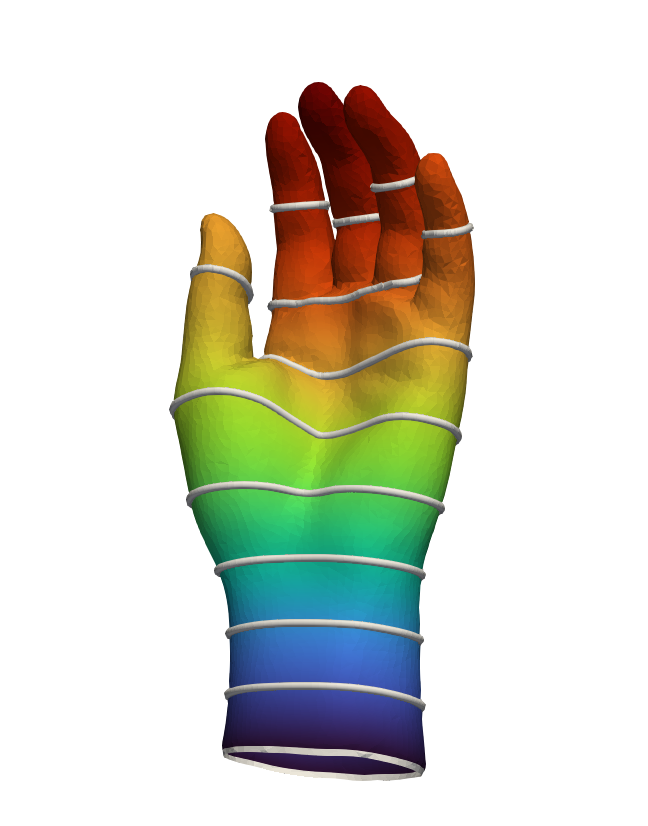}
          \subcaption{$p = 3$}
    \end{subfigure}
     \begin{subfigure}{0.18\textwidth}
        \centering
         \includegraphics[scale = 0.1]{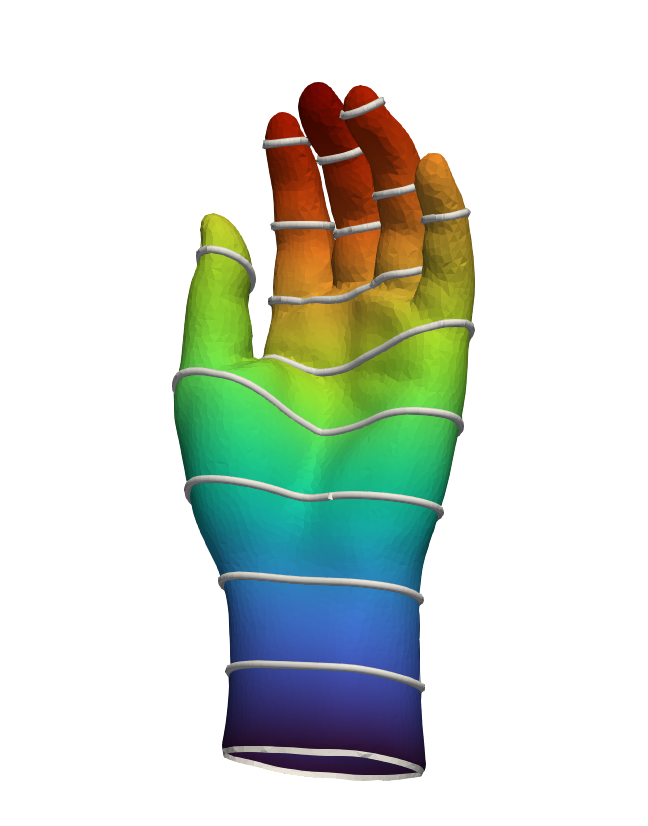}
           \subcaption{$p = 10$}
    \end{subfigure}
     \begin{subfigure}{0.18\textwidth}
     \centering
         \includegraphics[scale = 0.1]{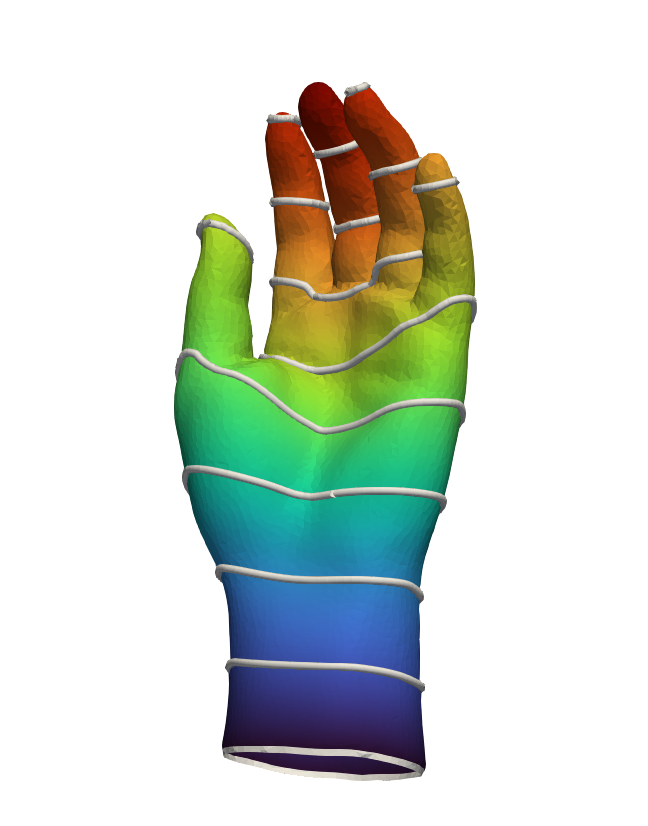}
          \subcaption{$p = 50$}
    \end{subfigure}\\ 
      \begin{subfigure}{0.18\textwidth} 
        \centering
         \includegraphics[scale = 0.07]{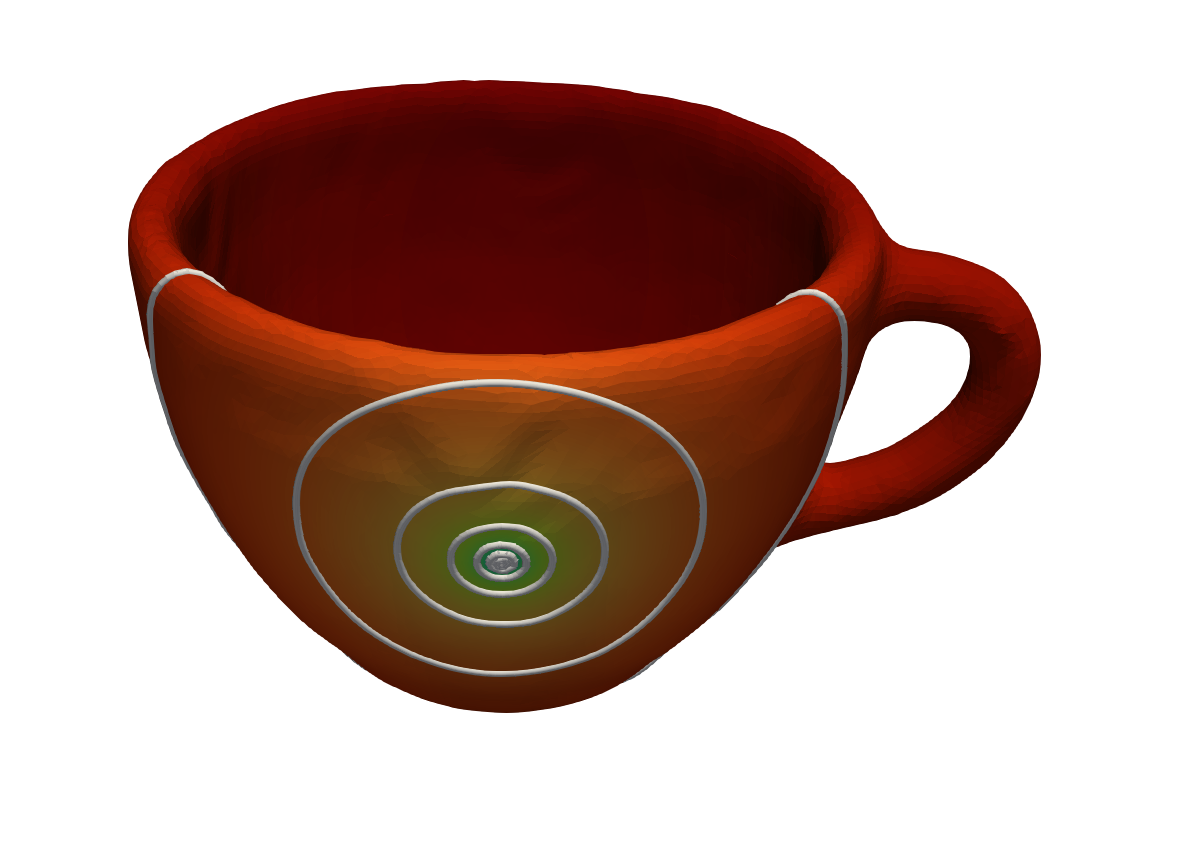}
           \subcaption{$p = 2$}
    \end{subfigure}
     \begin{subfigure}{0.18\textwidth}
        \centering
         \includegraphics[scale = 0.07]{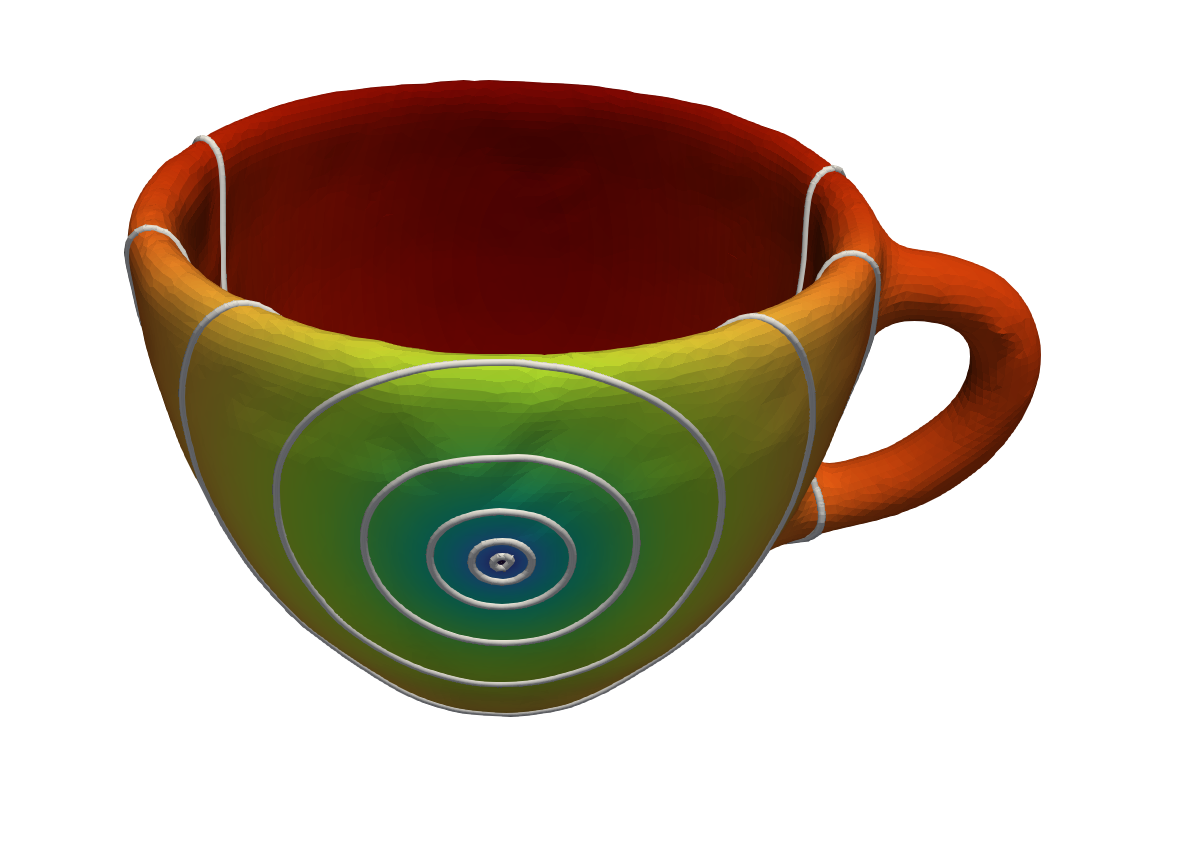}
           \subcaption{$p = 2.5$}
    \end{subfigure}
     \begin{subfigure}{0.18\textwidth}
     \centering
         \includegraphics[scale = 0.07]{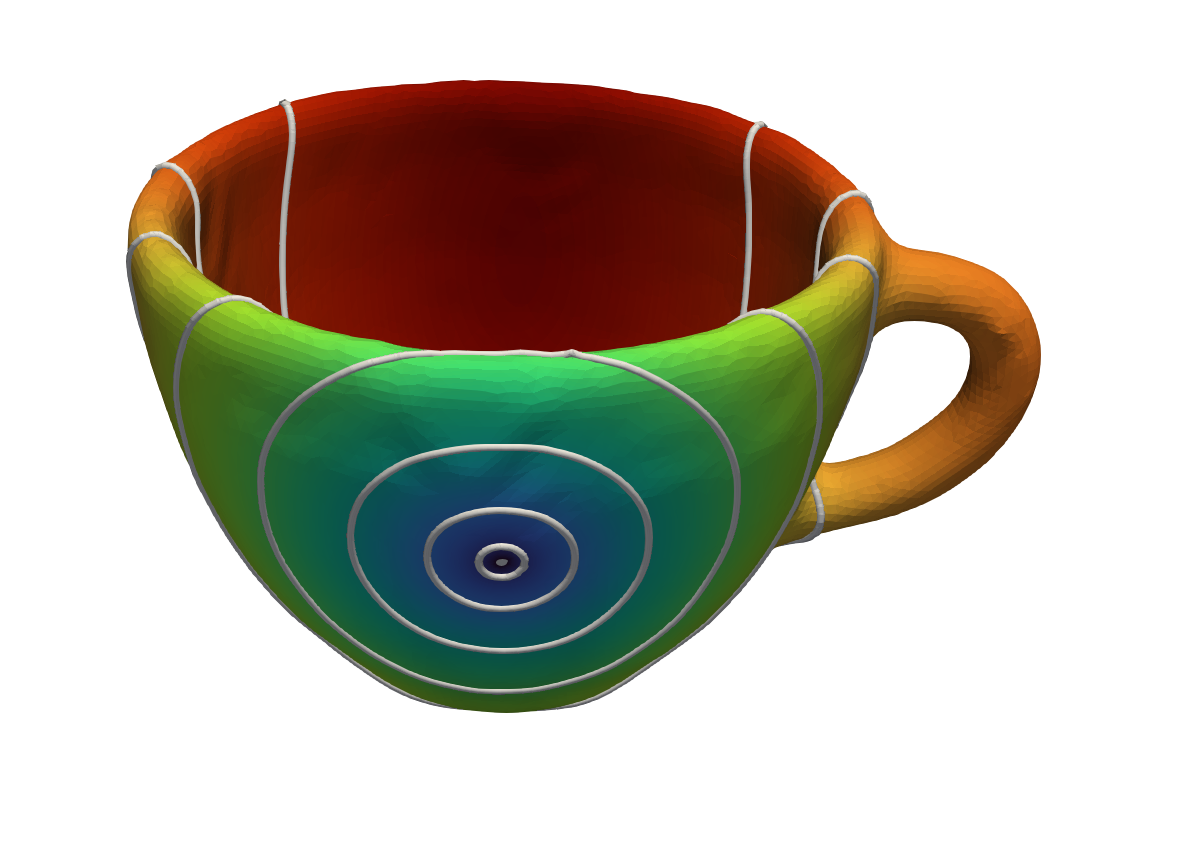}
          \subcaption{$p = 3$}
    \end{subfigure}
     \begin{subfigure}{0.18\textwidth}
        \centering
         \includegraphics[scale = 0.07]{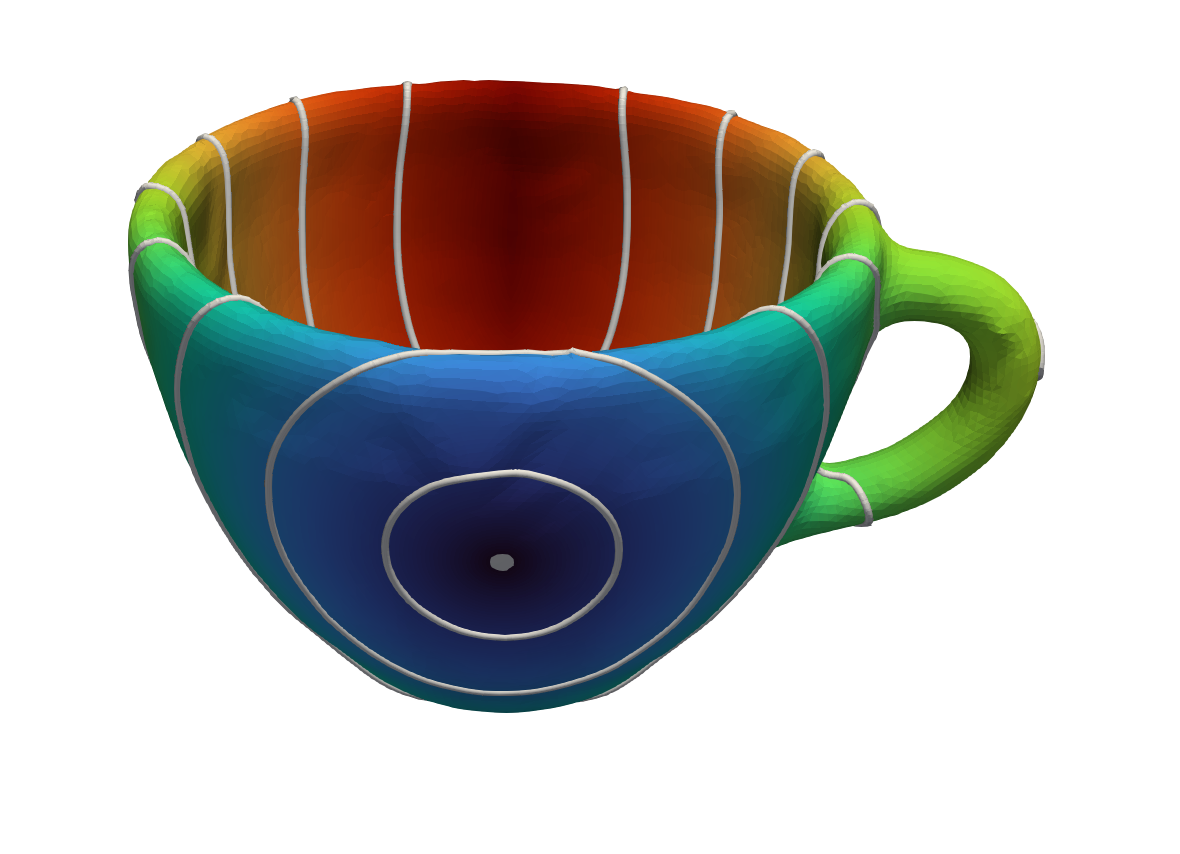}
           \subcaption{$p = 10$}
    \end{subfigure}
     \begin{subfigure}{0.18\textwidth}
     \centering
         \includegraphics[scale = 0.07]{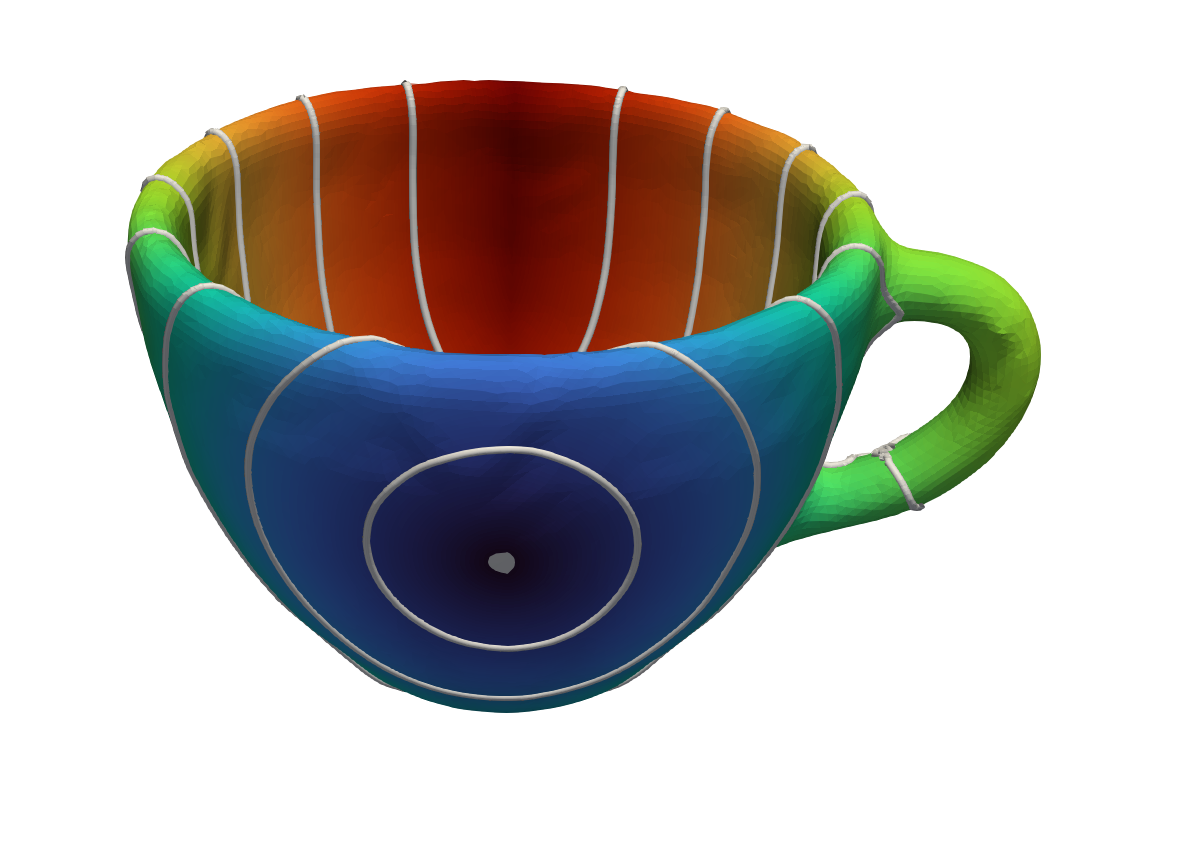}
          \subcaption{$p = 50$}
    \end{subfigure}\\
     \begin{subfigure}{0.18\textwidth}
        \centering
         \includegraphics[scale = 0.08]{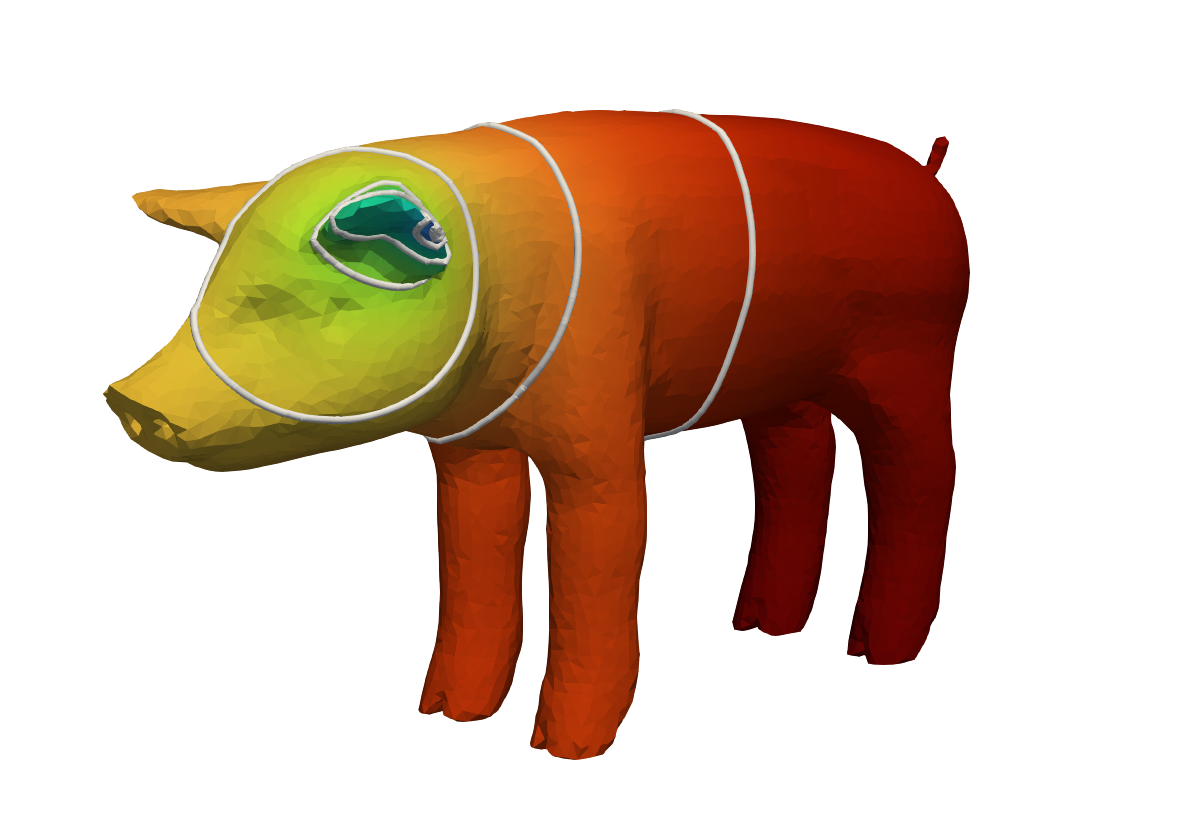}
           \subcaption{$p = 2$}
    \end{subfigure}
     \begin{subfigure}{0.18\textwidth}
        \centering
         \includegraphics[scale = 0.08]{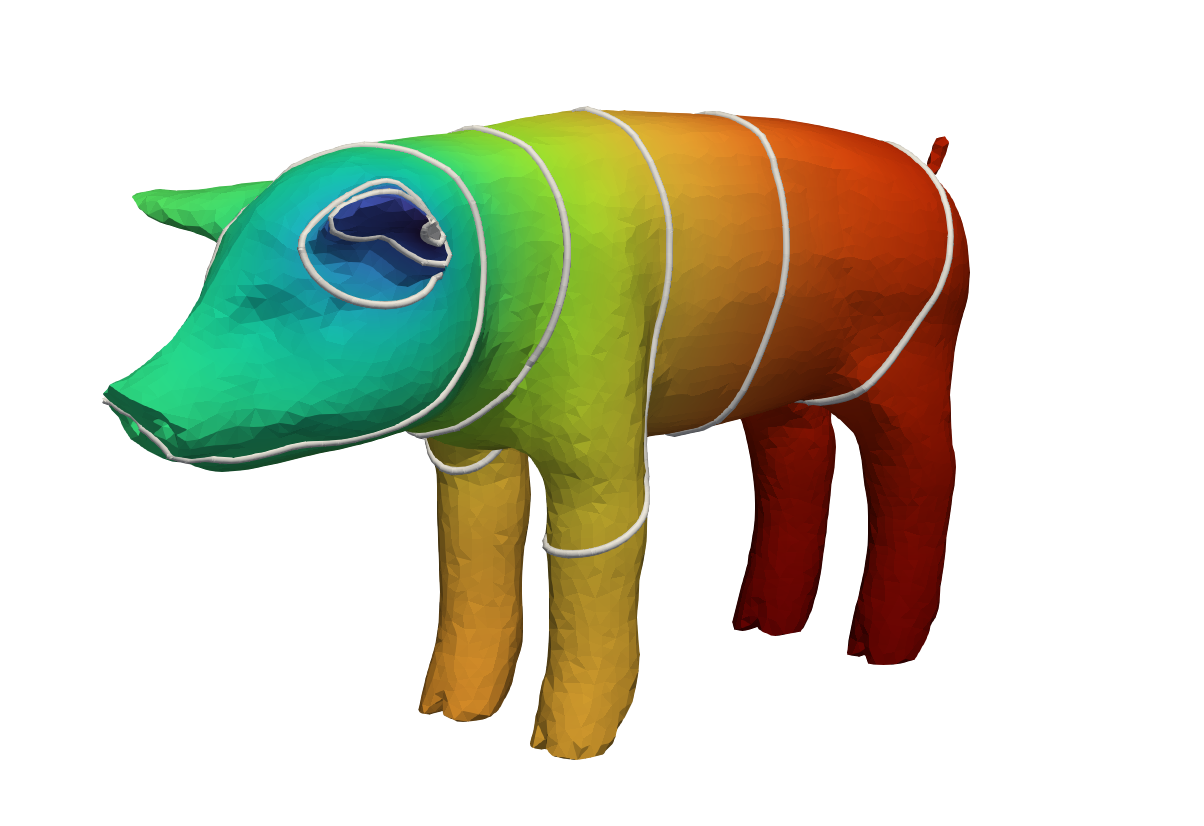}
           \subcaption{$p = 2.5$}
    \end{subfigure}
     \begin{subfigure}{0.18\textwidth}
     \centering
         \includegraphics[scale = 0.08]{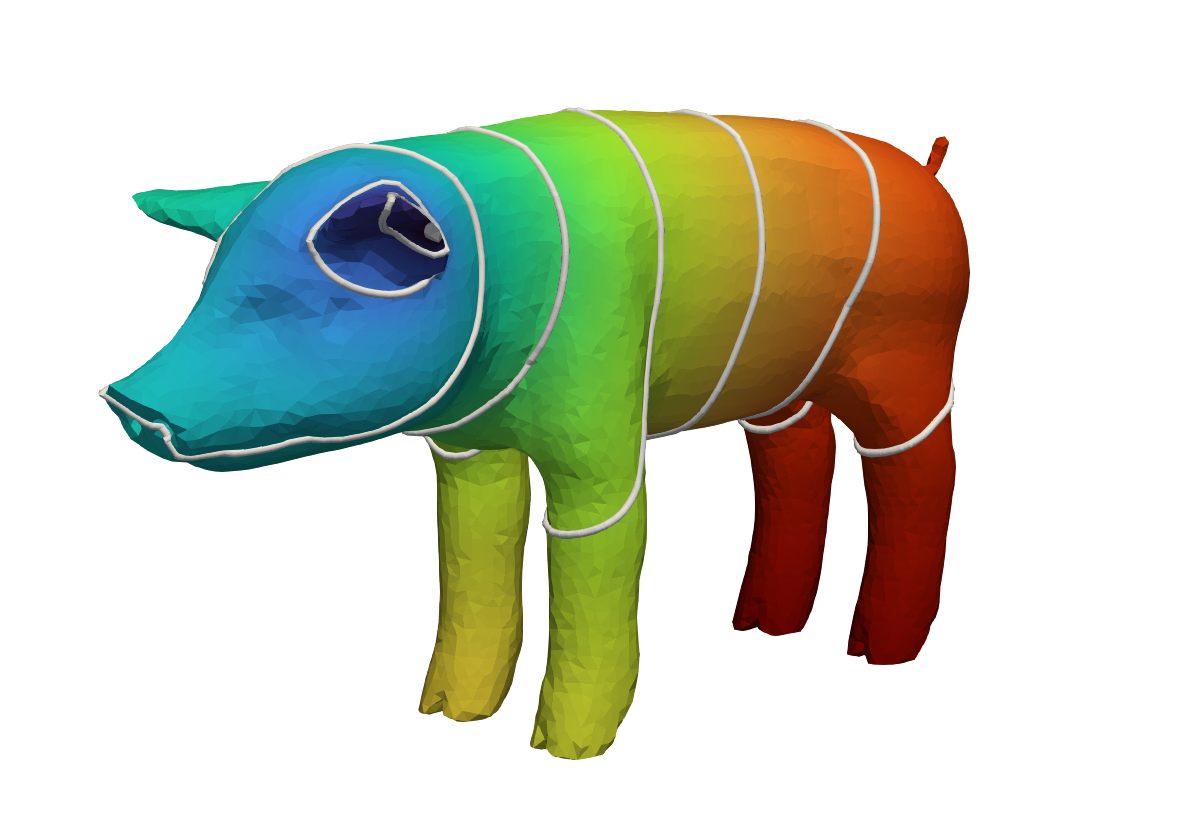}
          \subcaption{$p = 3$}
    \end{subfigure}
     \begin{subfigure}{0.18\textwidth}
        \centering
         \includegraphics[scale = 0.08]{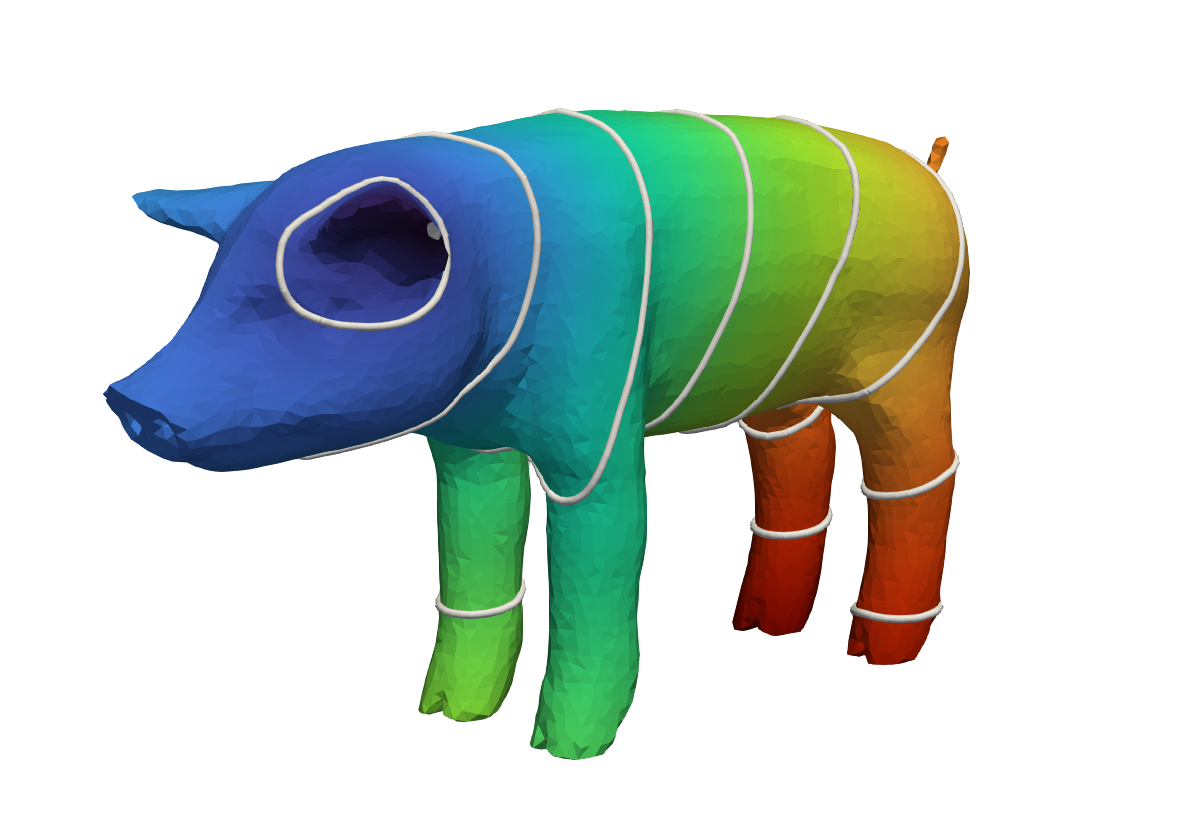}
           \subcaption{$p = 10$}
    \end{subfigure}
     \begin{subfigure}{0.18\textwidth}
     \centering
         \includegraphics[scale = 0.08]{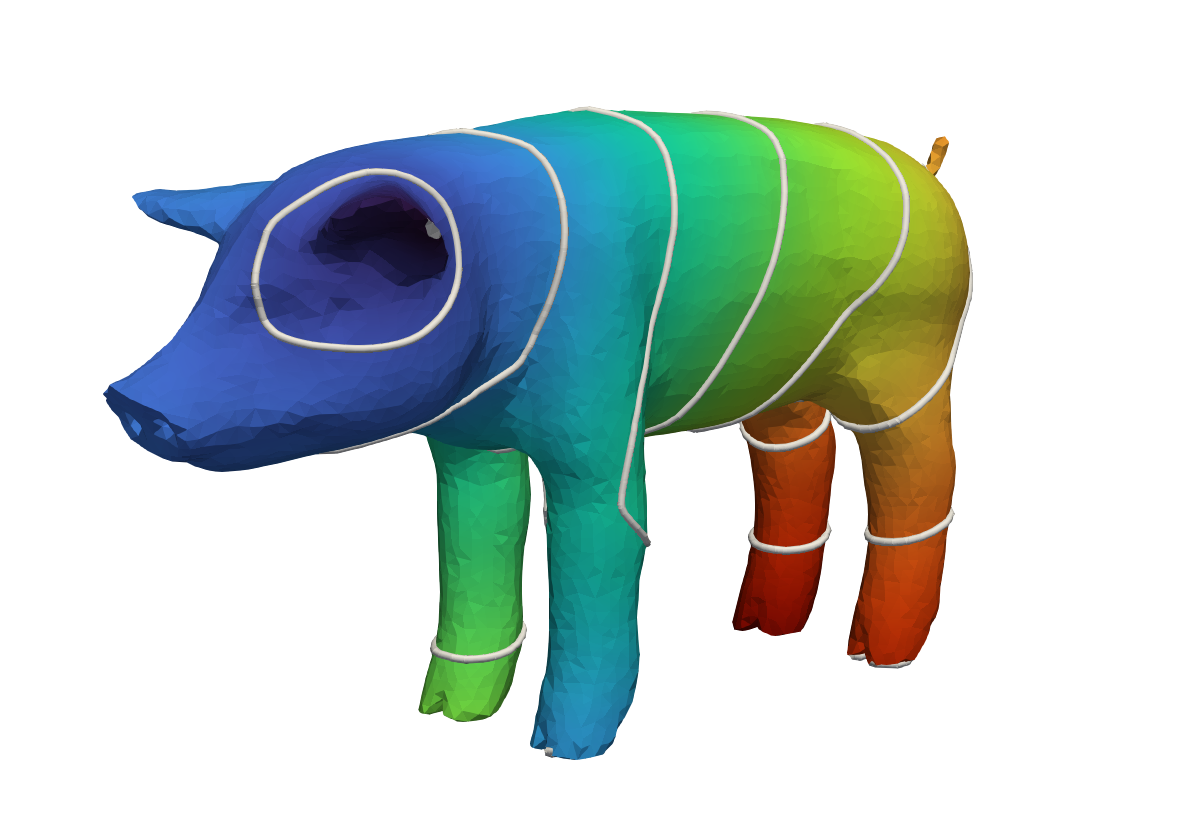}
          \subcaption{$p = 50$}
    \end{subfigure}\\
    \caption{Contour plots of $\upApprox$ on four surface meshes for various values of $p$. 
Row 1 (a)–(e): duck mesh with $32{,}064$ cells and $\partial\Omega$ taken along the central axis. $\Omega$ corresponds to one half of the surface, but both symmetric halves are shown for visualization.
Row 2 (f)–(j): hand mesh with $193{,}836$ cells and $\partial\Omega$ at the natural wrist boundary. 
Row 3 (k)–(o): mug mesh with $360{,}888$ cells and $\partial\Omega$ consisting of a single point (or small $\epsilon$-ball) on the outer surface. 
Row 4 (p)–(t): pig mesh with $50{,}454$ cells and $\partial\Omega$ consisting of a single point (or small $\epsilon$-ball) on the left ear. 
For visualization, the regions where the Dirichlet condition $\upApprox=0$ is imposed are shown in gray on each surface. 
    Within each row, contour levels are chosen consistently for visual clarity. 
    }
    \label{fig:12_fun}
\end{figure}

\begin{figure}
    \centering
    \includegraphics[scale = 0.5]{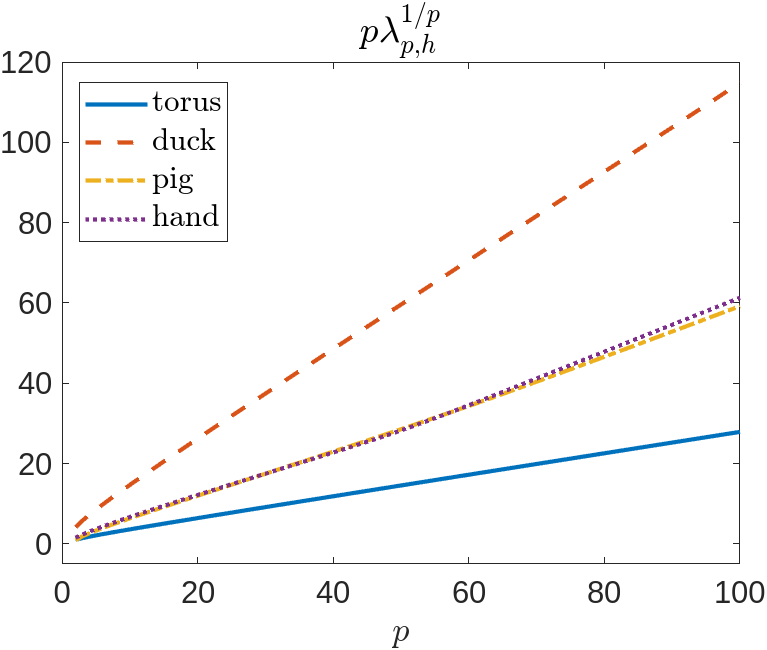}
    \caption{Verification that $p \lampApprox^{1/p}$ is strictly increasing in $p$ (see inequality~\eqref{eqn:mono}). We ensure that the slope is positive between consecutive discrete $p$ values. Results are plotted for some representative surface examples but note that adherence to the monotonicity inequality was observed in all of our considered examples. The meshes consist of the following numbers of quadrilateral cells: half torus ($784{,}896$), duck ($32{,}064$), pig ($50{,}454$), and hand ($193{,}836$).}
    \label{fig:13_ineq_check}
\end{figure}

\section{Conclusions}
\label{sec:conclusions}

In this work, we have studied the first $p$-Laplacian eigenpair on both Euclidean and surface (non-Euclidean) domains, with the extension to surface domains representing a novel contribution to the numerical literature. A key component of our approach is a new domain rescaling strategy that enables accurate computations for large $p$ values. We carried out numerical convergence studies, including comparisons with exact one-dimensional solutions, to assess the accuracy of our method. Additionally, by combining a Newton inverse-power iteration with $p$ continuation, we were able to efficiently compute eigenpairs across a wide range of $p$ (typically up to $p \sim 100$). For domains where the ridge set and maximal set of the distance-to-boundary function align exactly, our results numerically illustrate the connection between the eigenfunctions and the limiting (geodesic) distance-to-boundary function. Overall, we have demonstrated the effectiveness of our approach in handling complex geometries and provide a foundation for future work on higher eigenpairs, alternative boundary conditions, and potential applications of $p$-Laplacian eigenpairs such as shape optimization problems. 

\section*{Acknowledgments} 
The authors gratefully acknowledge the financial support of NSERC Canada (RGPIN 341834 for RF and RGPIN 2022-03302 for SR).

\bibliographystyle{plain}   
\bibliography{ref.bib} 

\appendix

\section{Regularity class of the $\sin_p$ function} 
\label{app:sinp-regularity}

We show that on a one-dimensional interval $\Omega = (a, b)$, $\up \in C^{1, 1/(p-1)}(\overline{\Omega})$ for finite $p>2$, by local asymptotic expansion for $\sin_p$ near the local maximum $\pi_p/2$, where a cusp forms. Note that for $p=2$, we have the ordinary sine function, and so $u_2 \in C^{\infty}(\overline{\Omega})$. 

The first eigenfunction is
\[
\up(x) = \sin_p \left(\pi_p \frac{x-a}{b-a} \right),
\]
and since $\sin_p$ attains a maximum at $\pi_p/2$, $\up$ attains its maximum at $x_0 = (a+b)/2$. 

Recall from Section~\ref{subsec:1d} that $\sin_p$ is defined through the inverse of the integral 
\[F_p(x) = \int_0^x (1- t^p)^{-1/p} dt, \quad F_p(1) = \frac{\pi_p}{2}.\] Therefore, to investigate local behavior near the maximum, we consider asymptotic behavior of $F_p$ near $1$.  

Let $\epsilon > 0$ be small. Then 
\begin{align*}
F_p(1) - F_p(1-\epsilon) &= \int_{1-\epsilon}^{1} (1- t^p)^{-1/p} dt \\
&= \int_0^{\epsilon} (1- (1-\tilde t)^p)^{-1/p} d\tilde t ,
\end{align*}
where we have applied the change of variable $\tilde t = 1-t$ so that for convenience we may expand about $\tilde t = 0$. We have 
\[
 1 - (1-\tilde t)^p 
= p \tilde t - \frac{p(p-1)}{2}\tilde t^2 + \mathcal{O}(\tilde t^3), 
\]
so, to leading order, our integrand becomes 
\[
( 1 - (1-\tilde t)^p)^{-1/p} = (p \tilde t)^{-1/p}\bigl(1+o(1)\bigr).
\]
Then
\[
F_p(1) - F_p(1-\epsilon)
\sim \int_0^{\epsilon} (p \tilde t)^{-1/p} d \tilde t 
= \frac{p^{(p-1)/p}}{p-1}\epsilon^{(p-1)/p}.
\]
Thus, for $\epsilon$ small we have
\[
F_p(1) - F_p(1-\epsilon)
= \frac{p^{(p-1)/p}}{p-1}\;\epsilon^{(p-1)/p}\bigl(1+o(1)\bigr),
\]

Since $\sin_p(x) = F_p^{-1}(x)$, for $t$ near $\pi_p/2$, we have
\[\frac{\pi_p}{2}-t = F_p(1) - F_p(\sin_p(t)) \sim \frac{p^{(p-1)/p}}{p-1} \left(1-\sin_p (t)\right)^{(p-1)/p}.\]
Taking the $p/(p-1)$th root we obtain
\[1-\sin_p(t) \sim p^{-1}{(p-1)^{p/(p-1)}} \left|t-\frac{\pi_p}{2}\right|^{p/(p-1)}.\]
Hence, for $x$ near $x_0$,
\begin{equation}
\label{eqn:sinp-asym}
     \up(x) = \sin_p \left(\pi_p \frac{x-a}{b-a} \right) \sim 1-K_p
\left|x-x_0\right|^{p/(p-1)},
\end{equation}
where $K_p := p^{-1}{(p-1)^{p/(p-1)}} \left( \frac{\pi_p}{b-a} \right)^{p/(p-1)}$.

The local approximation~\eqref{eqn:sinp-asym} shows the formation of a cusp at $x_0$ for $p>2$. Indeed, differentiating~\eqref{eqn:sinp-asym} yields
\begin{equation}
    \label{eqn:sinp'-asym}
    \up'(x) 
    \sim -\frac{p}{p-1} K_p \left|x-x_0\right|^{1/(p-1)} \mathrm{sign}(x-x_0),
\end{equation}
for $x$ near $x_0$. Therefore, $\up'$ is H\"older continuous at $x_0$, with exponent $1/(p-1)$. Also, $\up$ is smooth away from the maximum at $x_0$, as $F_p(x)$ is smooth for $x<1$. We conclude that $\up'\in C^{0,1/(p-1)}(\overline{\Omega})$, and hence $\up\in C^{1,1/(p-1)}(\overline{\Omega})$. Furthermore, for any $p>2$, $\up$ is not $C^2$, and the H\"older exponent of $\up'$ decreases to $0$ as $p \to \infty$.

Finally, note that for $p=2$ the local expansion near $x_0$ reduces to
\[
u_p(x)\sim 1-\tfrac12 \left(\frac{\pi}{b-a}\right)^2 (x-x_0)^2,
\]
corresponding to a smooth maximum, as consistent with the Taylor expansion of the regular sine function.  

\section{Geodesic distance function as the limiting $p \to \infty$ eigenfunction}
\label{app:surface-limits}

In this appendix we show that, when the set of maximal distance points coincides with the ridge of the distance function,
the limiting $p \to \infty$ eigenfunction equals the normalized distance-to-boundary function. 

The definition of a viscosity solution as in~\cite{Katzourakis2015} for the limiting $p \to \infty$ PDE~\eqref{eqn:limitpdea} is provided below and will be used in the proof of the proposition to follow.

\begin{definition}[Viscosity solution]\label{def:visc-sol}
Let $\Omega$ be an open subset of a surface $S$ and denote by $B_g(\bx_0,r)$ a geodesic ball centered at $\bx_0$ of radius $r$. A continuous function $u: \overline{\Omega}\to\mathbb{R}$ is said to be a 
\emph{viscosity subsolution} (respectively, \emph{supersolution}) of
\[
\min \left\{|\gradS u| - \LamInf u, -\infLapS u\right\}=0,
\qquad\text{in }\Omega,
\]
if for every $C^2$ smooth $\psi:  B_g(\bx_0,r) \subseteq \Omega\to \mathbb{R}$ such that $u - \psi$ has a local maximum (respectively, minimum) at $\bx_0$ with $\psi(\bx_0) = u(\bx_0)$, the following holds:
\begin{equation*}
\min \left\{|\gradS \psi(\bx_0)| - \LamInf \psi(\bx_0), -\infLapS \psi(\bx_0)\right\} \leq 0 \qquad \text{(respectively, $\geq 0$)}.
\end{equation*}
This is equivalent to 
\begin{align*}
\text{(subsolution)}\qquad 
& |\gradS \psi(\bx_0)|-\LamInf \psi(\bx_0) \leq 0 
\quad \;\:\text{or}\;\quad
-\infLapS \psi(\bx_0) \leq 0,\\
\text{(supersolution)}\qquad 
& |\gradS \psi(\bx_0)|-\LamInf \psi(\bx_0)  \geq 0 
\quad\text{and}\quad  
-\infLapS \psi(\bx_0) \geq 0.
\end{align*}

If $u$ is both a viscosity subsolution and supersolution,
it is called a \emph{viscosity solution}.
\end{definition}

\begin{proposition} Let $\Omega$ be an open and bounded subset of a surface $S$, and assume that the set $\mathcal{M}$ of maximal distance-to-boundary points coincides exactly with the ridge set $\mathcal{R}$ (see~\eqref{eqn:calR} and~\eqref{eqn:calM}).
Then the function $d(\bx) = \distx$ is a positive viscosity solution of
\begin{subequations}
\label{eqn:limitpde-appendix}
    \begin{alignat}{2}
    \min  \left \{|\gradS u|-\LamInf u , -\infLapS u\right \} = 0, & \qquad \text{in } \Omega,  \label{eqn:limitpde2}\\
    u = 0, & \qquad \text{on } \partial \Omega,  \label{eqn:bc2}
    \end{alignat} 
\end{subequations}
where $\LamInf = \|d\|_\infty^{-1} $.
\label{prop1}
\end{proposition}

\begin{proof}
Away from the ridge, a standard direct calculation may be carried out to show that the equation is satisfied in the classical sense. We present this calculation first; the calculation is a direct extension from the Euclidean case. We will then show that the equation is satisfied in the viscosity sense at the ridge. 
\smallskip

{\em Case $\bx_0 \in \Omega \setminus \mathcal{R}$.} 
Away from the ridge set, the distance-to-boundary function $d(\cdot)$ is smooth.  At any point $\bx_0 \in \Omega \setminus \mathcal{R}$, let $\by_0 \in \partial \Omega$ denote the unique closest boundary point, and let $\bgamma(s)$ be the unit-speed geodesic on $S$ connecting $\bx_0$ to $\by_0$.
 By taking the second derivative of $d(\cdot)$ along the geodesic $\bgamma(s)$ we find
 \begin{equation}
 \label{eqn:d2ds2}
 \frac{d^2}{ds^2} d(\bgamma(s)) \Big|_{s=0} = \dot{\bgamma}(0)^\top \HessS d(\bx_0) \, \dot{\bgamma}(0). 
 \end{equation}
 On the other hand, from $d(\bgamma(s)) = d(\bx_0) - s$ we infer that the second derivative of $d(\cdot)$ along the geodesic vanishes:
 \begin{equation*}
 \frac{d^2}{ds^2} d(\bgamma(s)) \Big|_{s=0} =0,
 \end{equation*}
 which combined with~\eqref{eqn:d2ds2} gives
 \begin{equation}
 \label{eqn:gdotHsgdot}
 \dot{\bgamma}(0)^\top \HessS d(\bx_0) \, \dot{\bgamma}(0) =0.
 \end{equation}

 By definition of the distance function, $d(\bgamma(s))$ decreases at the maximal rate in the direction of the gradient and hence, $\gradS d(\bgamma(s))$ is tangent to $\bgamma(s)$. In particular, we have
 \[
 \gradS d(\bx_0) = -\dot{\bgamma}(0),
 \]
 and by combining with~\eqref{eqn:gdotHsgdot} we find 
 \[
 \gradS d(\bx_0)^\top \HessS d(\bx_0) \, \gradS d(\bx_0) = 0.
 \]
 Then, by the definition of the surface infinity Laplacian, we get
 \begin{equation}
 \label{eqn:Deltainf0}
 -\infLapS d(\bx_0) = 0.
 \end{equation}

 Finally, note that $|\gradS d(\bx_0)|=1$, and by definition of $\mathcal{M}$ and $\LamInf$, we have
 \[
 d(\bx_0) < \| d \|_\infty = \Lambda_\infty^{-1}.
 \]
 Hence, 
  \[
  |\gradS d(\bx_0)|-\LamInf d(\bx_0) > 0,
 \]
 which together with~\eqref{eqn:Deltainf0} shows~\eqref{eqn:limitpde2}. 

\smallskip

{\em Case $\bx_0 \in \mathcal{R}$.}
The geodesic distance-to-boundary function $d(\cdot)$ satisfies the Eikonal equation, $|\gradS d| = 1$, in the viscosity sense at the ridge~\cite{MantegazzaMennucci2003}.  Therefore for any $\bx_0 \in \mathcal{R} = \mathcal{M}$, as $d(\bx_0)=\LamInf^{-1}$, we have
\begin{equation}
\label{eqn:limitpde-pt1} 
   |\gradS d(\bx_0)| - \LamInf d(\bx_0) = 0,
\end{equation}
in the viscosity sense, meaning $d(\cdot)$ is both a viscosity supersolution and viscosity subsolution of~\eqref{eqn:limitpde-pt1}~\cite{Katzourakis2015}. 

\smallskip
{\em (i) Viscosity subsolution for $\bx_0 \in \mathcal{R}$.}
We will show that $d(\cdot)$ is a viscosity subsolution of~\eqref{eqn:limitpde2} according to Definition~\ref{def:visc-sol}. We already know that $d(\cdot)$ is a viscosity subsolution of~\eqref{eqn:limitpde-pt1}. Thus, $ |\gradS d(\bx_0)|-\LamInf d(\bx_0) \leq 0$ in the viscosity sense, which is enough as it implies $\min  \left \{ |\gradS d|-\LamInf d , -\infLapS d\right \} \leq 0$ in the viscosity sense at the ridge. 

\smallskip
{\em (ii) Viscosity supersolution for $\bx_0 \in \mathcal{R}$.}
Now we will show that $d(\cdot)$ is a viscosity supersolution of~\eqref{eqn:limitpde2} according to Definition~\ref{def:visc-sol}. 
Suppose $r > 0$, $\bx_0 \in \Omega$, and $\psi:  B_g(\bx_0,r) \subseteq \Omega\to \mathbb{R}$ is a smooth function such that $d - \psi$ has a local minimum at $\bx_0$ with $\psi(\bx_0) = d(\bx_0)$. As $d(\cdot)$ is a viscosity supersolution of the Eikonal equation, it holds that 
\begin{equation}
|\gradS \psi(\bx_0)| - \LamInf \psi(\bx_0) \geq 0.
\label{eqn:Eikonalsuper}
\end{equation}
It remains to show that $-\infLapS \psi(\bx_0) \geq 0$.

Since the supersolution condition~\eqref{eqn:Eikonalsuper} along with the touching condition $\psi(\bx_0) = d(\bx_0) = \LamInf^{-1}$ ensures 
\begin{equation}
  |\gradS\psi(\bx_0)| \geq \LamInf\psi(x_0) = 1,
  \label{eqn:gradgeq1}
\end{equation}
we may let $\bv = \gradS\psi(\bx_0)/|\gradS\psi(\bx_0)|$.
Now consider the unit-speed geodesic that starts from $\bx_0$ in the direction $\bv$, that is, $\bgamma(t) = \exp_{\bx_0}(t\bv)$. Here, $\exp$ denotes the Riemannian exponential map. For $ 0 < h < r$, let $\by_h \in \partial\Omega$ be a point such that $\mathrm{dist}(\bgamma(h), \by_h) = d(\bgamma(h))$, and similarly, let $\by_0 \in \partial\Omega$ satisfy $\mathrm{dist}(\bx_0 , \by_0) = d(\bx_0)$.

Then,
\begin{equation}
\label{eqn:est-dgammah}
\begin{aligned}
d(\bgamma(h)) &\leq \mathrm{dist}(\bgamma(h), \by_0) \\[2pt]
& \leq \mathrm{dist}(\bgamma(h), \bx_0) + \mathrm{dist}(\bx_0, \by_0) \\[2pt]
& = h + d(\bx_0),
\end{aligned}
\end{equation}
where for the second line we used the triangle inequality, and for the third we used that the geodesic $\bgamma$ has unit speed, and hence, $\mathrm{dist}(\bgamma(h), \bx_0) = h$.

Since $d - \psi$ has a local minimum at $\bx_0$, 
\[
 \psi(\bgamma(h)) - \psi(\bx_0) \leq d(\bgamma(h)) - d(\bx_0),
 \]
 which combined with~\eqref{eqn:est-dgammah} leads to
 \begin{equation}
 \label{eqn:psi-ineq}
 \psi(\bgamma(h)) - \psi(\bx_0) \leq h.
 \end{equation}
Taylor expand $\psi$ along $\bgamma$ (use $\bgamma'(0) = \bv$) to get 
\[
\psi(\bgamma(h)) = \psi(\bx_0) + h\langle \gradS \psi(\bx_0), \bv \rangle + \frac{h^2}{2}\langle \HessS \psi(\bx_0)\bv, \bv \rangle + o(h^2).
 \]
Combining the Taylor expansion with~\eqref{eqn:psi-ineq} we find
 \begin{equation}
      h |\gradS\psi(\bx_0)| + \frac{h^2}{2}\langle \HessS \psi(\bx_0)\bv, \bv\rangle + o(h^2) \leq h.
      \label{eqn:Taylorcombine}
 \end{equation}
Dividing by $h$ and taking the $h \to 0^+$ limit gives 
 \[
  |\gradS\psi(\bx_0)| \leq 1,
 \]
 and from~\eqref{eqn:gradgeq1} we know \(
  |\gradS\psi(\bx_0)| \geq  1\).
 Therefore, it holds that $ |\gradS\psi(\bx_0)| = 1$ and so~\eqref{eqn:Taylorcombine} becomes  
 \[
       h + \frac{h^2}{2}\langle \HessS \psi(\bx_0)\bv, \bv\rangle + o(h^2) \leq h.
 \]
Now subtract $h$ on both sides of the inequality above, and divide by $h^2$ to get
\[
\frac{1}{2}\bv^{\top} \HessS \psi(\bx_0) \bv + \frac{o(h^2)}{h^2} \leq 0.
\]
Taking $h \to 0^+$ yields
\[
\bv^{\top} \HessS \psi(\bx_0) \bv \leq 0.
\]
Since $\bv= \gradS\psi(\bx_0)/|\gradS\psi(\bx_0)| = \gradS\psi(\bx_0)$, this is equivalent to
\[
-\infLapS \psi(\bx_0) = -\gradS\psi(\bx_0)^{\top} \HessS \psi(\bx_0) \gradS \psi(\bx_0) \geq 0.
\]

\smallskip
The boundary condition~\eqref{eqn:bc2} is trivially satisfied. Thus, $d$ indeed satisfies~\eqref{eqn:limitpde-appendix} in the viscosity sense at the ridge. 

\end{proof}

\begin{remark}
\label{rmk:uinf-dist}
Proposition~\ref{prop1} implies that $\uInf(\cdot) = d(\cdot)/\|d\|_\infty$ satisfies the $p \to \infty$ limiting eigenvalue problem given by~\eqref{eqn:limitpde}. 

\end{remark}

\end{document}